\documentclass[11pt]{article}
\usepackage{amsmath,amsthm,amsfonts,amssymb}
\usepackage[all]{xy}
\usepackage[usenames]{color} 

\newcounter{enumitemp}
\newenvironment{enumeratecontinue}{
 \setcounter{enumitemp}{\value{enumi}}
 \begin{enumerate}
 \setcounter{enumi}{\value{enumitemp}}
}
{
 \end{enumerate}
}

\newcommand\pref[1]{(\ref{#1})}

\newtheorem{thm}{Theorem}[section]

\newtheorem{theorem}[thm]{Theorem}

\newtheorem{lemma}[thm]{Lemma}

\newtheorem*{lemma*}{Lemma}
\newtheorem{cor}[thm]{Corollary}
\newtheorem{corollary}[thm]{Corollary}
\newtheorem{proposition}[thm]{Proposition}
\newtheorem*{proposition*}{Proposition}

\newtheorem{prop}[thm]{Proposition}

\newtheorem{fact}[thm]{Fact}

\theoremstyle{definition}

\newtheorem{definition}[thm]{Definition} 

\newtheorem*{defn*}{Definition}

\newtheorem{notn}[thm]{Notation}

\newtheorem{ex}[thm]{Example}

\newtheorem{remark}[thm]{Remark}
\theoremstyle{remark}

\newcounter{remarks}
{\paragraph*{Remarks}\smallskip
 \begin{list}{\arabic{remarks}. }{\usecounter{remarks}%
 \setlength{\leftmargin}{0in}%
 \setlength{\rightmargin}{0in}%
 \setlength{\labelsep}{0pt}%
 \setlength{\labelwidth}{0pt}%
 \setlength{\listparindent}{0pt}%
 }
}
{
\end{list}
}

\newcommand\from:
\newcommand\inv{{-1}}
\newcommand\subgroup{<}
\newcommand\infinity\infty
\newcommand\na{\text{na}}
\newcommand\supp{\text{supp}}
\newcommand\disjunion\coprod
\newcommand\act\curvearrowright

\DeclareMathOperator{\Fix}{Fix}
\DeclareMathOperator{\Per}{Per}

\DeclareMathOperator{\PF}{PF}

\DeclareMathOperator{\cl}{cl}

\DeclareMathOperator\image{Image}

\DeclareMathOperator\core{core}

\DeclareMathOperator\Ker{Ker}
\DeclareMathOperator\Fr{Fr}

\newcommand{\R}{{\mathbb R}}
\newcommand\reals{\R}

\newcommand{\Z}{{\mathbb Z}}
\newcommand{\C}{{\mathcal C}}
\newcommand{\T}{{\mathbb T}}
\def\H{{\mathbb H}}
\newcommand{\Q}{{\mathbb Q}}

\newcommand{\D}{{\mathcal D}}
\newcommand{\E}{{\mathcal E}}
\newcommand{\V}{{\mathcal V}}
\newcommand{\W}{{\mathcal W}
\newcommand{\M}{\mathcal M}}
\newcommand{\K}{{\mathcal K}}

\DeclareMathOperator{\Out}{\mathsf{Out}}
\DeclareMathOperator{\Aut}{\mathsf{Aut}}

\DeclareMathOperator{\mcg}{\mathsf{MCG}}
\DeclareMathOperator\MCG{\mcg}
\DeclareMathOperator{\Stab}{\mathsf{Stab}}

\newcommand{\ffs}{free factor system}

\newcommand{\pg}{PG}

\newcommand{\upg}{UPG}
\newcommand{\F}{\mathcal F}

\renewcommand\L{\mathcal L}

\def\B{\mathcal B}
\newcommand{\A}{\mathcal A}

\renewcommand\T{\mathcal T}
\newcommand\cQ{\mathcal Q}

\newcommand{\fG} {f : G \to G}
\newcommand{\ti} {\tilde}
\newcommand{\iNp} {indivisible Nielsen path}

\newcommand{\filt}{\emptyset = G_0 \subset G_1 \subset \ldots \subset G_N = G}

\newcommand{\eg}{EG}
\newcommand{\noneg}{NEG}
\renewcommand\neg\noneg

\newcommand{\wt}{\widetilde}

\newcommand\agen{{\text{ag}}}

\newcommand{\ct}{CT}
\newcommand{\cts}{CTs}

\newcommand{\comment}[1]{}

\newcommand\BH{\cite{BestvinaHandel:tt}}
\newcommand\BookZero{\cite{BFH:laminations}}
\newcommand\BookOne{\cite{BFH:TitsOne}}
\newcommand\BookTwo{\cite{BFH:TitsTwo}}
\newcommand\BookThree{\cite{BFH:Solvable}}

\newcommand\recognition{\cite{FeighnHandel:recognition}}

\newcommand\SubgroupsZero{\cite{HandelMosher:SubgroupsIntro}}
\newcommand\SubgroupsOne{\cite{HandelMosher:SubgroupsI}}

\newcommand\SubgroupsThree{\cite{HandelMosher:SubgroupsIII}}
\newcommand\SubgroupsFour{\cite{HandelMosher:SubgroupsIV}}
\newcommand\FSHyp{\cite{HandelMosher:FreeSplittingHyperbolic}}

\DeclareMathOperator\interior{int}
\newcommand\bdy\partial

\newcommand\intersect\cap
\newcommand\union\cup
\newcommand\<\langle
\renewcommand\>\rangle
\newcommand\meet\wedge
\newcommand\composed{\circ}
\newcommand\cross\times
\newcommand\restrict{\bigm |}

\newcommand\inject\hookrightarrow

\newcommand\abs[1]{\left|#1\right|}

\DeclareMathOperator\rank{rank}

 \newcommand\surjection\twoheadrightarrow
 
\newcommand\suchthat{\bigm|}

\DeclareMathOperator\hs{{\cal L}^h}
\DeclareMathOperator\IA{IA}

\newcommand\IAThree{\IA_n(\Z/3)}
\newcommand\fscnp{{\cal FS'}(F_n)}
\newcommand\cH{{\cal H}}

\DeclareMathOperator\FS{{\cal FS}}

\newcommand\fscn{{\cal FS}(F_n)}
\DeclareMathOperator\FF{{\cal FF}}

\newcommand\low{-}
\DeclareMathOperator\closure{cl}
 
 \title{The free splitting complex of a free group II: \\ Loxodromic outer automorphisms}

\author{Michael Handel and Lee Mosher}

\begin{document}

\maketitle
\section{Introduction}

Consider a group $G$ acting by isometries on a Gromov hyperbolic metric space~$X$. An element $g \in G$ is \emph{loxodromic} if for some (any) $x \in X$ the orbit map $n \mapsto g^n \cdot x$ is a quasi-isometric embedding $\Z \mapsto X$. The terminology comes from the case of hyperbolic 3-space where such an isometry leaves invariant a ``loxodromic curve'' on the 2-sphere at infinity. The action of a loxodromic $g$ on the the Gromov closure $\overline X = X \union \bdy X$ is a north--south action with attracting--repelling fixed point pair $\bdy_\pm g = \lim_{n \to \pm\infinity} g^n \cdot x$. Two loxodromic elements $g,g' \in G$ are said to be \emph{coaxial} if the unordered fixed point pairs $\{\bdy_\pm g\}$, $\{\bdy_\pm g'\}$ are equal, and \emph{independent} if those pairs are disjoint. Understanding loxodromic behavior is important, for example, in proving the Tits alternative by ``hyperbolic ping-pong'' arguments, and for studying second bounded cohomology (see e.g.\ \cite{BestvinaFujiwara:bounded}).

Before stating our results, here are some examples. A Gromov hyperbolic group acts on its Cayley graph, each element is either finite order or loxodromic, any two loxodromic elements are either co-axial or independent, and if $g \in G$ is loxodromic then one has equality of stabilizer subgroups $\Stab(\bdy_- g) = \Stab(\bdy_+ g) \equiv \Stab(\bdy_\pm g)$ and this subgroup is virtually cyclic \cite{Gromov:hyperbolic}.

The mapping class group $\mcg(S)$ of a finite type surface $S$ acts on its curve complex~$\C(S)$, hyperbolicity of which was proved by Masur and Minsky \cite{MasurMinsky:complex1}. A mapping class $\phi \in \mcg(S)$ acts loxodromically on $\C(S)$ if and only if $\phi$ is pseudo-Anosov, which occurs if and only if $\phi$ has infinite order and does not preserve any simplex of $\C(S)$. Two loxodromics are either co-axial or independent, and for a single loxodromic $\phi$ the subgroup $\Stab(\bdy_\pm \phi)$ is virtually cyclic. These properties are proved first on the level of the stable and unstable lamination pair $\Lambda^s_\phi, \Lambda^u_\phi$, and are then transferred to $\C(S)$ by showing that for pseudo-Anosov $\phi,\psi \in \mcg(S)$ one has $\bdy_+\phi = \bdy_+ \psi$ if and only if $\Lambda^u_\phi = \Lambda^u_{\psi}$. 

$\Out(F_n)$ acts on the free factor complex $\FF(F_n)$, hyperbolicity of which was proved by Bestvina and Feighn \cite{BestvinaFeighn:FFCHyp}, and Theorem 9.3 of that paper proves that $\phi \in \Out(F_n)$ acts loxodromically on $\FF(F_n)$ if and only if $\phi$ is fully irreducible (and see Remark~\ref{RemarkFSLox}), which occurs if and only $\phi$ has no periodic simplices in $\FF(F_n)$. Also, two fully irreducibles are either co-axial or independent, and $\Stab(\bdy_\pm \phi)$ is virtually cyclic when $\phi$ is fully irreducible. Again these properties are related to attracting/repelling lamination pairs: the corresponding properties for laminations were proved in \BookZero; and from \cite{BestvinaFeighn:FFCHyp} it follows that two fully irreducibles are co-axial on $\FF(F_n)$ if and only if they have the same lamination pair.

\smallskip\noindent\textbf{Overview of results.} $\Out(F_n)$ acts naturally from the right on the free splitting complex $\FS(F_n)$, hyperbolicity of which was proved in Part~I of this work \cite{HandelMosher:FreeSplittingHyperbolic}. 
Here we study the loxodromic elements for this action, characterizing loxodromic behavior in terms of attracting/repelling laminations and similarly characterizing elements with bounded orbits and with a periodic point (Theorem~\ref{trichotomy}). We also prove the same ``co-axial versus independent'' dichotomy as in all of the above examples (Theorem~\ref{disjoint axes}). But there are some interesting features of this study which depart from the above examples. One feature (Theorem~\ref{trichotomy}) is that there are many more loxodromics acting on $\FS(F_n)$ than on $\FF(F_n)$. In mapping class groups, for $\phi \in \mcg(S)$ to be pseudo-Anosov there are two equivalent formulations: $\phi$ has a stable/unstable lamination pair that fills the surface; and $\phi$ has irreducible powers. This equivalence breaks down in $\Out(F_n)$, yielding two different meanings for ``loxodromic'': $\phi$ is loxodromic on $\FS(F_n)$ if and only if $\phi$ has a filling lamination pair; whereas $\phi$ is loxodromic on $\FF(F_n)$ if and only if it is fully irreducible, a strictly stronger condition. Another feature (Theorem~\ref{ThmFinitelyGenerated}) is that when $\phi \in \Out(F_n)$ acts loxodromically on $\FS(F_n)$, the subgroup $\Stab(\bdy_\pm\phi)$ need not be virtually cyclic: it can contain a higher rank abelian subgroup of linearly growing outer automorphisms; it can also map homomorphically onto to a surface mapping class group; in general $\Stab(\bdy_\pm\phi)$ is a mixture of these two behaviors.

\smallskip\noindent\textbf{Statements of results.} See Section~\ref{SectionBackground} for a brief review of attracting/repelling lamination pairs, of free factor systems and free factor supports, and of reducible, irreducible, and fully irreducible outer automorphisms. We let $\L(\phi)$ denote the set of attracting laminations of~$\phi$. The laminations of $\L(\phi)$ and $\L(\phi^\inv)$ come in pairs $\Lambda^+_\phi, \Lambda^-_\phi$ defined by requiring that they have the same free factor support, and if the support is not a proper free factor system then we say that the pair $\Lambda^\pm_\phi$ \emph{fills} $F_n$. The notation $\Lambda^\pm_\phi = (\Lambda^+_\phi,\Lambda^-_\phi)$ denotes the ordered pair, and $\{\Lambda^\pm_\phi\} = \{\Lambda^+_\phi,\Lambda^-_\phi\}$ denotes the unordered pair. As mentioned above, if $\phi$ is fully irreducible then it has a filling lamination pair $\Lambda^\pm_\phi$ (the unique element of $\L^\pm(\phi)$), but the converse is not true in general.

The following ``trichotomy theorem'' characterizes which elements are loxodromic, which have bounded orbits and which have a periodic point.

\begin{theorem} \label{trichotomy} The following holds for all $\phi \in \Out(F_n)$.
\begin{enumerate}
\item \label{item:loxodromic} The action of $\phi$ on $\fscn$ is loxodromic if and only if some element of $\L(\phi)$ fills. 
\item \label{item:bounded} If the action of $\phi$ on $\fscn$ is not loxodromic then the action has bounded orbits. 
\item \label{item:periodic vertex}The action of $\phi$ on $\fscn$ has a periodic point (in fact a periodic vertex) if and only if the full set of attracting laminations $\L(\phi)$ does not fill.
\end{enumerate}
\end{theorem}
\noindent
See Example~\ref{filling reducible} for a reducible $\phi$ that acts loxodromically on $\FS(F_n)$, and see Example~\ref{ExampleBddNoPeriodic} for a $\phi$ with bounded orbits but without periodic points.

\bigskip

Our next theorem describes a natural, equivariant, bijective correspondence between laminations and loxodromic fixed points. The group $\Out(F_n)$ has a natural left action on the set of attracting laminations $\Lambda$; we denote this action by $\theta \cdot \Lambda$ for each $\theta \in \Out(F_n)$. It also has a natural right simplicial action on $\FS(F_n)$ that extends homeomorphically to the points $\beta$ of the Gromov boundary of $\FS(F_n)$; we denote this action by $\beta^\theta$. 
%

\bigskip



\begin{theorem}\label{disjoint axes} 
Given $\phi,\psi \in \Out(F_n)$ and filling lamination pairs $\Lambda^\pm_\phi \in \L^\pm(\phi)$ and $\Lambda^\pm_\psi \in \L^\pm(\psi)$, one of the following holds:
\begin{enumerate}
\item\label{ItemCoAxial}
$\{\Lambda^\pm_\phi\} = \{\Lambda^\pm_\psi \}$ and $\{\bdy_\pm \phi\} = \{\bdy_\pm \psi\}$, and so $\phi,\psi$ are co-axial.
\item\label{ItemIndependent}
$\{\Lambda^\pm_\phi\} \intersect \{ \Lambda^\pm_\psi\} = \emptyset$ and $\{\bdy_\pm \phi\} \intersect \{\bdy_\pm \phi\} = \emptyset$, and so $\phi,\psi$ are independent.
\end{enumerate}
There is an equivariant 1--1 correspondence $\beta \leftrightarrow \Lambda$ between attracting/repelling loxodromic fixed points $\beta \in \bdy\FS(F_n)$ and repelling/attracting [sic] filling laminations $\Lambda$, as follows:
\begin{enumeratecontinue}
\item\label{ItemLamRepCorrespond}
$\beta \leftrightarrow \Lambda$ if and only if there exists a loxodromic $\phi \in \Out(F_n)$ with filling lamination pair $\Lambda^\pm_\phi \in \L^\pm(\phi)$ such that $\beta = \bdy_-\phi$ and $\Lambda = \Lambda^+_\phi$.
\item\label{ItemLamRepEquivariant}
For each $\theta \in \Out(F_n)$ and each corresponding pair $\beta \leftrightarrow \Lambda$ we have $\beta^\theta \leftrightarrow \theta^\inv \!\cdot \!\Lambda$.
\end{enumeratecontinue}
\end{theorem}

\smallskip

As an application we have the following result solely about attracting laminations, not referring to any complexes on which $\Out(F_n)$ acts. 

\begin{cor} \label{inverse lamination} \quad
\begin{enumerate}
\item\label{ItemTwoLamPairs}
For any $\phi,\psi \in \Out(F_n)$ and lamination pairs $\Lambda^\pm_\phi \in \L^\pm(\phi)$ and $\Lambda^\pm_\psi \in \L^\pm(\psi)$, if $\Lambda^+_\phi = \Lambda^+_\psi$ then $\Lambda^-_\phi = \Lambda^-_\psi$. 
\item\label{ItemOneLamPair}
For any $\phi \in \Out(F_n)$ and any lamination pair $\Lambda^\pm_\phi \in \L^\pm(\phi)$ we have \break $\Stab(\Lambda^-_\phi) = \Stab(\Lambda^+_\phi) = \Stab(\Lambda^\pm_\phi)$.
\end{enumerate}
\end{cor}

\begin{proof} To prove~\pref{ItemTwoLamPairs}, assuming $\Lambda^+_\phi = \Lambda^+_\psi$, the pairs $\Lambda^\pm_\phi$, $\Lambda^\pm_\psi$ have the same free factor support $\F = \{[F]\}$ where $F$ is a nontrivial free factor. When $\F$ is not proper, i.e.\ when $F=F_n$ and these pairs fill, the conclusion follows from Theorem~\ref{disjoint axes}~\pref{ItemCoAxial} and~\pref{ItemIndependent}. When $\F$ is proper then the corollary reduces to the filling case by passing to powers of $\phi,\psi$ that fix $\F$ and then replacing $\phi,\psi$ with their restrictions \, $\phi \restrict F$, \, $\psi \restrict F \, \in \, \Out(F)$ \, (\SubgroupsOne, Section~1.1.3).

We now turn to \pref{ItemOneLamPair}. By symmetry we need only prove $\Stab(\Lambda^-_\phi) \subgroup \Stab(\Lambda^+_\phi)$, so consider $\theta \in \Stab(\Lambda^-_\phi)$ and consider the lamination pair $\Lambda^\pm_{\theta\phi\theta^\inv} = \theta(\Lambda^\pm_\phi) \in \L^\pm(\theta\phi\theta^\inv)$. We have $\Lambda^-_{\theta\phi\theta^\inv} = \theta(\Lambda^-_\phi) = \Lambda^-_\phi$, and together with~\pref{ItemTwoLamPairs} it follows that $\theta(\Lambda^+_\phi) = \Lambda^+_{\theta\phi\theta^\inv} = \Lambda^+_\phi$, proving that $\theta \in \Stab(\Lambda^+_\phi)$.
\end{proof}

When $\phi \in \Out(F_n)$ acts loxodromically on $\fscn$ with filling pair $\Lambda^\pm_\phi$, our next theorem describes the stabilizer group $\Stab(\bdy_\pm \phi) = \Stab(\Lambda^\pm_\phi)$ and in particular proves that it is finitely generated. Recall the finite index, torsion free, normal subgroup 
$$\IAThree = \text{kernel}\bigl(\Out(F_n) \to \Aut(H_1(F_n;\Z/3)\bigr)
$$ 
The homomorphism $\PF_{\Lambda^+}$ in the statement of the theorem is the restriction to $\IAThree$ of the expansion factor homomorphism for $\Lambda_\eta^+$ that is defined in Section 3.3 of \BookOne; see Section~\ref{SectionBackground} for a review.


\begin{theorem} \label{ThmFinitelyGenerated} Suppose that $\eta \in \IAThree$ is rotationless, that $\Lambda_\eta^+ \in \L(\phi)$ is filling, and that $K$ is the kernel of $\PF= \PF_{\Lambda^+} : \Stab(\Lambda_\eta^+) \cap \IAThree \to \R$. There exist compact surfaces $S_1,\ldots, S_m$ with nonempty boundary and a homomorphism 
$$\Theta \from K \to \mcg(S_1) \times \ldots \times \mcg(S_m)
$$ 
whose image has finite index, and whose kernel is a finitely generated, abelian group of linearly growing outer automorphisms. In particular, $K$ is finitely generated.
\end{theorem}

Compare \BookZero\ Section~2 where it is shown that if $\phi$ is fully irreducible then the kernel of the expansion factor homomorphism is finite, in which case $K$ is trivial and $\Stab(\Lambda^\pm_\phi)$ is virtually cyclic. See Example~\ref{linear example} in which $\image(\Theta)$ is trivial and $K$ is a rank~$2$ abelian subgroup of linearly growing outer automorphisms. And see Example~\ref{surface example} in which $\Theta$ maps $K$ onto a finite index subgroup of a mapping class group. 

\smallskip\noindent\textbf{Failure of the acylindrical and WPD properties.} By combining Theorem~\ref{ThmFinitelyGenerated} with Examples~\ref{linear example} and~\ref{surface example}, we note that the action of $\Out(F_n)$ on $\FS(F_n)$ is not acylindrical in the sense of Bowditch \cite{Bowditch:tight}, nor does it satisfy the weaker condition that each $\phi \in \Out(F_n)$ acting loxodromically on $\FS(F_n)$ satisfies the WPD property of Bestvina and Fujiwara \cite{BestvinaFujiwara:bounded}, because those conditions imply that $\Stab(\bdy_\pm \phi)$ is virtually cyclic. Note Bestvina and Feighn \cite{BestvinaFeighn:FFCHyp} show that fully irreducible elements acting on $\FF(F_n)$ do satisfy WPD; it remains unknown whether the action of $\Out(F_n)$ on $\FF(F_n)$ is acylindrical.

\smallskip\noindent\textbf{Application to second bounded cohomology.} 
In \cite[Theorem E]{HandelMosher:BddCohomologyI}, the results of the current paper are applied to prove that certain loxodromic elements of $\Out(F_n)$ satisfy the WWPD property of Bestvina, Bromberg, and Fujiwara \cite{BBF:MCGquasitrees}. This is a key step of the proof given in \cite{HandelMosher:BddCohomologyI,HandelMosher:BddCohomologyII} that each finitely generated subgroup of $\Out(F_n)$ is either virtually abelian or has second bounded cohomology of uncountable dimension.

\smallskip\noindent\textbf{The case of rank~$2$.} The results of this paper are trivial in rank~$1$, and in rank~$2$ follow from well known results. The abelianization map $F_2 \mapsto \Z^2$ induces an isomorphism $\Out(F_2) \approx GL_2(\Z)$ (see \cite{Vogtmann:OuterSpaceSurvey} for a reference to Nielsen). The complex $\FS(F_2)$ equivariantly contains the Farey graph $\Gamma$ having vertex set $\Q$ where vertices $\frac{p}{q}, \frac{r}{s}$ are connected by an edge whenever $ps-qr=\pm 1$ (see \cite{CullerVogtmann:RankTwo}). The graph $\Gamma$ is Gromov hyperbolic \cite{Manning:pseudocharacters}. The Gromov boundary of $\Gamma$ has an equivariant bijection with the irrational numbers $\reals-\Q$, the elements of $GL_2(\Z)$ acting loxodromically on $\Gamma$ are exactly the matrices having trace of absolute value $>2$, and these correspond exactly with the exponentially growing elements of $\Out(F_2)$ each of which is fully irreducible and has a filling lamination.

\smallskip\noindent\textbf{Remarks on the proofs.} The table of contents gives a guide to the proofs of Theorems~\ref{trichotomy} and~\ref{ThmFinitelyGenerated}. The proof of Theorem~\ref{disjoint axes} draws on methods from the rest of the paper: one part is proved in Section~\ref{SectionDistinctEnds} using methods from the proof of Theorem~\ref{trichotomy}; and the remainder is proved in Section~\ref{SectionProofCompletion} by applying Theorem~\ref{ThmFinitelyGenerated} and its methods of proof, and by applying the main result from our work \SubgroupsZero\ on decomposition of subgroups of $\Out(F_n)$.

\vfill\break

\setcounter{tocdepth}{2}
\tableofcontents


\section{Background}
\label{SectionBackground}

In this section we set notation and provide references to \BookOne, \recognition\ and \SubgroupsOne\ for readers that want further details.

\paragraph{Marked graphs, paths.} We assume that $F_n$ has been identified with $\pi_1(R_n, *)$ where $R_n$ is the graph with one vertex $*$ and $n$ oriented edges representing a free basis of~$F_n$. A~\emph{marked graph} $G$ is a graph such that each vertex has valence at least two and such that $G$ is equipped with a homotopy equivalence $\rho :R_n \to G$ called the \emph{marking} on $G$. The marking provides an identification of $\pi_1(G)$ with $F_n$ that is well defined up to composition with an inner automorphism. Thus conjugacy classes of elements and of subgroups of $\pi_1(G)$ correspond bijectively to conjugacy classes of elements and of subgroups of $F_n$. A \emph{path} in $G$ is an immersion of a (possibly trivial, infinite or bi-infinite) closed subinterval of $\mathbb{R}$ having endpoints, if any, at vertices. We do not distinguish between paths that differ only by an orientation preserving reparameterization of their domains, and so a path is determined by its associated edge path and we identify a path with its edge path. A general continuous map  $\alpha : [0,1] \to G$ with endpoints at intervals straightens to a unique path denoted $[\alpha]$. A path $\sigma$ is \emph{crossed} by a path $\tau$ if either $\sigma$ or $\bar \sigma$ is a subpath of $\tau$. A \emph{circuit} in $G$ is an immersion of a circle, and various ``path'' terminologies apply as well to circuits. As with paths, circuits which differ by orientation preserving reparameterization of their domain are not distinguished.

A \emph{core subgraph} of $G$ is a subgraph in which each vertex has valence~$\ge 2$. Every subgraph $H \subset G$ contains a unique maximal core subgraph called the \emph{core of $H$}, which equals the union of all circuits in $H$.

\paragraph{Homotopy equivalences and $f_\#$.} Every homotopy equivalence of marked graphs $f \from G \to G'$ is henceforth assumed to map vertices to vertices, and to restrict on each edge of~$G$ to either an immersion or a constant. 

In the case of a self-homotopy equivalence $f \from G \to G$, after choosing base vertex $v \in G$ and a path from $f(v)$ back to $v$, the induced action of $\fG$ on the fundamental group $\pi_1(G,v)$ induces in turn a well defined outer automorphism of $\pi_1(G,v)$ and hence a well defined $\phi \in \Out(F_n)$; we say that $\fG$ represents $\phi$ or is a topological representative of $\phi$. If $\sigma \subset G$ is a finite path (resp.\ circuit) then $f(\sigma)$ is homotopic rel endpoints to a unique path (resp.\ circuit) that we denote $f_\#(\sigma)$. Note that $f_\#$ can be iterated and that $(f_\#)^k = (f^k)_\#$.

\paragraph{Subgroup systems. Carrying and meet ($\sqsubset$ and $\meet$).} The conjugacy class of a finite rank subgroup $A \subgroup F_n$ is denoted $[A]$. If $A_1,\ldots,A_k$ are pairwise nonconjugate, nontrivial, finite rank subgroups then the set $\A = \{[A_1],\ldots,[A_k]\}$ is called a \emph{subgroup system}. Each $[A_i]$ is a \emph{component} of $\A$. If $A_1,\ldots,A_k$ are non-trivial free factors and if $F_n = A_1 \ast \ldots \ast A_k$ or $F_n = A_1 \ast \ldots \ast A_k\ast B$ for some non-trivial free factor $B$ then $\A$ is a \emph{free factor system}. More generally if there exists a minimal $\reals$-tree action $F_n \act T$ with trivial arc stabilizers such that $\A$ is the set of conjugacy classes of nontrivial point stabilizers then $\A$ is a \emph{vertex group system}. Given another subgroup system $\A' = \{[A'_1],\ldots,[A'_l]\}$ we use the notation $\A \sqsubset \A'$ to mean that for each $i \in \{1,\ldots,k\}$ there exists $j \in \{1,\ldots,l\}$ such that $A_i$ is conjugate to a subgroup of~$A'_j$; we refer to this relation by saying that $\A$ is \emph{carried by} or \emph{contained in~$\A'$}, also that $\A'$ is an \emph{extension of~$\A$}, or simply that \emph{$\A \sqsubset \A'$ is an extension}. A \emph{filtration by free factor systems} is a sequence of extensions of free factor systems $\F_0 \sqsubset \cdots \sqsubset \F_K$. The \emph{meet} $\F_1 \wedge \F_2$ of two free factor systems $\F_1$ and $\F_2$ is the unique maximal free factor system that is contained in both $\F_1$ and $\F_2$. By a version of Grushko's theorem, $\F_1 \wedge \F_2$ is the set of nontrivial conjugacy classes of subgroups of the form $A_1 \intersect A_2$ such that $[A_1] \in \F_1$, $[A_2] \in \F_2$. The meet operation on pairs extends to a well-defined operation on any set of free factor systems. To each marked graph $G$ and subgraph $K$ there corresponds a free factor system denoted $\F(K)$ or $[K]$, by taking the conjugacy classes of the subgroups of $\pi_1(G) \approx F_n$ corresponding to the noncontractible components of~$K$. See Section 2.6 of \BookOne, and Sections 1.1.2 and 3.1 of \SubgroupsOne.
 
\paragraph{Lines and free factor support.} The \emph{space of lines} $\B=\B(F_n)$ is the quotient of $\bdy F^n \times \bdy F^n - \Delta$ by transposing coordinates and by letting $F_n$ act. Each line is realized in each marked graph as a bi-infinite path which is unique modulo orientation. A line is \emph{birecurrent} if for some (any) such realization, each finite subpath is repeated infinitely often in both directions. A line $\ell$ is \emph{carried by} a free factor system $\F = \{[A_1],\ldots,[A_k]\}$ if for some (any) marked graph $G$ with subgraph $H$ corresponding to $\F$, the realization of $\ell$ in $G$ is contained in~$H$. A conjugacy class is carried by $\F$ if it is represented by an element in one of the $A_i$'s, equivalently if the periodic line representing that conjugacy class is carried by $\F$. The \emph{free factor support} of a collection of lines or conjugacy classes is the meet of all free factor systems that carry each element of that collection, equivalently the unique minimal free factor system carrying the entire collection. A set of lines \emph{fills} $F_n$ if its free factor support is $\{[F_n]\}$. For each subgraph $K$ of a marked graph $G$ the free factor system $\F(K)$ is the free factor support of the set of conjugacy classes represented by circuits in $K$. 
See Section 2.6 of \BookOne\ or Section 2.5 of \recognition\ or Sections 1.1.2 and 1.2.2 of \SubgroupsOne.
 
\paragraph{Filtration, strata, height, relative train track maps.} A \emph{filtration} of a marked graph $G$ is a nested sequence of subgraphs $\emptyset = G_0 \subset G_1 \subset G_2 \subset \ldots G_N = G$. A path or circuit has \emph{height} $i$ if it is contained in $G_i$ but not $G_{i-1}$. For any filtration by free factor systems $\emptyset = \F_0 \sqsubset \F_1 \sqsubset \cdots \sqsubset \F_N = \{[F_n]\}$ there exists a filtered marked graph denoted as above such that $[G_i]=\F_i$. Given a filtration of $G$ as above and a topological representative $f \from G \to G$ of some outer automorphism, if $f(G_i) \subset G_i$ for all $i$ then the filtration is \emph{$f$-invariant} or just \emph{invariant} if $f$ is clear from the context. The union of edges contained in $G_i$ but not $G_{i-1}$ is a subgraph denoted $G_i \setminus G_{i-1} = H_i$ called the \emph{$i^{th}$ stratum}. If $f(H_i) \subset G_{i-1}$ then $H_i$ is a \emph{zero stratum}. When working with a topological representative $f \from G \to G$ of some $\phi \in \Out(F_n)$ and an $f$-invariant filtration of~$G$, we always assume that each stratum $H_i$ is either a zero stratum or an \emph{irreducible stratum}, the latter meaning that for any two edges $E,E' \subset H_i$ there exists $k \ge 1$ such that $f^k_\#(E)$ crosses $E'$. Also, we always assume that each irreducible stratum is either \neg\ which stands for \emph{nonexponentially growing} or \eg\ which stands for \emph{exponentially growing}: what these mean is that for some (every) edge $E \subset H_i$ the number of $H_i$ edges crossed by the path $f^k_\#(E)$ is at most $1$ for all $k$ in the \neg\ case, and has exponential upper and lower bounds in the \eg\ case. We sometimes say that an edge is \eg\ if it belongs to an \eg\ stratum and similarly for \neg\ edges. 

There is a useful hierarchy of better and more useful classes of topological representatives, although their existence sometimes requires first passing to a power. At the most basic, every outer automorphism is represented by a \emph{relative train track map} which is a topological representative with certain tightness properties imposed on its \eg\ strata. After passing to a positive power, subdividing certain edges, and straightening $f$ the following properties hold: any \neg\ stratum $H_i$ consists of a single oriented edge $E$ that satisfies $f(E) = E\cdot u$ for some (possibly trivial) path in $G_{i-1}$; and any \eg\ stratum $H_r$ has the property that for all sufficiently large $k$ the $f^k_\#$-image of each edge in $H_r$ crosses every edge in $H_r$ and the number of such crossings grows exponentially in~$k$. Furthermore, any \eg\ stratum has $\ge 2$ edges. See Section 5 of \cite{BestvinaHandel:tt}, Subsection 1.5.1 of \SubgroupsOne\ or Subsections 2.6 and 2.7 of \recognition.
 
\paragraph{Nielsen paths, splittings, highest edge splittings.} If $\sigma$ is a finite, nondegenerate path and if $f_\#^k(\sigma) = \sigma$ for some $k \ge 1$ then we say that $\sigma$ is a \emph{periodic Nielsen path}; if $k = 1$ then $\sigma$ is a \emph{Nielsen path}. A Nielsen path is \emph{indivisible} if it is not the concatenation of two non-trivial Nielsen subpaths. 
 
A decomposition of a path or circuit into subpaths is a \emph{splitting}, written $\sigma = \sigma_1 \cdot \ldots \cdot \sigma_m$, if $f^k_\#(\sigma)$ decomposes into subpaths $f^k_\#(\sigma_1) \cdot \ldots \cdot f^k_\#(\sigma_m)$ for all $k \ge 1$. See Section 4 of \BookOne, Sections 2.2 and 4.2 of \recognition, and Definition 1.27 of \SubgroupsOne.

Considered a filtered topological representative $f \from G \to G$ of $\phi \in \Out(F_n)$, and an \neg\ stratum $H_i \subset G$ consisting of a single non-fixed, oriented edge $E_i$ such that $f(E_i) = E_i u_i$ where the path $u_i$ is either trivial or is contained in $G_{i-1}$. A \emph{basic path of height~$i$} is a path having one the three forms $E_i \gamma$,~$\gamma \overline E_i$, or $E_i \gamma \overline E_i$ for some path $\gamma$ in $G_{i-1}$ which is allowed to be trivial for the first two forms. The \emph{basic splitting property for \neg\ edges} says that for each path $\gamma$ in $G$ of height~$i$ with endpoints at vertices, the path $\gamma$ splits canonically into a concatenation of basic subpaths of height $i$ and subpaths in $G_{i-1}$; the splitting points for $\gamma$ occur at the initial point of each occurrence of $E_i$ and the terminal point of each occurrence of $\overline E_i$ in $\gamma$. This splitting is called the \emph{highest edge splitting} of~$\gamma$. 
 See Lemma~4.1.4 of~\BookOne.

\paragraph{Rotationless relative train track maps, principal vertices.} Given a relative train track map $f \from G \to G$, a periodic vertex $v \in G$ is \emph{nonprincipal} if one of the following happens: $v$ is not an endpoint of a periodic Nielsen path, and there are exactly two periodic directions at $v$, both of which are contained in the same \eg\ stratum; or $v$ is contained is contained in a component $C$ of the set of periodic points such that $C$ is topologically a circle, and each point in $C$ has exactly two periodic directions. A periodic vertex which is not nonprincipal is said to be \emph{principal}. We say that $f \from G \to G$ is \emph{rotationless} if each principal vertex $v \in G$ is fixed and each periodic direction at $v$ has period one. See Definition~3.18 of \recognition\ and Section~1.5.1 of \SubgroupsOne\ for further discussion.
 
\paragraph{Rotationless outer automorphisms, \ct s, complete splitting.} Every $\psi \in \Out(F_n)$ has an iterate $\phi = \psi^k$ which is \emph{rotationless}, which implies that certain naturally occurring actions of $\phi$ on finite sets are trivial. As an example, every $\phi$-periodic free factor system is fixed by~$\phi$. It follows that if $\F \sqsubset \F'$ is an extension of $\phi$-invariant free factor systems, and if there is no $\phi$-invariant free factor system strictly between $\F$ and $\F'$, then there is no $\phi$-periodic free factor system strictly between them; in this case we say that $\phi$ is \emph{fully irreducible} relative to the extension $\F \sqsubset \F'$. Every rotationless outer automorphism is represented by a particularly nice kind of rotationless relative train track map $f \from G \to G$ called a ``\ct'' which stands for ``completely split relative train track representative''. In fact if $\F_0 \sqsubset \F_1 \sqsubset \cdots \sqsubset \F_K$ is any $\phi$-invariant filtration by free factor systems then there is a \ct\ having core filtration elements representing each~$\F_i$. In a \ct, a \emph{complete splitting} of a path or circuit is a splitting into terms that are either single edges in irreducible strata, \iNp s, exceptional paths (see below) or certain paths in zero strata. We shall often refer to the defining properties of \cts\ by their \textbf{(Parenthesized Titles)} as found in the citations below. For example, the \textbf{(Completely split)} property says that $f(E)$ is completely split for each edge $E$ of an irreducible stratum, and similarly for certain paths in zero strata. Also, the property \textbf{(Filtration)} says that for each filtration element $G_i$ of the given filtration $\emptyset = G_0 \subset G_1 \subset G_2 \subset \ldots G_N = G$, the maximal core subgraph of $G_i$ is also a filtration element, and if $[G_{i-1}] \ne [G_i]$ then $\phi$ is fully irreducible relative to $[G_{i-1}] \sqsubset [G_i]$. See Sections 3.3 and 4.1 of \recognition\ or Definitions~1.27, 1.28 and 1.29 of \SubgroupsOne.

\paragraph{Twist paths, \neg\ Nielsen paths, exceptional paths, linear families.} \quad\\ Suppose that $\fG$ is a \ct. An \neg\ edge $E$ is \emph{linear} if $f(E) = E \cdot u$ for some Nielsen path~$u$. In this case, there is a closed root-free Nielsen path $w$ such that $u= w^{d}$ for some $d \ne 0$. The path $w$ is called the \emph{twist path associated to $E$} or sometimes just a \emph{twist path}. Paths of the form $Ew^p \overline E$ are \iNp s and every \iNp\ with \neg\ height is of this form. The unoriented conjugacy class determined by $w$ is called the \emph{axis} or \emph{twistor} associated to $E$. If $E_i$ and $E_j$ are distinct linear edges with the same axes then $w_i = w_j$ and $d_i \ne d_j$. In this case we say that $E_i$ and $E_j$ belong to the same \emph{linear family}. If in addition $d_i$ and $d_j$ have the same sign then a path of the form $E_i w^s \overline E_j$ is called an \emph{exceptional path in the linear family associated to $w$} or sometimes just an \emph{exceptional path}. See the beginning of Section 4.1 of \recognition\ or Definition 1.27 or \SubgroupsOne.

\paragraph{Attracting laminations and the set $\L(\phi)$.}  Given $\phi \in \Out(F_n)$ an \emph{attracting lamination} is a set of lines $\Lambda \subset \B$ which is the closure of a single bi-recurrent, nonperiodic line $\ell \in \Lambda$ for which there is an open set $U \subset \B$ and $k \ge 1$ such that $\{\phi^{ik}(U) \bigm| i \ge 0\}$ is a neighborhood basis of $\ell$ (equivalently of $\Lambda)$. We call $U$ an \emph{attracting neighborhood} of $\ell$, and we call $\ell$ a \emph{generic line} of $\Lambda$. We use $\L(\phi)$ to denote the finite set of all attracting laminations of~$\phi$. If $\phi$ is rotationless then we may take $k=1$ in the definition of an attracting neighborhood, in which case there is a natural bijection between the set of \eg\ strata of any representative \ct\ $\fG$ and the set $\L(\phi)$. We need two characterizations of this bijection: $H_r \leftrightarrow \Lambda$ if and only if a generic leaf of $\Lambda$ has height $r$; equivalently $H_r \leftrightarrow \Lambda$ if and only if the free factor system $[G_r]$ properly contains the free factor system $[G_{r-1}]$, and $[G_r]$ is the support of the union of the lines in the lamination $\Lambda$ and the lines carried by $[G_{r-1}]$; the equivalence of these is a direct consequence of the \ct\ defining property \textbf{(Filtration)}. Also, there is a pairing between elements of $\L(\phi)$ and the elements of $\L(\phi^{-1})$ characterized by the property that paired laminations have the same free factor support. See Sections 3.1 and 3.2 of \BookOne.

Suppose that $\Lambda^+ \in \L(\phi)$ corresponds to $H_r$. By Lemma 3.1.15 of \BookOne, if both ends of a line $\ell \in \Lambda^+$ are dense in $\Lambda^+$ then $\ell$ is generic; the converse is obvious. We say that $\ell $ is \emph{semi-generic} if exactly one of its ends is dense in~$\Lambda^+$ and is \emph{ageneric} if neither of its ends is dense in~$\Lambda^+$. Lemma 3.1.15 of \BookOne\ also states that ageneric leaves are entirely contained in $G_{r-1}$ and so form a sublamination $\Lambda$, called the \emph{ageneric sublamination}, that is the unique maximal proper sublamination of $\Lambda$. The free factor support $\F_{ag}$ of the ageneric sublamination is proper because it is contained in the free factor system determined by~$G_{r-1}$.

\paragraph{Enveloping of zero strata.} If $H_r$ is an \eg\ stratum then there is at most one (up to a change of orientation) \iNp\ of height~$r$. Let $u< r$ be the maximal index for which $H_u$ is irreducible. If there is an \iNp\ of height $r$ then $u = r-1$. In the general case where $u \le r-1$, we denote $H^z_r = \bigcup_{u < i \le r} H_i$; the terms of this union are zero strata which we say are \emph{enveloped} by $H_r$, and they are the unique contractible components of $G_{r-1}$, so each component of $G_u = G_r \setminus H^z_r$ is noncontractible. See Definition 2.18 of \recognition\ or Definition 1.28 of \SubgroupsOne.


\paragraph{The expansion factor homomorphism $\PF$.}
Given $\Lambda \in \L(\phi)$ for some $\phi \in \Out(F_n)$, the stabilizer of $\Lambda$ is denoted $\Stab(\Lambda) \subgroup \Out(F_n)$. By Corollary 3.3.1 of \BookOne\ there is a homomorphism $\PF = \PF_\Lambda \from \Stab(\Lambda) \to \Z$ whose kernel consists of those $\psi$ such that neither $\L(\psi)$ nor $\L(\psi^{-1})$ contains $\Lambda$.  It  is the unique surjective homomorphism such that $\PF(\phi) > 0$, and such that there exists a number $\mu > 1$ satisfying the following properties. Let $\fG$ be any \ct\ representing a rotationless power of $\phi$, with \eg\ stratum $H_r$ corresponding to $\Lambda$. For each finite path $\sigma$ in $G$ let $EL_r(\sigma)$ denote the number of times $\sigma$ crosses edges of $H_r$. Given $\psi \in \Stab(\Lambda)$ let $g \from G \to G$ be any topological representative of $\psi$ defined on~$G$. Then $\PF_\Lambda(\psi)$ is the unique number with the property that for any $\epsilon>0$ there exists $N>0$ such that for each finite subpath $\sigma$ of a generic leaf of $\Lambda$ in $G$, if $EL_r(\sigma) \ge N$ then
$$\abs{\frac{EL_r(g_\#(\sigma))}{EL_r(\sigma)} - \mu^{\PF_\Lambda(\psi)}} \le \epsilon 
$$
When $\Lambda$ is clear from context we write simply $\PF(\psi)$. See Section 3.3 of \BookOne\ for details, in particular Definition 3.3.2, Proposition 3.3.3 and Corollary 3.3.1.

\paragraph{\pg\ versus \upg.} An outer automorphism $\phi \in \Out(F_n)$ is polynomially growing or \pg\ if and only if $\L(\phi) = \emptyset$. Assuming that $\phi$ is \pg, we say that $\phi$ is \upg\ if the action of $\phi$ on $H_1(F_n) \approx \Z^n$ is unipotent. Every \pg\ element of $\IA_n(\Z/3)$ is \upg. A \pg\ subgroup of $\Out(F_n)$ is one whose elements are all \pg, and similarly for \upg.
See Section 5.7 of \BookOne.

\paragraph{Multi-edge extensions, and \eg\ strata.} 

We say that an extension of free factor systems $\F \sqsubset \F'$ is a \emph{one edge extension} if it is realized in some marked graph $G$ by a pair of core subgraphs $H \subset H'$ such that $H' \setminus H$ consists of one edge of $G$. Otherwise, $\F \sqsubset \F'$ is a \emph{multi-edge extension}. Let $\phi \in \Out(F_n)$ be rotationless. If $\phi$ is fully irreducible relative to a properly nested extension $\F \sqsubset \F'$ of $\phi$-invariant free factor systems then the following hold: if $\F \sqsubset \F'$ is a one-edge extension then every attracting lamination carried by $F'$ is carried by $F$; whereas if $\F \sqsubset \F'$ is a multi-edge extension then there exists a unique attracting lamination $\Lambda \in \L(\phi)$ that is carried by $\F'$ but not by~$\F$. See \BookOne\ Section 3.1 and Corollary 3.2.2.
 
\paragraph{Weak attraction and the nonattracting subgroup system $\A_\na(\Lambda^\pm)$.} Consider a rotationless $\phi \in \Out(F_n)$ and a lamination pair $\Lambda^\pm$ for $\phi$. Choose a \ct\ $f \from G \to G$ representing~$\phi$. Given a conjugacy class $[a]$ of $F_n$, $a \in F_n$, and letting $\sigma$ be the circuit in $G$ representing $[a]$, one says that $[a]$ is \emph{weakly attracted} to $\Lambda^+$ if for each finite subpath $\gamma$ of some (every) generic leaf of $\Lambda^+$ (realized in $G$) there exists $K = K(\gamma)$ such that $\gamma$ is a subpath of $f^k_\#(\sigma)$ for all $k \ge K$. Weak attraction of lines---elements of $\B$---is defined similarly. Weak attraction is well-defined independent of the choice of \ct. The \emph{nonattracting subgroup system} for~$\Lambda^+$, denoted $\A_\na(\Lambda^+)$, is the unique vertex group system with the property that a conjugacy class $[a]$ of $F_n$ is not weakly attracted to $\Lambda^+$ if and only if there exists $[A] \in \A_\na(\Lambda^+)$ such that $a$ is conjugate to an element of~$A$. If $\Lambda^- \in \L(\phi^{-1})$ is paired with $\Lambda^+$ then $\A_\na(\Lambda^+) = \A_\na(\Lambda^-)$ so we usually write $\A_\na(\Lambda^\pm)$. Although not indicated in the notation, the definition of $\A_\na(\Lambda^\pm)$ depends on $\phi$ as well as $\Lambda^\pm$. See Section 6 of \BookOne\ and Section 1 of \SubgroupsThree.

\paragraph{Geometricity of \eg\ strata and attracting laminations.} Given a \ct\ $f \from G \to G$, its \eg\ strata are classified as being either \emph{geometric} or \emph{non-geometric}. Given $\phi \in \Out(F_n)$, its attracting laminations $\Lambda \in \L(\phi)$ are also classified as geometric or nongeometric, in that exactly one of the following holds: \emph{either} for every \ct\ $f \from G \to G$ representing a rotationless power of $\phi$, the \eg\ stratum corresponding to $\Lambda$ is geometric; \emph{or} for every such \ct\ the \eg\ stratum corresponding to $\Lambda$ is nongeometric. An \eg\ stratum $H_r$ is geometric if and only if there exists a closed, indivisible Nielsen path $\rho$ of height $r$, if and only if the nonattracting subgroup system $\A_\na(\Lambda_r)$ is \emph{not} a free factor system; also, $\rho$ is the unique indivisible height~$r$ Nielsen path up to reversal. Furthermore, if $H_r$ and $\Lambda_r$ are geometric, and if $H_r$ is the top stratum, then $\A_\na(\Lambda_r)$ consists of the free factor system $[G_{r-1}]$ and one additional rank~1 component~$[\<\rho\>]$. See Section 5.3 of \BookOne\ and Section 2 of \SubgroupsOne; also see Section~\ref{SectionExpansionKernel} below for a more in depth review.

\section{Free splittings and marked graph pairs}
\label{SectionFSAndMGP}

In this section we review the free splitting complex (see \cite[Section 1]{HandelMosher:FreeSplittingHyperbolic}) and we describe a new approach using marked graph pairs.

\paragraph{Free splittings.} A \emph{free splitting} of $F_n$ is a minimal simplicial action of $F_n$ on a simplicial tree $T$ with trivial edge stabilizers; we follow the convention of suppressing the action and letting $T$ stand for the free splitting. Two free splittings are equivalent if there is an equivariant homeomorphism between their trees. The homeomorphism is not assumed to be simplicial so the equivalence class of a free splitting is completely determined by the \emph{natural simplicial structure} on $T$, meaning the unique simplicial structure in which all vertices have valence at least three. Note that if two free splittings have the same underlying trees and if their $F_n$ actions differ only by conjugation by an element of $F_n$ then the free splittings are equivalent. A \emph{$k$-edge splitting} is one with $k$ orbits of natural edges.

Given a marked graph $G$ with universal cover $\wt G$, the marking on $G$ provides an identification of the group of covering translation of $\wt G$ with $F_n$ that is well defined up to composition with an inner automorphism and so determines a well defined equivalence class of free splittings. Every free splitting for which the action is proper occurs in this manner. 

\begin{remark} One can understand free splittings up to equivalence by studying an appropriate quotient object. For example, associated to a free splitting $T$ is its quotient graph of groups $T / F_n$ \cite{ScottWall}. The fundamental group $\pi_1(T/F_n)$ is defined in the category of graphs of groups \cite{Serre:trees}, and $\pi_1(T/F_n)$ is identified with $F_n$ up to conjugacy, which provides a ``marking'' for $T / F_n$. We could therefore understand free splittings as marked graphs of groups up to appropriate equivalence. Rather than pursue this line of thought, we avoid the concept of $\pi_1(T/F_n)$ altogether. Instead we pursue the more topological approach of a ``marked graph pair'', which is very close to a ``graph of spaces'' as used in \cite{ScottWall}. This leads to an understanding of free splittings as marked graph pairs up to equivalence. 
\end{remark}


\paragraph{Marked graph pairs.} Given marked graphs $(G,\rho)$, $(G',\rho')$, a homotopy equivalence $h \from G \to G'$ \emph{preserves markings} if the maps $\rho', h \rho \from R_n \to G'$ are homotopic. Equivalently, there is a lift $\ti h : \wt G \to \wt G'$ that is equivariant with respect to the $F_n$-actions on $\wt G$ and~$\wt G'$. Two marked graphs $(G,\rho)$ and $(G',\rho')$ are \emph{equivalent} if there is a homeomorphism $h : G \to G'$ that preserves markings. Equivalently, $\wt G$ and $\wt G'$ determine equivariantly homeomorphic free splittings.
 
 If $H$ is a subgraph of $G$ then we write $G - H$ for the complement of $H$ in $G$, and $G \setminus H$ for the closure of $G-H$. Thus $G \setminus H$ is the subgraph that is the union of all edges not contained in $H$.

A \emph{marked graph pair} is a pair $(G,H)$ where $G$ is a marked graph and $H$ is a proper, natural subgraph of $G$ such that $H$ contains every natural vertex of $G$. The number of natural edges in $G \setminus H$ is the \emph{co-edge number} of~$(G,H)$.
 
\begin{definition} \label{defn:relation} Define a relation $\sim$ on marked graph pairs $(G,H)$, $(G',H')$ as follows. We define $(G,H) \sim (G',H')$ if there is a homotopy equivalence of pairs $h \from (G,H) \to (G',H')$ such that 
\begin{enumerate}
\item\label{ItemPairMarkings} 
$h \from G \to G'$ preserves markings.
\item\label{ItemPairEdges}
There is a bijection $E \longleftrightarrow E'$ between the natural edges of $G \setminus H$ and the natural edges of $G' \setminus H'$ so that for each corresponding pair $E,E'$ we have $h(E) = \mu'E'\nu'$ where each of $\mu'$ and $\nu'$ is a path in $H'$, possibly trivial. 
\end{enumerate}
\end{definition}

\begin{lemma}\label{LemmaPairEqRel}
The relation of Definition~\ref{defn:relation} is an equivalence relation on marked graph~pairs. 
\end{lemma}
 
\begin{proof} The reflexive and transitive properties are clear so it suffices to assume that $h \from (G,H) \to (G',H')$ satisfies \pref{ItemPairMarkings} and \pref{ItemPairEdges}, and to produce $g \from (G',H') \to (G,H)$ which is a homotopy inverse of $h$ in the category of pairs that satisfies \pref{ItemPairEdges} (clearly $g$ automatically satisfies~\pref{ItemPairMarkings}).

Since $h$ restricts to a homotopy equivalence $H \to H'$, we may denote the components as $H = \union C_i$ and $H' = \union C'_i$, with identical index sets, so that the relation $C_i \leftrightarrow C'_i$ is a bijection of components and $h(C_i) \subset C'_i$. Lift $h$ to an equivariant map of universal covers $\ti h \from \wt G \to \wt G'$. Let $\wt H \subset \wt G$ and $\wt H' \subset \wt G'$ be the total lifts of $ \H$ and $H'$ respectively. We have subdivisions denoted $\wt H = \union \wt C_i$ and $\wt H' = \union \wt C'_i$ where $\wt C_i$ and $\wt C'_i$ are the total lifts of $C_i$ and $C'_i$ respectively, and for each $i$ we have component subdivisions denoted $\wt C_i = \union \wt C_{ij}$ and $\wt C'_i = \union \wt C'_{ij}$, with identical index sets, so that the relation $\wt C_{ij} \leftrightarrow \wt C'_{ij}$ is a bijection between components of $\wt H$ and components of $\wt H'$, and so that $\ti h(\wt C_{ij}) \subset \wt C'_{ij}$. 

Construct an equivariant map $\ti g \from \wt G' \to \wt G$ as follows. For each natural vertex $v' \in G'$ choose a lift $\ti v' \in \wt G'$. Let $C'_i$ be the component of $H'$ containing $v'$, and let $\wt C'_{ij}$ be the component of $\wt C'_i$ containing $\ti v'$. Choose a natural vertex $\ti v \in \wt C_{ij}$ and define $\ti g(\ti v') = \ti v$. Having defined $\ti g$ on one vertex in each $F_n$-orbit of natural vertices of $\wt G'$, now extend it equivariantly over all natural vertices, and then extend it equivariantly over all natural edges of $\wt G'$ so that it is either injective or constant on each edge. For each edge $\wt E' \subset \wt C'_{ij}$ having endpoints $u,w$ we clearly have $\ti g(u), \ti g(w) \in \wt C_{ij}$ and so $\ti g(\wt E') \subset \wt C_{ij}$. It follows that $\ti g (\wt C'_{ij}) \subset \wt C_{ij}$. 

The map $\ti g$ descends to a homotopy equivalence $g \from (G',H') \to (G,H)$ which is a homotopy inverse of $h$ in the category of pairs. The bijection $E \leftrightarrow E'$ between edges of $G \setminus H$ and edges of $G' \setminus H'$ lifts to a bijection $\wt E \leftrightarrow \wt E'$ between edges of $\wt G \setminus \wt H$ and edges of $\wt G' \setminus \wt H'$ so that $\ti h(\wt E) = \ti \mu' \wt E' \ti \nu'$ where $\ti\mu',\ti\nu'$ are each paths in $\wt H'$, possibly trivial. Let $\ti v_1,\ti v_2$ be the initial and terminal endpoints of $\wt E$ respectively, and let $\ti v'_1,\ti v'_2$ be the initial and terminal endpoints of $\wt E'$ respectively. Since $\ti \mu'$ is a path in $\wt H$ connecting $\ti h(\ti v_1)$ to $\ti v'_1$, there is a corresponding component pair $\wt C_{ij} \leftrightarrow \wt C'_{ij}$ such that $\ti v_1 \in \wt C_{ij}$ and $\ti v'_1 \in \wt C'_{ij}$. It follows that $\ti g(v'_1) \in \wt C_{ij}$, and so there is a path $\ti \mu \subset \wt C_{ij} \subset \wt H$ from $\ti g(v'_1)$ to $\ti v_1$. Similarly there is a path $\ti \nu \subset H$ from $\ti v_2$ to $\ti g(v'_2)$. By construction $\ti g(\wt E') = \ti \mu \wt E \ti \nu$. This completes the proof that $g$ satisfies~\pref{ItemPairEdges} and so completes the proof of the lemma.
\end{proof}
 
An equivalence class of marked graphs determines an equivalence class of proper free splittings. The same is true for an equivalence class of marked graph pairs, and furthermore every free splitting arises in this manner, as shown in the following lemma.
 
For any marked graph pair $(G,H)$, letting $\wt H \subset \wt G$ be the total lift of $\wt H$, there is a free splitting denoted $\wt G / \wt H$ which is obtained by collapsing to a point each component of $\wt H$.
 
\begin{lemma} \label{equivalent free splittings} For each equivalence class $[(G,H)]$ of marked graph pairs, the equivalence class $\<G,H\>$ of the free splitting $\wt G / \wt H$ is well defined. Moreover, $[(G,H)] \longleftrightarrow \<G,H\>$ defines a bijection between the set of equivalence classes of co-edge $k$ marked graph pairs and the set of equivalence classes of $k$-edge free splittings. 
\end{lemma}

\begin{proof} 
 Consider equivalent marked graph pairs $(G,H)$ and $(G',H')$ and choose $h \from (G,H) \to (G',H')$ satisfying Definition~\ref{defn:relation}. Equip $\wt G$ and $ \wt G'$ with actions on $F_n$ that are compatible with the markings and let $\ti h\from \wt G \to \wt G'$ be an equivariant lift. Then $\ti h$ induces a bijection between the components of $\wt H$ and of $\wt H'$, and a bijection $\wt E \longleftrightarrow \wt E'$ between the natural edges of $\wt G \setminus \ti H$ and of $\wt G' \setminus \wt H'$, so that for each $\wt E$ we have $\ti h(E) = \ti \mu'\wt E' \ti \nu'$ where $\ti \mu'$ and $\ti \nu'$ are paths in $\wt H'$, possibly trivial. After collapsing each component of $\wt H$ to a point and each component of $\wt H'$ to a point, $\ti h$ induces an equivariant map that is homotopic to an equivariant homeomorphism $\hat h\from \wt G/\wt H \to \wt G'/\ti H'$. (The induced map itself may fail to be injective on edges because the subintervals that map to $\ti \mu'$ or $\ti \nu'$ are each collapsed to points.) This proves that there is a well defined map $[(G,H)] \mapsto \<G,H\>$.

Surjectivity is well known, in that any free splitting $T$ can be 
obtained from some properly discontinuous free splitting $T'$ by 
collapsing to a point each component of some invariant subforest of $T'$; see e.g.\ \cite{HandelMosher:distortion_v1} (an early version of \cite{HandelMosher:distortion}), Section 3.2, pages 30--31, under the heading ``How to construct trees in $\mathcal K^T_n$''.

It remains to show that the map is injective. Consider marked graph pairs $(G,H)$, $(G',H')$ and the corresponding free splittings given by collapse maps
$$\pi \from G \to \wt G / \wt H = T
$$
$$\pi' \from G' \to \wt G' / \wt H' = T'
$$
Suppose that there is an equivariant homeomorphism $\tau \from T \to T'$. We must show that $(G,H)$ and $(G',H')$ are equivalent marked graph pairs, so we must produce a homotopy equivalence $h \from (G,H) \to (G',H')$ as in Definition~\ref{defn:relation}. The maps
$$\wt G \xrightarrow{\pi} T \xrightarrow{\tau} T' \xleftarrow{\pi'} \wt G'
$$
induce equivariant bijections denoted
$$\wt C_{ij} \leftrightarrow w_{ij} \leftrightarrow w'_{ij} \leftrightarrow \wt C'_{ij}
$$
amongst the sets
$$\{\text{components of $\wt H$}\} \leftrightarrow \{\text{vertices of $T$}\} \leftrightarrow \{\text{vertices of $T'$}\} \leftrightarrow \{\text{components of $\wt H'$}\}
$$
We now proceed as in the proof of Lemma~\ref{LemmaPairEqRel}: for each natural vertex $v \in G$ choose a lift $\ti v \in \wt G$, let $\wt C_{ij}$ be the component of $\wt H$ containing $v$, and choose $\ti h(v)$ to be a natural vertex in $\wt C'_{ij}$; then extend $\ti h$ equivariantly over all natural vertices and then equivariantly over all natural edges so that it is injective or constant on each edge. It follows that $\ti h(\wt C_{ij}) \subset \wt C'_{ij}$ for each corresponding component pair $\wt C_{ij} \leftrightarrow \wt C'_{ij}$. The map $\ti h$ descends to a homotopy equivalence $h\from (G, H) \to (G',H')$ that preserves markings. Thus item~\pref{ItemPairMarkings} of Definition~\ref{defn:relation} is satisfied. 

There is an equivariant bijection $\wt E \longleftrightarrow \wt E'$ between the set of natural edges $\wt E \subset \wt G \setminus \wt H$ and the set of natural edges $\wt E' \subset \wt G' \setminus \wt H'$ with the property that $\wt E'$ projects to the edge in $T'$ that is the $\tau$ image of the projection into $T$ of $\wt E$. Let $\ti v_1$ and $\ti v_2$ be the initial and terminal endpoints of $\wt E$ respectively and let $\ti v_1'$ and $\ti v_2'$ be the initial and terminal endpoints of $\wt E'$ respectively. By construction $\tau(\pi(\ti v_1)) = \pi'(\ti v'_1)$, and so we have corresponding vertices $\pi(\ti v_1)=w_{ij} \leftrightarrow w'_{ij} = \pi'(\ti v'_1)$ of $T$ and $T'$ respectively, and corresponding components $\wt C_{ij} \leftrightarrow \wt C'_{ij}$ of $\wt H$ and $\wt H'$ respectively such that $\ti v_1 \in \wt C_{ij}$ and $\ti v'_1 \in \wt C_{ij}$. Since $\ti h(\ti v_1) \in \wt C'_{ij}$ there is a path $\ti \mu' \subset \wt C'_{ij} \subset \wt H'$ from $\ti h(\ti v_1)$ to $\ti v'_1$. Similarly there is a path $\ti \nu' \subset \wt H'$ from $\ti v'_2$ to $\ti h(\ti v_2)$. By construction $\ti h(\wt E) = \ti \mu' \wt E' \ti \nu'$ so $h$ satisfies item \pref{ItemPairEdges} of Definition~\ref{defn:relation} and we are done. 
\end{proof}
 
\paragraph{The free splitting complex $\fscn$.} An equivariant simplicial map $f \from S \to T$ between free splittings is a \emph{collapse map} if $f$ is injective over the interior of each edge of $T$; an edge of $S$ belongs to the \emph{collapsed subgraph} $\sigma$ if its $f$-image is a vertex. Thus $\sigma$ is an $F_n$-invariant proper subforest and $T$ is obtained from $S$ by collapsing each component of $\sigma$ to a point. Up to equivalence of $T$ we may assume that $\sigma$ is a natural subforest.
 
The free splitting complex $\fscn$ is a simplicial complex having one $k$-simplex for each equivalence class of $(k+1)$-edge free splitting, and with the simplex corresponding to $T$ being a face in the simplex corresponding to $S$ if there is a non-trivial collapse map $S \mapsto T$, collapsing to a point each component of some proper natural subforest of $S$. Denote the first barycentric subdivision of $\fscn$ by $\fscnp$. See \cite{HandelMosher:FreeSplittingHyperbolic} for more details. 
 
In the language of marked graph pairs we have
 
 
\begin{lemma} \label{LemmaCollapseMaps} For any marked graph pair $(G,H)$ and any proper natural subgraph $H'$ that properly contains $H$, the simplex in $\fscn$ determined by $\<G,H'\>$ is a face of the simplex in $\fscn$ determined by $\<G, H\>$. Moreover, all the faces of the simplex determined by $\<G, H\>$ have this form.
\end{lemma}

\begin{proof} If $H'$ is a proper natural subgraph of $G$ that properly contains $H$ then $\wt H'$ is a proper natural forest in $\wt G$ that properly contains $\wt H $. The image of $\wt H'$ in the tree $\wt G / \wt H$ obtained from $\wt G$ by collapsing each component of $\wt H$ to a point is a proper natural forest of $\wt G / \wt H$ that contains at least one orbit of edges. Collapsing that forest defines a collapse map $\wt G / \wt H \to \wt G / \wt H'$. Since $\wt G / \wt H$ represents $\langle G,H\rangle$ and $\wt G / \wt H'$ represents $\langle G,H'\rangle$ this proves that the simplex determined by $\<G,H'\>$ is a face of the simplex determined by $\<G, H\>$. 

For the converse consider the free splitting $S = \wt G / \wt H$ and a face of its simplex corresponding to the free splitting $T$ which is obtained from $S$ by collapsing to a point each component of an $F_n$-invariant, proper, natural subforest $\sigma \subset S$. Let $V$ be the natural vertex set of $S$ and let $\wt H' \subset \wt G$ be the preimage of $\sigma \union V$ under the collapse map $\wt G \to S$. Then clearly $\wt H \subset \wt H'$ and there is a map $\wt G \to T$ that collapses each component of $\wt H'$ to a point, and so $T$ and $\wt G / \wt H'$ are equivalent. 
\end{proof}

Suppose that $T$ is a free splitting and that $a \cdot x$ denotes the image of $x \in T$ under the action of $a \in F_n$. For any $\Phi \in \Aut(F_n)$, define a new free splitting $T^\Phi$ with the same underlying tree $T$ by having the image of $x$ under the new action of $a$ be $\Phi(a) \cdot x$. If $\Phi_1$ and $\Phi_2$ determine the same element $\phi \in \Out(F_n)$ then they differ by an inner automorphism and $T^{\Phi_1}$ and $T^{\Phi_2}$ are equivalent free splittings. We therefore have a well defined right action of $\Out(F_n)$ on $\fscn$ taking $T$ to $T^\phi$ represented by either of $T^{\Phi_1}$ or $T^{\Phi_2}$. The right action equation $T^{\phi\psi} = (T^\phi)^\psi$ is easily checked. Note that if $a \in F_n$ acts elliptically on $T$ then $\Phi^{-1}(a)$ acts elliptically on $T^\Phi$. 

\newcommand\pre[1]{{}^{#1\!}}

We can express this action in the language of marked graph pairs as follows. Suppose that $G$ is a marked graph with marking $\rho \from R_n \to G$. Let $f \from G \to G$ be a homotopy equivalence representing the outer automorphism $\phi$, i.e.\ there is a homotopy equivalence $\Phi \from R_n \to R_n$ inducing an automorphism $F_n \approx \pi_1 R_n \xrightarrow{\Phi_*} \pi_1 R_n \approx F_n$ that represents $\phi$, such that $\rho\Phi$ and $f \rho$ are homotopic. Let $\pre{f}G = G^\phi$ be the marked graph with underlying graph $G$ and marking $f \rho$. For any subgraph $H \subset G$ defining a marked graph pair $(G,H)$, the marked graph pair $(\pre{f}G,H)=(G^\phi,H)$ is denoted $\pre{f}(G,H) = (G,H)^\phi$ and the free splitting that this pair determines is denoted $\pre{f}\langle G,H\rangle = \langle G,H\rangle^\phi$. When one is given a pair of self-homotopy equivalences $f,g \from G \to G$ representing $\phi,\gamma \in \Out(F_n)$, then one has action equations $\pre{fg}G = \pre{f}(\pre{g} G) = \pre{f}G^{\gamma} = (G^\phi)^\gamma = G^{\phi \gamma}$ and similarly for marked graph pairs; this is the reason for using pre-superscript notation.

Consider a homotopy equivalence of a marked graph pair $f \from (G,H) \to (G,H)$. However $f$ may change the marking on $G$, it preserves marking in the context of $f \from G \to \pre{f}G$ and of $f \from (G,H) \to \pre{f}(G,H)$. Given two such homotopy equivalences $f$ and $g$, we also have the following marking preserving maps: $g \from \pre{f}(G,H) \to \pre{gf}(G,H)$; and $gf \from (G,H) \to \pre{gf} (G,H)$. These facts are important later in several applications of Definition~\ref{defn:relation}.

\section{Theorem \ref{trichotomy}: Dynamics of elements of $\Out(F_n)$ on~$\fscn$}

In this section we prove the first of our three main results, the trichotomy theorem.

\vspace{,1in}

\noindent {{\bf Theorem~\ref{trichotomy}.} \emph{ The following hold for all $\phi \in \Out(F_n)$.}
\begin{enumerate}
\item \emph{The action of $\phi$ on $\fscn$ is loxodromic if and only if some element of $\L(\phi)$ fills.} 
\item \emph{If action of $\phi$ on $\fscn$ is not loxodromic then the action 
has bounded orbits.}
\item \emph{The action of some iterate of $\phi$ on $\fscn$ fixes a point (in fact a vertex) if and only if $\L(\phi)$ does not fill.}
\end{enumerate}
\noindent
The proof of Theorem~\ref{trichotomy} is spread across subsections representing a case analysis:
\begin{description}
\item[Section~\ref{SectionPeriodicVertex}:] If $\L(\phi)$ does not fill if and only if $\phi$ has a periodic vertex. 
\item[Section~\ref{SectionNonLoxImpliesBddOrbits}:] If no element of $\L(\phi)$ fills then $\phi$ has bounded orbits. 
\item[Section~\ref{SectionLoxChar}:] If some element of $\L(\phi)$ fills then $\phi$ is loxodromic.
\end{description}
\noindent
Also, Section~\ref{SectionFillingLamsLemmas} contains facts about filling laminations needed for Section~\ref{SectionLoxChar}.

\begin{ex} \label{filling reducible} If $\phi$ is fully irreducible then the unique element of $\L(\phi)$ fills so we are in case \pref{item:loxodromic}. It is also easy to construct reducible examples in which an element of $\L(\phi)$ fills. Let $G$ be a marked graph with an invariant subgraph $G_1 \subset G $ and unique vertex $v$ where $G_1$ is a rose of rank $m\ge 2$ and $H_2 = G \setminus G_1$ is a rank two rose with edges $A$ and $B$. Assume that the marking identifies the fundamental groups of $G_1 \subgroup G$ with $F_m \subgroup F_{m+2}$. Choose a closed path $\sigma \subset G_1$ based at $v$ such that the conjugacy class determined by $\sigma$ fills $F_{m}$. Define $\fG$ to be the identity on $G_1$ and 
$$A \mapsto A\sigma \bar B \sigma B \qquad B \mapsto B\sigma A \sigma \bar B \sigma B
$$
and let $\phi$ be the outer automorphism determined by $f$.
Then $\fG$ is a relative train track map, and $H_2$ is an \eg\ stratum with an associated lamination $\Lambda$. For each $k\ge 0$, $f^k_\#(B)$ is an initial subpath of $f_\#^{k+1}(B)$. The singular ray $R_B$ determined by $B$ is the union of the increasing sequence $$B \subset f_\#(B) \subset f_\#^2(B) \subset \ldots$$ The line $L = R_B^{-1} \sigma R_B$ is a weak limit of the subpaths $f_\#^k( \bar B\sigma B)$ of $f^{k+1}_\#(B)$ and so is a leaf of $\Lambda$. Any free factor $\F$ that carries $L$ also carries the line $L \diamond L = R_B^{-1} \sigma^2 R_B$ obtained by concatening $L$ with itself and then tightening. Similarly, $\F$ carries each $R_B^{-1} \sigma^m R_B$. Since $\sigma$ is a weak limit of these lines, $\sigma$ is carried by $\F$. Thus the smallest free factor $\F$ that carries $\Lambda$ properly contains $[F_{m}]$. Corollary 3.2.2 of \BookOne\ implies that $\phi$ does not preserve any co-rank one free factor that contains $[F_m]$ so $\F = [F_{m+2}]$ and $\Lambda$ fills. 

Note that in Example~\ref{filling reducible}, if $\theta \in \Out(F_{m+2})$ is represented by an automorphism $\Theta $ that fixes each element of the subgroup $\<A,B,\sigma\>$ then $\theta$ is represented by a homotopy equivalence of $G$ that commutes with $f$ up to homotopy rel $v$. In this case, $\theta$ commutes with $\phi$ and so is contained in the stabilizer of $\Lambda$; see Examples~\ref{linear example} and \ref{surface example}.
\end{ex}


\begin{ex}\label{ExampleBddNoPeriodic}
We can modify the above example to create an element satisfying the conclusion of \pref{item:bounded} but not satisfying \pref{item:periodic vertex}, that is, a $\phi \in \Out(F_n)$ acting with bounded orbits on $\fscn$ but with no periodic points. Define $G'$ from $G$ by adding a third stratum $H_3 = G' \setminus G$ consisting of two loops $A'$ and $B'$ attached to the unique vertex of $G$. The marking identifies $G \subset G'$ with $F_{m+2} \subgroup F_{m+4}$. Extend $f$ to $f' \from G' \to G'$ by 
$$ A' \mapsto A' \sigma \bar B' \sigma B' \qquad B' \mapsto B'\sigma A' \sigma \bar B' \sigma B'$$
Then $f'\from G' \to G'$ is a train track map with \eg\ strata $H_2$ and $H_3$ and associated laminations $\Lambda_1$ and $\Lambda_2$. As in Example~\ref{filling reducible}, $\F(\Lambda_1)$ carries $[F_m]$ and the conjugacy classes determined by the loops $A$ and $B$; and $\F(\Lambda_2)$ carries $[F_m]$ and the conjugacy classes determined by the loops $A'$ and $B'$. If $\F$ is a free factor system that carries both $\Lambda_1$ and $\Lambda_2$ then the component of $\F$ that carries $\Lambda_1$ and the component of $\F$ that carries $\Lambda_2$ both carry $[F_m]$ and so must be equal. Thus $\F$ is a single free factor that carries conjugacy classes that generate $H_1(F_{m+4})$. It follows that $\F$ is not a proper free factor and so $\L(\phi) = \{\Lambda_1,\Lambda_2\}$ fills even though neither $\Lambda_1$ nor $\Lambda_2$ fills. 
\end{ex}

\subsection{$\L(\phi)$ does not fill if and only if $\phi$ has a periodic vertex.}
\label{SectionPeriodicVertex}

In general, a simplicial automorphism of a simplicial complex has a periodic point $x$ if and only if it has a periodic vertex $v$, namely any vertex of the simplex whose interior contains~$x$. This applies in particular to elements of $\Out(F_n)$ acting on $\fscn$. The following lemma characterizes this behavior.

\begin{lemma} \label{periodic vertex} For all $\phi \in \Out(F_n)$, the following are equivalent.
\begin{enumerate}
\item[(a)] The action of some iterate of $\phi$ on $\fscn$ fixes a vertex.
\item[(b)] The action of some iterate of $\phi$ on $\fscnp$ fixes a vertex.
\item[(c)] $\L(\phi)$ does not fill.
\end{enumerate}
\end{lemma}

\begin{proof} It is obvious that (a) implies (b). 

 Assuming that (b) holds we will prove (c). There exist $k \ge 1$ and a marked graph pair $(G,H)$ such that $\langle G,H\rangle ^{\phi^k} = \langle G,H\rangle$. Equivalently, there is a homotopy equivalence $h \from (G,H) \to (G,H) $ so that $h \from (G,H) \to (G,H)^{\phi^k}$ satisfies Definition~\ref{defn:relation}. In particular, $h \from G \to G$ represents $\phi^k$. Note that for all $m \ge 1$, Definition~\ref{defn:relation} is satisfied by $h \from (G,H)^{\phi^{(m-1)k}} \to (G,H)^{\phi^{mk}}$ and hence also by $h^m \from (G,H) \to (G,H)^{\phi^{mk}}$. Since the bijection that $h$ induces on the natural edges of $G\setminus H$ has finite order, we may choose $m > 1$ so that the bijection induced by $h^m$ on natural edges is the identity. Choosing such an $m$ and replacing $\phi$ by $\phi^{mk}$ and $h$ by $h^m$ we may assume that $h$ represents $\phi$ and that $h(E) = \mu E \nu$ for each edge $E$ of $G \setminus H$ and for paths $\mu$ and $\nu$, dependent on $E$, that are either trivial or contained in $H$. It follows that $\L(\phi)$ is carried by $[H]$ and does not fill. (One way to see this is to note that for each circuit $\sigma \subset G$ there is a uniform bound to the number of times that an edge in $G \setminus H$ is crossed by $h^i_\#(\sigma)$. This proves that if $\sigma$ is weakly attracted to $\Lambda \in \L(\phi)$ then $\Lambda$ is carried by $[H]$. Since $\sigma$ is arbitrary and every element of $\L(\phi)$ weakly attracts at least one circuit, $\L(\phi)$ is carried by $[H]$.) 
 
To prove that $(c)$ implies $(a)$, suppose that $\L(\phi)$ does not fill. After replacing $\phi$ with an iterate we may assume that $\phi$ is rotationless. Since $\L(\phi)$ is $\phi$-invariant we may apply Theorem~4.28 of \recognition\ to conclude that $\phi$ is represented by a \ct\ $\fG$ such that $\L(\phi)$ is carried by a proper $f$-invariant subgraph $G_i$ of $G$. The highest stratum $H_N$ in the filtration is therefore a single edge $E$ satisfying $f(E) = uEv$ where $u,v$ are paths in the $f$-invariant subgraph $G_{N-1}$. This proves that $f \from (G,G_{N-1}) \to (G,G_{N-1})^\phi$ satisfies Definition~\ref{defn:relation} and hence that $\<G,G_{N-1}\>$ is $\phi$-invariant.
\end{proof}

\subsection{When no element of $\L(\phi)$ fills, $\phi$ has bounded orbits.} 
\label{SectionNonLoxImpliesBddOrbits}

The next lemma gives a lamination criterion for verifying that an outer automorphism acts with bounded orbits on $\fscn$. 
This lemma immediately implies the ``only if'' direction of Theorem~\ref{trichotomy}~(1), and it reduces the proof of Theorem~\ref{trichotomy}~(2) to the ``if'' direction of~(1) which will be proved later.
 
\begin{lemma} \label{bounded orbit} If $\phi \in \Out(F_n)$ and if no single element of $\L(\phi)$ fills then the action of $\phi$ on $\fscn$ has bounded orbits. 
\end{lemma}

The idea of the proof is to show that a rotationless power of $\phi$ has a topological representative having the structure of Example~\ref{ExampleBddNoPeriodic}, as made precise in the lemma to follow. After stating the lemma, we apply it to prove that $\phi$ has bounded orbits on $\fscn$, then we prove the lemma.
 
 \begin{lemma} \label{invariant subgraphs} If $\phi \in \Out(F_n)$ is rotationless, if $\L(\phi)$ fills, and if no single element of $\L(\phi)$ fills, then there is marked graph $G$, a homotopy equivalence $\fG$ representing $\phi$, and $f$-invariant proper core subgraphs $K_1 $ and $K_2$ such that: 
\begin{enumerate} 
\item \label{item:KOneCupKTwo}
$G = K_1 \cup K_2$.
\item \label{item:fixed on frontier} The frontier vertices $\Fr(K_1)$ are fixed by $f$.
\end{enumerate}
\end{lemma}
 
\begin{proof}[Proof of Lemma~\ref{bounded orbit}, assuming Lemma~\ref{invariant subgraphs}:] \,\,\,
We may assume by Lemma~\ref{periodic vertex} that $\L(\phi)$ fills.  After replacing $\phi$ with an iterate if necessary, we may also assume that $\phi$ is rotationless.   Choose $\fG$ and $K_1,K_2 \subset G$ as in Lemma~\ref{invariant subgraphs}. Let $K_3$ be the core of $K_1 \cap K_2$, let $V$ be the set of natural vertices of $G$ and let $H_i = K_i \cup V$ for $i=1,2,3$. It suffices to show that for all $k \ge 1$ there is a path in $\fscnp$ of length at most four between $\<G,H_3\>$ and $\pre{f^k}\<G,H_3\>$. 
Since \pref{item:KOneCupKTwo} and \pref{item:fixed on frontier} remain true when $f$ is replaced by $f^k$, we may assume $k=1$.
 
Define $f_1 \from G \to G$ to agree with $f$ on $K_1$ and to be the identity on $G \setminus K_1$. Define $f_2:G \to G$ to be the identity on~$K_1$ and to agree with $f$ on $G \setminus K_1$. Continuity of these maps follows from \pref{item:fixed on frontier}. The restriction $f_1 \from K_1 \to K_1$ is a homotopy equivalence by Lemma 6.0.6 of \BookOne.  It follows that $f_1 : G \to G$ is a homotopy equivalence.  To see this, choose $x \in \Fr(K_1)$ and note that each element of $\pi_1(G,x)$ is uniquely represented by a closed path based at $x$ and decomposed into an alternating  concatenation of paths $\alpha_i \subset K_1$ and $\beta_i \subset G \setminus K_1$, all with endpoints in $\Fr(K_1)$.  This decomposition is a splitting for $f_1$.  Moreover,  for any $x,y \in \Fr(K_1)$, $f_1$ induces a bijection of the set of homotopy classes  of paths in $K_1$ [resp. $G \setminus K_1$] with one endpoint at $x$ and the other at $y$.   (This is obvious for $G \setminus K_1$; for $K_1$ it is a consequence of the fact   that $f \restrict K_1$ induces an automorphism of $\pi_1(K_1,x)$.) It follows that $f_1$ induces an automorphism of  $\pi_1(G,x)$ and is hence a homotopy equivalence.  The details are left to the reader.  Using that  $f = f_2 f_1 : G \to G$ is a homotopy equivalence, it follows that    $f_2$ is a homotopy equivalence.



To complete the proof it suffices to exhibit a sequence of five vertices in $\fscnp$ of the form
$$\<G,H_3\> \longrightarrow \<G,H_1\> = \pre{f_1}\<G,H_1\> \longleftarrow \pre{f_1}\<G,H_3\> \longrightarrow\pre{f_1}\<G, H_2\> = \pre{f}\<G, H_2\> \longleftarrow \pre{f}\<G, H_3\> 
$$
such that vertices in this sequence that are connected by an arrow are either equal or bound an edge in $\fscnp$. The existence of this sequence is justified as follows.

The subgraphs $K_1, K_2, K_3$ are proper core subgraphs and so $H_1, H_2,H_3$ are natural subgraphs of~$G$ that contain all natural vertices.  Thus each of the seven vertices depicted in the above sequence is well defined. Since each arrow is induced by an inclusion of natural subgraphs, the pair of vertices that each arrow connects are either equal or bound an edge in $\fscnp$ by Lemma~\ref{LemmaCollapseMaps}. For the two equalities we apply Lemma~\ref{equivalent free splittings} and Definition~\ref{defn:relation} to $f_1 : (G,H_1) \to \pre{f_1}(G,H_1)$ and $f_2 : \pre{f_1}(G,H_2) \to \pre{f_2 f_1}(G,H_2) = \pre{f}(G,H_2)$ respectively. Both of these maps preserve marking (see the final paragraph of Section~\ref{SectionFSAndMGP}). Since each edge of $G \setminus H_1 = G \setminus K_1$ is fixed by $f_1$, it follows that $f_1 : (G,H_1) \to \pre{f_1}(G,H_1)$ satisfies Definition~\ref{defn:relation} and so $\<G,H_1\> = \pre{f_1}\<G,H_1\>$ by Lemma~\ref{equivalent free splittings}. The map $f_2 : \pre{f_1}(G,H_2) \to \pre{f}(G,H_2)$ 
also satisfies Definition~\ref{defn:relation}, because each edge $E \subset G \setminus H_2 =  E \subset G \setminus K_2$ is in $K_1$ and is therefore fixed by $f_2$, and so $\pre{f_1}\<G, H_2\> = \pre{f}\<G, H_2\>$.
\end{proof}

\begin{remark} A version of the above \lq distance four\rq\ argument is used in \FSHyp\ to correct an error in an early version of that paper; see the Remark between Steps 2 and 3 of the proof of Proposition 6.5 of \FSHyp.
\end{remark}

We will need the following facts in the proof of Lemma~\ref{invariant subgraphs}.

\begin{fact}\label{FactNielsenPathHierarchy}
In any \ct, for any stratum $H_j$, any height $j$ indivisible Nielsen path $\rho$ decomposes as an alternating concatenation of maximal subpaths in $H_j$ and Nielsen subpaths of \hbox{height~$<j$}. As a consequence, any one-edge subpath $E$ of $\rho$ of height $i<j$ is contained in a height $i$ Nielsen subpath of $\rho$. 
\end{fact}
\noindent
This fact follows from (\noneg\ Nielsen Paths) in the case when $H_j$ is \neg, and from \cite{FeighnHandel:recognition} Lemma 4.24 when $H_j$ is \eg.

\vspace{.1in}


\begin{proof}[Proof of Lemma~\ref{invariant subgraphs}.] Let $\hs$ be the set of $\Lambda \in \L(\phi)$ for which there exists a \ct\ representing $\phi$ whose top stratum corresponds to $\Lambda$. Note that $\hs\ne \emptyset$ because $\L(\phi)$ fills, and so the highest stratum in any \ct\ representing $\phi$ is \eg. Moreover, $\hs$ contains at least two elements because the smallest free factor system carrying any element of $\L(\phi)$ is proper and so is realized by a proper filtration element in some \ct\ representing~$\phi$. Note also that each element of $\hs$ is topmost, meaning that it is not contained in any other element of~$\L(\phi)$. 
 
Choose $\Lambda_1 \ne \Lambda_2 \in \hs$, let $\A'$ denote the free factor support of the set $\L(\phi) - \{\Lambda_1,\Lambda_2\}$, and choose a \ct\ $\fG$ representing $\phi$ and having a core filtration element $G_u$ such that $\A' = [G_u]$. The highest stratum $H_N$ must be an \eg\ stratum corresponding to one of~$\Lambda_1,\Lambda_2$. We choose the notation so that $H_N$ corresponds to $\Lambda_2$. Furthermore, if exactly one of these two laminations is geometric we choose $\Lambda_2$ to be that lamination, which can be accomplished by insisting that the free factor support of the set $\L(\phi)-\{\Lambda_2\}$ is also realized by a core filtration element. For now these are the only properties of $f$ we use. Later we add a further constraint to $f$ in the case that both $\Lambda_1$ and $\Lambda_2$ are geometric. 

Let $H_r$ be the stratum corresponding to~$\Lambda_1$, and let $H_s$ with $s<r$ be the next lower irreducible stratum, so the filtration element $G_s = G_r \setminus H^z_r$ has no contractible components (Lemma~4.15 of \recognition). We have $G_u \subset G_s \subsetneq G_r \subsetneq G_N=G$. Since $\Lambda_1 \not\subset \Lambda_2$ it follows from Corollary~3.1.11 of \BookOne
that no edge of $H_N$ is weakly attracted to~$\Lambda_1$ and hence no term in the complete splitting of an $f_\#$-iterate of an edge of $H_N$ is weakly attracted to~$\Lambda_1$. For any $i$ such that $r<i<N$ note that $H_i$ is not an \eg\ stratum.

We next analyze weak attraction properties of edges $E$ of height $>r$ by proving the following:
\begin{description}
\item[($i$)] If $E$ is non-fixed and non-linear then $E$ occurs as a term in the complete splitting of $f_\#^k(E')$ for some edge $E'$ of $H_N$ and some $k \ge 1$. 
\item[($ii$)] $E$ is not weakly attracted to $\Lambda_1$. 
\item[($iii$)] If a term $\mu$ in the complete splitting of $f^k_\#(E)$ crosses an edge of $H_r$ then $\mu$ is a Nielsen path or an exceptional path. 
\end{description} 
\noindent If $E$ belongs to a zero stratum then ($i$) follows from (Zero Strata) and the fact that $H_N$ is the only \eg\ stratum with height greater than $r$. If $E$ is a non-fixed non-linear irreducible edge then $E$ is crossed by $\Lambda_2$ because $\L(\phi)$ fills. It follows that $E$ is contained in a term $\nu$ of the complete splitting of $f_\#^k(E')$ for some edge $E'$ of $H_N$ and some $k \ge 1$. Item ($i$) then follows from Fact~\ref{FactNielsenPathHierarchy} and (\noneg\ Nielsen Paths). 
Item ($ii$) is obvious if $E$ is fixed or linear; in the remaining cases ($ii$) follows from ($i$) and the fact that no edge $E'$ of $H_N$ is weakly attracted to $\Lambda_1$. To prove ($iii$) we must show that $\mu$ is neither a subpath of a zero stratum nor a single edge. The former is obvious and the latter follows from ($ii$). 

\medskip

We now verify conclusions (1)--(2) of Lemma~\ref{invariant subgraphs} under a special assumption. Let $K_1 = G_r$ and note that $K_1$ is an $f$-invariant core subgraph containing each edge of $H_r$, so $K_1$ contains the lamination $\Lambda_1$ (as realized in $G$). By \cite[Remark 4.9]{FeighnHandel:recognition} and (Zero Strata), each point in the frontier of $K_1$ is a principal vertex and is hence fixed by $f$. Item (2) is therefore satisfied. Next we formulate:

\begin{description}
\item[Special Assumption:] $G \setminus H^z_r$ is $f$-invariant.
\end{description}
\smallskip
Let $K'_2 = G \setminus H^z_r$, and let $K_2 = \core(K'_2)$. Note that $G = K_1 \union K'_2$. Since $K'_2$ is $f$-invariant and contains $H_N$, it follows that $K'_2$ contains $\Lambda_2$. Also $K'_2$ contains $G_u$ and so it contains each lamination in $\L(\phi) - \{\Lambda_1,\Lambda_2\}$. Thus $K'_2$, and hence its core $K_2$, contains each lamination in $\L(\phi) - \{\Lambda_1\}$. Since $\L(\phi)$ fills, $\Lambda_1$ must cross each edge of $G \setminus K_2$, and so each such edge is contained in $K_1$, proving~\pref{item:KOneCupKTwo}.


It remains to prove that $K_2$ is $f$-invariant. Let $G_{s'}$ be the highest core filtration element that is contained in $G_s$. Then $G_r$ is built up starting from $G_{s'}$ as follows: first add the \noneg\ edges comprising $G_s \setminus G_{s'}$ to form $G_s$; then add the zero strata, if any, that are enveloped by $H_r$ to form $G_{r-1}$; and finally add $H_r$. It follows that
$$K_2 = G_{s'} \union \bigl(\text{a subgraph of $G_s \setminus G_{s'}$}\bigr) \union (G \setminus G_r)
$$
Clearly $G_{s'}$ is $f$-invariant. If $E \subset G_s \setminus G_{s'}$ then the \ct\ defining property \textbf{(Periodic edges)} implies that $E$ is not fixed and so can be oriented so that $f(E) = Eu$ for some circuit $E$ that is necessarily contained in $G_{s'}$, and so if $E \subset K_2$ then $f(E) \subset K_2 \union G_{s'} \subset K_2$. It remains to prove that if $E \subset G \setminus G_r$ then $f(E) \subset K_2$. If $E$ is fixed this is obvious. If $E$ is linear then $f(E) = E u$ where $u$ is a circuit \cite[Lemma 4.21]{FeighnHandel:recognition}, and since this circuit is contained in $K'_2$, it must be contained in $K_2=\core(K'_2)$, implying that $f(E) \subset K_2$. Otherwise, by item ($i$) it follows that $f(E)$ is a concatenation of terms of the complete splitting of $f^{k+1}_\#(E')$ for some edge $E' \subset H_N$ and some $k \ge 1$, so $f(E)$ is a subpath of $\Lambda_2$ and hence is contained in~$K_2$. 

\bigskip

It remains to justify the \emph{Special Assumption}. We divide into three cases. In Case~1 the \emph{Special Assumption} is true, and in Case~2 it becomes true after imposing an additional constraint on the choice of the \ct\ $f \from G \to G$. But Case~3 requires careful handling, in that the \emph{Special Assumption} can fail. 

\paragraph{Case 1: $\Lambda_1$ is non-geometric and there are no Nielsen paths of height $r$.} \quad\\ By Fact~\ref{FactNielsenPathHierarchy}, no Nielsen path (and hence no exceptional path) crosses an edge of $H_r$. Applying item ($iii$) above, it therefore follows that if $E$ is an edge with height $> r$ then $f^k_\#(E)$ crosses no edge of $H_r$ for $k \ge 1$. If $f(E)$ intersects a zero stratum enveloped by $H_r$ then $f(E)$ is contained in that zero stratum because it does not cross any edges of $H_r$. It follows that $E$ must be contained in a zero stratum enveloped by $H_N$; but then some $f^k_\#$-image of some edge in $H_N$ would cross an edge of~$H_r$. This contradiction implies that that $f(E) \subset G \setminus H^z_r$. Knowing that $G_s=G_r \setminus H^z_r$ is $f$-invariant, it follows that $G \setminus H^z_r$ is $f$-invariant. 
 
\paragraph{Case 2: $\Lambda_1$ and $\Lambda_2$ are geometric.} The strata $H_r$ and $H_N$ are therefore both geometric. Lemma 4.24 of \recognition\ implies that $H_r^z = H_r$ (and so $s=r-1$) and $H^z_N = H_N$, and so to verify the \emph{Special Assumption} we must show that $G \setminus H_r$ is $f$-invariant. Let $\rho_r$ be the unique indivisible Nielsen path of height~$r$ in $G$. Since $H_r$ is geometric, the path $\rho_r$ is closed with base point $x$ in $G_r$ but not in $G_{r-1}$.

We claim that, after imposing an additional constraint on the choice of $f \from G \to G$, no edge of 
height~$>r$ is incident to~$x$. Assuming this claim holds, we verify the \emph{Special Assumption} by proving that for each edge $E$ of height~$>r$ the path $f(E)$ does not cross any edge of $H_r$. The path $f(E)$ must cross at least one edge of height~$>r$, because $H_N=H^z_N$ is the unique \eg\ stratum above $H_r$ and so, by (Zero Strata), each stratum above $H_r$ is irreducible. By ($iii$) each non-trivial maximal subpath of $f(E)$ in $H_r$ is an iterate of $\rho_r$ or $\rho_r^{-1}$ and so begins and ends at $x$. But no such subpath of $f(E)$ exists, for if it did then its complement in $f(E)$ would be nonempty, would have an endpoint on~$x$, and would cross an edge of height~$>r$ incident to $x$, a contradiction. 

We now prove the claim, first describing the additional constraint. For $i =1,2$, choose a \ct\ $f_i :G^i \to G^i$ representing $\phi$ in which $\Lambda_i$ corresponds to the highest stratum~$H^i_{N_i}$. Let $\rho_i$ be the unique indivisible Nielsen path in $G^i$ of height $N_i$, a closed path representing a conjugacy class in $F_n$ that we denote $c_i = [\rho_i]$. Also let $[\<\rho_i\>]$ denote the conjugacy class of the infinite cyclic subgroup of $F_n$ generated by $\rho_i$. We~have:
\begin{description}
\item [(a)] (\SubgroupsThree, Section 1.1, \emph{``Remark: The case of a top stratum''}) \quad\\ $[\<\rho_i\>]$ is a rank~$1$ element of the subgroup system $\A_\na(\Lambda_i)$ and $\A_\na( \Lambda_i) - \{[\<\rho_i\>]\}$ is carried by $[G^i_{N_i-1}]$.
\item[(b)] (\SubgroupsOne, Proposition 2.18~(2)) \,\, The conjugacy class $c_i$ is an element of a finite $\phi$-invariant set $\C_i$ of conjugacy classes such that $\F_\supp(\C_i) = \F_\supp(\Lambda_i)$ and such that $\C_i - c_i$ is carried by $[G^i_{N_i-1}]$.
\end{description}
\noindent 
Each individual element of $\L(\phi)$ and of $\C_1 \cup \C_2$ is $\phi$-invariant by Lemma 3.30 of \recognition. The following set is therefore $\phi$-invariant:
$$L = \bigl(\L(\phi) - \{\Lambda_1,\Lambda_2\}\bigr) \cup (\C_1 - c_1 ) \cup (\C_2 - c_2)
$$
Let $\A''$ denote the free factor support of $L$. Since $\L(\phi) - \{\Lambda_1,\Lambda_2\} \subset L$ it follows that $[G_u] = \A' \sqsubset \A''$. Applying the \ct\ existence theorem, we may therefore impose an additional constraint on $f \from G \to G$ by requiring that there is a core filtration element $G_t$ such that $[G_t] = \A''$. In $G$ we thus have $G_u \subset G_t$ representing $\A' \sqsubset \A''$.

With this constraint we prove the claim. The free factor system $[G^1_{N_1-1}]$ carries the set $(\L(\phi) - \Lambda_1) \union (\C_1 - c_1) \union \C_2$, and so $[G^1_{N_1-1}]$ carries $L$ implying that $[G_t]=\A'' \sqsubset [G^1_{N_1-1}]$. But $[G^1_{N_1-1}]$ does not carry $c_1$ which is represented in $G^1$ by the closed Nielsen path $\rho_1$ of height~$N_1$. It follows that $G_t$ carries $\C_1 - c_1$ but not $c_1$. Using from \textbf{(b)} that $\C_i$ and $\Lambda_i$ have the same free factor support, it follows that $G_t$ does not carry $\Lambda_1$, and therefore $G_t \subset G_{r-1}$. Since $\Lambda_1$ is carried by $G_r$ but not by $G_{r-1}$, and since $\C_1 - c_1$ is carried by $G_t$, it follows from \textbf{(b)} that $c_1$ is carried by $G_r$ but not by $G_{r-1}$. Since $c_1$ is also carried by $\A_\na(\Lambda_1)$, Proposition 2.18~(4) of \SubgroupsOne\ implies that $c_1 = [\rho_r]$ up to a change of orientation. 

We proved ($ii$) earlier saying that no edge with height $> r$ is weakly attracted to $\Lambda_1$, and the same is clearly true of edges of height $<r$. It follows that there does not exist any closed path in $G \setminus H_r$ whose base point equals the base point~$x$ of $\rho_r$, for if such a closed path existed then together with~$\rho_r$ it would generate a rank~$2$ subgroup of $F_n$ supported by $\A_\na(\Lambda_1)$ and containing $\<\rho_r\>$, contradicting \textbf{(a)} and the fact that the subgroup system $\A_\na(\Lambda_1)$ is malnormal \cite[Proposition 1.4]{HandelMosher:SubgroupsIII}. Pick a generic leaf $\ell$ of~$\Lambda_2$. By ($iii$), each non-trivial maximal subpath of $\ell$ contained in $H_r$ is an iterate of $\rho_r$ or $\rho_r^{-1}$ and so begins and ends at $x$, and if any such subpath exists then its complementary subpaths in $\ell$ begin and end at $x$, a contradiction. Thus $\ell$ is contained in $G \setminus H_r$. If $\ell$ crosses $x$ then, since $\ell$ is birecurrent, there is a closed path in $G \setminus H_r$ based at~$x$, a contradiction. Since $\ell$ does not cross~$x$ but does cross every edge of $G \setminus H_r$, no such edge is incident to~$x$, proving the claim.


\paragraph{Case 3: $\Lambda_1$ is non-geometric and there is a Nielsen path of height $r$.} \quad \\ 
In this case $H^z_r = H_r$ (and $s=r-1$) by Lemma 4.24 of \recognition. But $G \setminus H_r = G \setminus H^z_r$ is not necessarily $f$-invariant. The height~$r$ Nielsen path $\rho_r$ is unique up to orientation, it has distinct endpoints $p,q \in G_r$, and it may be oriented with initial vertex $p$ and terminal vertex $q \not\in G_{r-1}$ (see Fact~1.42 of \SubgroupsOne). In terms of $\rho_r$ we shall see exactly how invariance of $G \setminus G_r$ can fail, and using this mode of failure as a guideline we shall then modify $f \from G \to G$.


Any Nielsen path (and hence any exceptional path) that crosses an edge in $H_r$ decomposes as a concatenation of subpaths that are either contained in $G \setminus H_r$ or are equal to $\rho_r$ or~$\rho_r^{-1}$---this follows by applying Fact~\ref{FactNielsenPathHierarchy} to each of the fixed edges and indivisible Nielsen paths into which the given Nielsen path decomposes. Combining this with ($iii$) it follows that if $E$ is an edge in $G \setminus H_r$ then $f(E)$ decomposes as a concatenation of subpaths $f(E) = \sigma_1 \ldots \sigma_m$ where each $\sigma_j$ is either contained in $G \setminus H_r$ or equal to $\rho_r$ or $\rho_r^{-1}$. 
 
Let $\D$ be the set of directions in $G\setminus G_r$ that are based at the terminal point $q$ of $\rho_r$. If $\D = \emptyset$ then for each edge $E \subset G \setminus H_r$ the path $f(E)$, which has endpoints in $G \setminus H_r$, must be contained in $G \setminus H_r$. The subgraph $K_2 = G \setminus H_r$ is therefore $f$-invariant, and we are done. 

Assuming that $\D \ne \emptyset$, we modify $\fG$ as follows. Define a new graph $G'$ by peeling a new oriented edge off of the path $\rho_r$ and peeling the directions $\cal D$ along with that new path. To be precise, detach the edges associated to $\cal D$ from $q$, reattach them to a new vertex $q'$, and add a new oriented edge $X'$ with initial endpoint $p$ and terminal endpoint $q'$. We view $K_1= G_r$ and $H_r \subset G_r$ as subgraphs of both $G$ and $G'$. Folding $X'$ with $\rho_r$ determines a homotopy equivalence $h: G' \to G$ that we use to mark $G'$. 
Letting $Z = G \setminus H_r$ and $Z' = G' \setminus (H_r \cup X')$, 
the map $h$ restricts to a homeomorphism $Z' \mapsto Z$. We may identify edges and edge paths in $Z'$ with edges and edge paths in $Z$ via the map $h$, writing $E' \leftrightarrow E$ for corresponding edges $E' \subset Z'$, $E \subset Z$, and $\tau' \leftrightarrow \tau$ for corresponding edge paths. The homotopy equivalence $\fG$ induces a homotopy equivalence $f':G' \to G'$ representing $\phi$ that agrees with $f$ on $G_r$, that fixes $X'$, and that takes each edge $E' \subset Z'$ to $f'(E') = \tau_1 \ldots \tau_m$ where $f(E) = \sigma_1 \ldots \sigma_m$ is as above and where $\tau_j = \sigma_j' \subset Z'$ if $\sigma_j \subset Z$ and $\tau_j = X'$ or $X'^{-1}$ if $\sigma_j = \rho_r$ or $\rho_r^{-1}$ respectively. This proves that $K_2 = Z'\cup X'$ is $f'$-invariant. 
We leave it to the reader to verify the remaining conclusions of Lemma~\ref{invariant subgraphs} as they were proved earlier under the \emph{Special Assumption}.

This completes the proof of Lemma~\ref{invariant subgraphs}, which completes the proof of Theorem~\ref{trichotomy}~(2) modulo the proof of~(1).
\end{proof}

\subsection{Lemmas on filling laminations}
\label{SectionFillingLamsLemmas}
Preparatory to the proof that existence of a filling lamination is sufficient for an outer automorphism to act loxodromically on $\fscn$, in this section we prove some facts about filling laminations.

Every \ffs\ $\F$ is realized as a core subgraph $H$ of some marked graph $G$. If $k$ is the minimum number of natural edges of $G$ contained in $G \setminus H$ over all choices of such a pair $H \subset G$, then we say that $\F$ is a \emph{co-edge $k$} free factor system and that $k$ is the \emph{co-edge number} for~$\F$.

 \begin{lemma}\label{fact:co-edge} \quad
\begin{enumerate}
\item \label{item:formula} If $\F = \{[F_1],\ldots,[F_p]\}$ and if $r_i$ is the rank of $F_i$ then the co-edge number of $\F$ is $k= (n - \sum r_i) + (p-1)= (n-1) - \sum_{i=1}^p (r_i-1)$.
 \item \label{item:co-edge inequality} If $\F_i$ is a free factor system with co-edge number $k_i$ for $i=1,2$, and if $\F_1 \sqsubset \F_2$, then $k_2 \le k_1$ with equality if and only if $\F_1$ is obtained from $\F_2$ by removing $\le k_2$ rank one components from $\F_2$. 
 \end{enumerate}
\end{lemma}

\begin{remark} We note the following consequences of Lemma~\ref{fact:co-edge}. First, given a free factor system $\F$, the extension $\F \sqsubset \{[F_n]\}$ is a one edge extension (see Section~\ref{SectionBackground}) if and only if $\F$ has co-edge number~$1$. More generally, an extension of free factor systems $\F = \{[F_1],\ldots,[F_p]\} \sqsubset \{[F'_1],\ldots,[F'_{p'}]\} = \F'$ is a one-edge extension if and only if one of the following two alternatives holds: each component of $\F'$ contains some component of $\F$ and $\sum_{i=1}^{p'} (\rank F'_i-1) - \sum_{j=1}^{p} (\rank(F_j)-1) = 1$; or $\F'$ is the union of $\F$ and a rank~$1$ component.
\end{remark}

\begin{proof} There is no loss in restricting attention to pairs $H \subset G$ in which each component of $H$ is a rose and to simplicial structures on $G$ in which all vertices have valence at least three. In the context of \pref{item:formula}, $H$ has $p$ vertices and $ \sum r_i$ edges. If $G$ has $q$ additional vertices then the obvious Euler characteristic calculation shows that the number of natural edges in $G \setminus H$ is $(n - \sum r_i) + (p-1) +q$, which is minimized by choosing $q=0$. This proves \pref{item:formula}.

For \pref{item:co-edge inequality} define $\F_2'$ by removing components from $\F_2$ that do not contain any element carried by $\F_1$ and let $k_2'$ be the co-edge number of $\F_2'$. Then $\F_1 \sqsubset \F_2'\sqsubset \F_2$ and the second formula in (1) implies that $k_2 \le k_2'$ with equality if and only if each removed component has rank one. Let $n_1$ and $n_2'$ be the sum of the ranks of the components of $\F_1$ and $\F_2'$ respectively and let $p_1 $ and $p_2'$ be the number of components of $\F_1$ and $\F_2'$ resepctively. Then $p_2' \le p_1$ and $n_2' \ge n_1$ so the first formula in (1) implies that $k_2' \le k_1$ with equality if and only if $p_2' = p_1$ and $n_2' = n_1$. Thus $k_2 \le k_1$ with equality if and only if each removed component of $\F_2$ has rank one and $\F_1 = \F_2'$. Since $k_2 \ge p_2-1$, at most $k_2$ such components can be removed. This completes the proof of \pref{item:co-edge inequality}.
\end{proof} 

\break

\begin{lemma}\label{ANA} Suppose that $\phi$ is rotationless and that $\Lambda \in \L(\phi)$ fills. 
\begin{enumerate}
\item\label{ItemANANotOneEdge}
Letting $\A = \A_{\na}(\Lambda)$ be the non-attracting subgroup system,
\begin{enumerate}
\item\label{item: nongeometric} If $\Lambda$ is non-geometric then $\A$ is a co-edge $\ge 2$ free factor system.
\item \label{item: geometric} If $\Lambda$ is geometric then 
$\A$ is obtained from a co-edge $\ge 2$ free factor system by adding a rank one component. 
\end{enumerate}
\item\label{ItemPFFSNotOneEdge}
Every proper $\phi$-invariant free factor system $\F$ has co-edge number $\ge 2$.
\end{enumerate}
\end{lemma}

\begin{proof} Let $\F$ be a proper, $\phi$-invariant free factor system. Let $f \from G \to G$ be a \ct\ with top stratum $H_r$ in which $\F$ is realized by a core filtration element $G_s \subset G_{r-1}$. Since $\Lambda$ fills, $H_r$ is an \eg\ stratum and $\Lambda$ is associated to $H_r$. Remark 1.3 of \SubgroupsThree\ implies that if $\Lambda$ is non-geometric then $\A = [G_{r-1}]$ which is a proper $\phi$-invariant free factor system, and if $\Lambda$ is geometric then $\A$ is obtained from $[G_{r-1}]$ by adding a rank one component. Conclusion~\pref{ItemANANotOneEdge} therefore follows from Corollary 3.2.2 of \BookOne\ which implies that the co-edge number of $[G_{r-1}]$ is $\ge 2$, and combining this with $\F \sqsubset [G_{r-1}]$ and with Lemma~\ref{fact:co-edge}, Conclusion~\ref{ItemPFFSNotOneEdge} also follows.
\end{proof}

\begin{lemma}\label{not in ANA} Suppose that $\phi$ is rotationless, that $\Lambda \in \L(\phi)$ fills and that $\A = \A_{\na}(\Lambda)$ is its non-attracting subgroup system. 
\begin{enumerate}
\item \label{item: one edge case} If $\F$ is a co-edge one \ffs\ then $\F$ carries a conjugacy class that is not carried by $\A$.
\item \label{item: two edge case} If $\F$ is a co-edge two \ffs\ then one of the following holds.
\begin{enumerate}
\item $\F$ carries a conjugacy class that is not carried by $\A$. 
\item $\F$ is obtained from $\A$ by removing $\le 3$ rank one components.
\end{enumerate}
\end{enumerate}
\end{lemma}

\begin{proof} If $\Lambda$ is non-geometric then $\A$ is a co-edge $\ge 2$ free factor system by Lemma~\ref{ANA}\pref{item: nongeometric}. The free factor system $\F \meet \A$ is contained in $\F$ and carries all the conjugacy classes of $\F$ that are carried by $\A$. If $\F \meet \A \ne \F$ then $\F$ carries a conjugacy class that is not carried by $\A$ and we are done. If $\F \meet \A = \F$ then $\F \sqsubset \A$. Since $\F$ has co-edge number $\le 2$ and $\A$ has co-edge number $\ge 2$, Lemma~\ref{fact:co-edge}\pref{item:co-edge inequality} implies that both $\F$ and $\A$ have co-edge number two and that $\F$ is obtained from $\A$ by removing $\le 2$ rank one components. 

If $\Lambda$ is geometric then by Lemma~\ref{ANA}\pref{item: geometric}) there is a rank one component $[\<\rho\>]$ of $\A$ whose complement $\A'$ in $\A$ is a co-edge $\ge 2$ free factor system. If $[F]$ is a component of $\F$ and $[F] \wedge \A' \sqsubset [F]$ is a proper inclusion, then there are many conjugacy classes in $[F]$ that are not carried by $\A'$ and we can choose one that is not contained in $[\<\rho\>]$ and hence not carried by $\A$. We are therefore reduced to the case that $\F \sqsubset \A'$. As above, it follows that $\F$ and $\A'$ have co-edge number two and that $\F$ is obtained from $\A'$ by removing $\le 2$ rank one component. Thus $\F$ is obtained from $\A$ by removing $\le 3$ rank one components and we are done. 
\end{proof}

 \begin{ex} Let $h:S \to S$ be a pseudo-Anosov homeomorphism of the orientable genus zero surface $S$ with four boundary components and let $\phi \in \Out(F_3)$ be the outer automorphism of $F_3$ determined by $h$ and an identification of $F_3$ with the fundamental group of $S$. Then $\L(\phi)$ has a single filling element $\Lambda$ and $\A_\na\Lambda$ has four rank one components, one for each component of $\partial S$. Any one, two, or three of these components form a co-edge two free factor system. Similar examples can be made in which $\phi$ has a geometric stratum but is not itself geometric. \end{ex}

\subsection{When some element of $\L(\phi)$ fills, $\phi$ is loxodromic.} 
\label{SectionLoxChar}

As explained at the beginning of Section~\ref{SectionNonLoxImpliesBddOrbits}, to prove Theorem~\ref{trichotomy}~(1) all that remains is to prove that if $\phi \in \L(\phi)$ has an attracting lamination $\Lambda^+_\phi \in \L(\phi)$ which fills $F_n$ then $\phi$ acts loxodromically on $\fscn$; see Corollary~\ref{loxodromic} below. The method of proof is to construct a map $W \from \fscn \mapsto \Z$ (Definition~\ref{W(F)}) which is equivariant with respect to the cyclic group $\<\phi^i\>$ acting on $\FS(F_n)$ and on $\Z$ where $\phi^i \cdot j = i+j$ (Lemma~\ref{BigW} \pref{BigWonS}) 
 and which is Lipschitz (Lemma~\ref{adjacent splittings}). 

The method of proof we use applies as well to show the result from \cite{BestvinaFeighn:FFCHyp}, Theorem 9.3, saying that if $\phi \in \Out(F_n)$ is fully irreducible then $\phi$ acts loxodromically on the free factor complex; see Remark~\ref{RemarkFSLox}.
 
The following lemma is used below, in Lemmas~\ref{WellLemma} and \ref{distinct ends}, in order to apply Proposition 3.1 of \SubgroupsFour.

\begin{lemma} \label{small intersection} Suppose that $\Lambda^+_\phi \in \L(\phi)$ and $\Lambda^+_\psi \in \L(\psi)$ are filling laminations with generic leaves $\gamma_\phi$ and $\gamma_\psi$ respectively. Assume that $\Lambda^+_\phi \ \ne \Lambda^+_\psi$. Then there is a proper free factor system that carries $\Lambda_\phi^+ \cap\Lambda_\psi^+$ and does not carry an end of either $\gamma_\phi$ or $\gamma_\psi$.
\end{lemma}

\begin{proof} 
Assuming without loss that $\Lambda_\phi^+ \not \subset \Lambda_\psi^+$, no leaf of $\Lambda_\phi^+ \cap\Lambda_\psi^+$ has closure equal to~$\Lambda_\phi^+$. The existence of a proper free factor system carrying $\Lambda_\phi^+ \cap\Lambda_\psi^+$ therefore follows from Lemma~3.1.15 of \BookOne. Since $\gamma_\phi$ and $\gamma_\psi$ are birecurrent and filling their ends are not carried by any proper free factor system.
\end{proof}

Suppose that $\phi, \phi^{-1} \in\Out(F_n)$ are rotationless, that $\Lambda^\pm $ is a lamination pair for $\phi$ and that the conjugacy class $c$ is not carried by $\A_{\na}(\Lambda^\pm)$. As $i \to \infty$, $\phi^i(c)$ contains longer and longer subpaths in common with $\Lambda^+$ and does not contain any very long subpaths of $\Lambda^-$ (see Section~\ref{SectionBackground}). Symmetrically, $\phi^{-i}(c)$ contains longer and longer subpaths in common with $\Lambda^-$ and does not contain any very long subpaths of $\Lambda^+$. In the middle, there is an interval of integers $w$ such that $\phi^w(c)$ has neither very long subpaths of $\Lambda^+$ nor very long subpaths of $\Lambda^-$. Intuitively this interval is the ``well'' of $c$, which we shall formalize using ``well functions''. The following definition and lemma makes this precise in our current context. See \cite{AlgomKfir:Contracting} for a similar definition.


\begin{definition}[Wells of conjugacy classes] \label{w(c)}
For the rest of the section we fix a rotationless $\phi \in \Out(F_n)$ and a dual lamination pair $\Lambda^+ \in \L(\phi)$,  $\Lambda^- \in \L(\phi)$ that fills. Let $\gamma^+$ and $\gamma^-$ be generic leaves of $\Lambda^+$ and $\Lambda^-$ respectively. Lemma~\ref{small intersection} (applied with $\psi = \phi^{-1}$) implies that the subsets $B_1 = \{\gamma^+\}$, $B_2 = \{\gamma^-\} \subset \B$ satisfy the hypotheses of Proposition 3.1 of \SubgroupsFour. The conclusion of that proposition is the existence of neighborhoods $U^+, U^- \subset \B$ of $\gamma^+$, $\gamma^-$ respectively so that the following holds:

\break

\begin{description}
\item[Separation of proper free factor systems:] For any proper free factor system $\F$,
\begin{itemize}
\item either $\F$~carries no line nor conjugacy class that is carried by $U^-$,
\item or $\F$ carries no line nor conjugacy class that is carried by $U^+$. 
\end{itemize}
\end{description}
Choose smaller neighborhoods if necessary, there is no loss in assuming that $U^+$ and $U^-$ are attracting neighborhoods for $\gamma^+$ and $\gamma^-$ respectively (see Definitions~3.1.1 and 3.1.5 of \BookOne). Thus $\phi(U^+) \subset U^+$ and $\phi^{-1}(U^-) \subset U^-$, and the collections of subsets $\{\phi^{k}(U^+) \suchthat k \ge 0\}$ and $\{\phi^{-k}(U^-) \suchthat k \ge 0\}$ are neighborhood bases in $\B$ for $\gamma^+$ and $\gamma^-$ respectively.

Recall from Theorem~F of \SubgroupsThree\ that the following are equivalent for any conjugacy class $c$ in $F_n$, namely: $c$ is not carried by $\A_{\na}\Lambda^\pm$; $c$ is weakly attracted to $\Lambda^+$ under the action of $\phi$; and $c$ is weakly attracted to $\Lambda^-$ under the action of $\phi^{-1}$. When these conditions are satisfied we say that $c$ is {\em weakly attracted}. For any weakly attracted $c$ there is a maximal integer $w_+(c) \in \Z$ such that $c \in \phi^{w_+(c)}(U^+)$ or equivalently $\phi^{-w_+(c)}(c) \in U^+$; similarly, there is a maximal integer $w_-(c) \in \Z$ such that $c \in \phi^{-w_-(c)}(U^-)$ or equivalently $ \phi^{w_-(c)}(c) \in U^-$. We refer to $w_+$ and $w_-$ as the {\em well functions of $\phi$ with respect to $U^+$ and $U^-$} respectively, and we emphasize that the domain of these functions is the set of weakly attracted conjugacy classes.

If $c$ is not carried by any a proper free factor system then we say that $c$ is {\em filling}. 
\end{definition}

\begin{lemma} \label{WellLemma} The following hold for all weakly attracted conjugacy classes $c$:
\begin{enumerate}
\item \label{WellAndIteration} $w_+(\phi^k(c)) = w_+(c)+k$ and $w_-(\phi^k(c)) = w_-(c)-k$ for all $k$.
\item \label{WellAndZero} $w_+(c) \ge 0$ if and only if $c \in U^+$ and $w_-(c) \ge 0$ if and only if $c \in U^-$. 
\end{enumerate}
Furthermore there is a positive integer $M$ such that the following also hold for all weakly attracted $c$:
\begin{enumeratecontinue}
\item \label{KeyWellProperty} 
If $c$ is not filling then $\abs{w_+(c) + w_-(c) } \le M$. 
\item \label{item:same free factor} If $c_1$ and $c_2$ are weakly attracted and if there is a proper free factor system that carries both $c_1$ and $c_2$ then $\abs{w_+(c_1) - w_+(c_2)}\le M$ and $\abs{w_-(c_1) - w_-(c_2)} \le M$.
\end{enumeratecontinue}
\end{lemma}

\subparagraph{Remark.} 
Intuitively the ``well of $c$'' may be thought of as the subinterval of $\Z$ bounded by the integers $w_+(c)$ and $-w_-(c)$, whose length equals $\abs{w_+(c) + w_-(c)}$. But this ``well'' behaves badly when $c$ fills, for instance there no bound to its length: one may construct a single $c$ containing any given leaf segment of $\Lambda_-$ and any given leaf segment of $\Lambda_+$, and those leaf segments may be chosen so that $w_+(c)$ and $w_-(c)$ both exceed any given positive number. Amongst those $c$ that do \emph{not} fill, item~\pref{KeyWellProperty} of the lemma pins down the size of the well. Also, item~\pref{item:same free factor} pins down the location of the wells of those conjugacy classes supported in a given proper free factor system, enabling us to extend the concept of ``well'' in Definition~\ref{W(F)}.

\begin{proof} Items \pref{WellAndIteration} and \pref{WellAndZero} are immediate consequences of the definitions. From \pref{WellAndIteration} we see that $ w_+(\phi^k(c)) + w_-(\phi^k(c))$ is independent of $k$. In proving \pref{KeyWellProperty} we may therefore assume that $w_-(c) = 0$, and so $c \in U^-$. Since $U^-,U^+$ satisfy separation of proper free factor systems, we have $c \not \in U^+$, and so by \pref{WellAndZero} we have $w^+(c) < 0$. By Theorem H of \SubgroupsThree\ there exists $M'$ such that for all weakly attracted $c'$, if $c' \not\in U^-$ then $\phi^{M'}(c') \in U^+$. By \pref{WellAndIteration} and \pref{WellAndZero} this applies to $c' = \phi(c)$ and so $\phi^{M'+1}(c) \in U^+$. A final application of \pref{WellAndIteration} and \pref{WellAndZero} shows that $w^+(c) \ge -M'-1$ and so \pref{KeyWellProperty} holds with $M = M'+1$.

We will prove the first inequality of \pref{item:same free factor}; the second follows by replacing $\phi$ with $\phi^{-1}$. We may assume without loss that $w_+(c_1) > w_+(c_2)$. By \pref{WellAndIteration} and \pref{WellAndZero} we may assume that $w_+(c_1) = 0$ and $c_1 \in U^+$. Since $U^-,U^+$ satisfy separation of proper free factor systems, it follows that $c_2 \not \in U^-$. Thus $w_-(c_2) < 0$ and so $w_+(c_2) > -M = -w_+(c_1)-M$ by \pref{KeyWellProperty}. Since $w_+(c_2) <  0 = -w_+(c_1)$, we are done. 
\end{proof}

\medskip

Lemma~\ref{WellLemma}~\pref{KeyWellProperty} shows that $w_+(c)$ and $-w_-(c)$ are coarsely equivalent functions when restricted to nonfilling $c$, and so for most of our purposes are interchangeable. In what follows we focus on $w_+(c)$.

\begin{definition}[Wells of free factor systems and of one-edge free splittings]
 \label{W(F)}Suppose that $\F$ is a proper free factor system that carries at least one weakly attracted conjugacy class. 
Define $W(\F)$ to be the minimum value of $w_+(c)$ as $c$ varies over all weakly attracted conjugacy classes that are carried by $\F$. Lemma~\ref{WellLemma}~\pref{item:same free factor} guarantees that $W(\F)$ is well defined. If $S$ is a one edge free splitting then its elliptic subgroups determine a co-edge one free factor system $\F(S)$ and so $W(S) = W(\F(S))$ is defined, because Lemma~\ref{not in ANA}~\pref{item: one edge case} guarantees the existence of a conjugacy class carried by $\F(S)$ but not by $\A_\na(\Lambda^\pm)$.
\end{definition} 

We record the following observations for easy reference.

\begin{lemma}\label{BigW}
 \begin{enumerate}
\item \label{BigWonF} $W( \phi^m(\F)) = W(\F) + m$ for each integer $m$ and each proper free factor system $\F$ that carries at least one weakly attracted conjugacy class.
\item \label{BigWonS}$W(S^{\phi^m}) = W(\F(S^{\phi^m})) = W(\phi^{-m}(\F(S))) = W(\F(S)) - m = W(S)-m$ for all $m$ and all one edge free splittings $S$.
\end{enumerate}
\end{lemma}

\begin{proof} Item \pref{BigWonF} follows from Lemma~\ref{WellLemma}~\pref{WellAndIteration} and the observation that $c \in \F$ is weakly attracted if and only if $\phi^m(c) \in \phi^m(\F)$ is weakly attracted. The second equality in \pref{BigWonS} follows from the fact that $[a]$ acts elliptically on $S$ if and only if $\phi^{-1} {[a]}$ acts elliptically on $S^\phi$. The first and fourth equality follow from the definitions and the third follows from~\pref{BigWonF}.
\end{proof}

 \begin{lemma} \label{bounded by 2M} If $\F_1$ and $\F_2$ are proper free factor systems and if there exists a weakly attracted conjugacy class $c$ that is carried by $\F_1$ and $ \F_2$ then $ | W(\F_1) - W(\F_2) | \le 2M$. 
 \end{lemma}
 
\begin{proof} $| W(\F_1) - \W(\F_2) | \le |W(\F_1) - w_+(c)| + | w_+(c) - W(\F_2)| \le M+M$ by Lemma~\ref{WellLemma}~\pref{item:same free factor}.
 \end{proof}

\begin{lemma} \label{adjacent splittings} If $S_1$ and $S_2$ are one-edge free splittings that bound an edge in $\fscn$ then $\abs{W(S_1) - W(S_2)} \le 8M$. \end{lemma}

\begin{proof} By Lemma~\ref{LemmaCollapseMaps} there is a co-edge two natural marked graph pair $(G,H)$ with the edges of $G \setminus H$ labelled $E_1,E_2$ so that $S_1 = \<G, H \cup E_1\>$ and $S_2 = \<G, H \cup E_2\>$. Recalling that $n = \rank(F_n) \ge 3$, it follows that $H$ is a nonempty subgraph of~$G$; recall also that each of its components is noncontractible. After collapsing a maximal subforest in $H$, we may assume that each component of $H$ is a rose. 

If $[H] $ carries a conjugacy class that is not carried by $\A_{\na}\Lambda^\pm$ then Lemma~\ref{bounded by 2M}, applied with $\F_i = [H \cup E_i]$, implies that $|W(S_1) - W(S_2)| \le 2M$. We may therefore assume that every element of $[H]$ is carried by $\A_{\na}\Lambda^\pm$. By Lemma~\ref{not in ANA}~\pref{item: two edge case}, there exist $x_1,x_2 ,x_3 \in F_n$ such that any conjugacy class carried by $\A_{\na}\Lambda^\pm$ but not by $[H]$ is contained in the set $\{[x_i^k] \suchthat 1 \le i \le 3, k \in \Z\}$. 

There are four cases to consider. The first three use the following method for bounding $|W(S_1) - W(S_2)|$.

Suppose that $\sigma_0$ is a closed path in $H$ and that for $i=1,2$, $\sigma_i$ is a closed path in $H\cup E_i$ that is not contained in $H$. Suppose further that all three paths have a common basepoint $v$ and that the elements $a,b,c \in F_n$ determined by $\sigma_0,\sigma_1$ and $\sigma_2$ respectively (under an identification of $\pi_1(G,v)$ with $F_n$) are part of a free basis for $F_n$, generating a rank~$3$ free factor $\<a,b,c\> \subgroup F_n$. Choose $m > 0$ so that each of $[ab^m], [ac^m], [bc^m]$ and $[cb^m]$ is not contained in $\{[x_i^k]; 1 \le i \le 3; k \in \Z\}$ and hence not carried by $\A_{\na}\Lambda^\pm$. Consider the rank two free factors 
$$\underbrace{\<a,b\>}_{X_1} \qquad \underbrace{\<ab^m,bc^m\>}_{X_2} \qquad \underbrace{\<b,c\>}_{X_3} \qquad \underbrace{\<ac^m,cb^m\>}_{X_4} \qquad \underbrace{\<a,c\>}_{X_5}
$$
The elements
$$ab^m \qquad\qquad bc^m \qquad\qquad cb^m \qquad\qquad ac^m
$$
are contained in 
$$X_1\cap X_2 \qquad X_2 \cap X_3 \qquad X_3 \cap X_4 \qquad X_4 \cap X_5
$$ respectively. Since $ab^m$ is carried by $[\<a,b\>] \sqsubset [H\cup E_1]$ and $ac^m$ is carried by $[\<a,c\>] \sqsubset [H\cup E_2]$, four applications of Lemma~\ref{bounded by 2M} imply that $\abs{W(S_1) - W(S_2)} \le 8M$.
 
We now turn to the case analysis, depending on how the subgraph $H$, which has at most three components, interacts with the edges $E_1,E_2$.
 
If $H$ is connected then $G$ is a rose and we let $\sigma_0$ be one of the loops in $H$ and $\sigma_i = E_i$ for $i=1,2$. 
 
If $H$ has three components, we can label them $H_0, H_1$ and $H_2$ where $H_0$ contains the initial vertices of both $E_1$ and $E_2$ and $H_i$ contains the terminal vertex of $E_i$ for $i =1,2$. For $j=0, 1,2$ let $\tau_j$ be a loop based at the unique vertex of $ H_j$. Let $\sigma_0 = \tau_0$ and for $i =1,2$ let $\sigma_i = E_i \tau_i \overline E_i$. 
 
 If $H$ has two components $H_0$ and $H_1$ and one of $E_1$ or $E_2$ is a loop then we may assume that $E_1$ has initial vertex in $H_0$ and terminal vertex in $H_1$ and that $E_2$ is a closed path with basepoint in $H_0$. For $j=0, 1$ let $\tau_j$ be a loop based at the unique vertex of $ H_j$. Let $\sigma_0 = \tau_0$, $\sigma_1 = E_1 \tau_1 \overline E_1$ and $\sigma_2 = E_2$.

The remaining case is that $H$ has two components $H_0$ and $H_1$ and neither $E_1$ nor $E_2$ is a loop. We may assume that both $E_1$ and $E_2$ have initial vertex in $H_0$ and terminal vertex in $H_1$. The argument in this cases is a variation of the one used in the three preceding cases. For $j=0,1$ let $\tau_j$ be a loop based at the unique vertex of $H_j$. Let $\sigma_0 = \tau_0$, let $\sigma_1 = E_1 \tau_1 \overline E_1$ and let $\sigma _2 = E_1 \overline E_2$. 
Let $a,b,c \in F_n$ be determined by $\sigma_0,\sigma_1,\sigma_2$ as above, and so $\<a,b,c\>$ is a rank~3 free factor and we have $[\<a,b\>] \sqsubset [H \cup E_1]$ and $[\< a,\bar c b c\>] \sqsubset [H \cup E_2 ]$. 

\break

Choose $m$ so that $[a b^m],[\bar c b^m]$ and $ [a^m \bar c \bar b c]$ are not in $\A_{\na}\Lambda^\pm$. Consider the rank two free factors 
$$\<a,b\> \qquad \<ab^m,\bar cb^m\> \qquad \<\bar cb^m,a^m\bar c \bar b c\> \qquad \<a,\bar c b c\>$$ 
whose consecutive intersections contain 
$$ab^m \quad\qquad\qquad \bar cb^m \quad\qquad\qquad a^m \bar c \bar b c
$$
respectively. Three applications of Lemma~\ref{WellLemma}~\pref{item:same free factor} show that $\abs{W(S_2) -W(S_1)} \le 6M$. 
\end{proof} 

The following corollary puts the pieces together to finish the proof of Theorem~\ref{trichotomy}~(1):
 
\begin{cor} \label{loxodromic} Suppose that $\phi \in\Out(F_n)$. If some $\Lambda^+ \in \L(\phi)$ fills then the action of $\phi$ on $\fscn$ is loxodromic.
\end{cor}

\begin{proof} After replacing $\phi$ by an iterate we may assume that $\phi$ and $\phi^{-1}$ are rotationless. For any one-edge splitting $S$, we have $W(S^{\phi^m}) = W(S) -m$ by Lemma~\ref{BigW} \pref{BigWonS}. Lemma~\ref{adjacent splittings} therefore implies that the distance between $S^{\phi^m}$ and $S$ in $\fscn$ grows linearly in $m$, which completes the proof of the corollary.
\end{proof}

\begin{remark} 
\label{RemarkFSLox}
As said in the introduction, it is known from \cite{BestvinaFeighn:FFCHyp} Theorem 9.3 that 
$\phi \in \Out(F_n)$ acts loxodromically on the free factor complex if and only if $\phi$ is fully irreducible. The ``if'' direction follows from the same method of proof as in Corollary~\ref{loxodromic}, but with a shorter argument using Lemma~\ref{bounded by 2M} in place of Lemma~\ref{adjacent splittings} and using that every nonfilling conjugacy class is not carried by $\A_\na(\Lambda^\pm_\phi)$, which is either empty or consists of a single, filling, rank~$1$ component depending on whether $\Lambda^\pm_\phi$ is a geometric lamination  pair.
\end{remark}

\subsection{Axes with distinct ends: a part of Theorem~\ref{disjoint axes}.}
\label{SectionDistinctEnds}

The following statement, `part' of the proof of Theorem~\ref{disjoint axes}, is proved here since it is a corollary to the methods of Section~\ref{SectionLoxChar}.

\begin{cor} \label{distinct ends} Given $\phi, \psi \in \Out(F_n)$ and filling lamination pairs $\Lambda_\phi^\pm \in \L^\pm(\phi)$ and $\Lambda_\psi^\pm \in \L^\pm(\psi)$, if $\Lambda_\phi^+ \ne \Lambda_\psi^+$ then $\bdy_- \phi \ne \bdy_- \psi$.
\end{cor}

\begin{proof} Assume that $\Lambda_\phi^+ \ne \Lambda_\psi^+$. It suffices to show that for any one-edge free splitting $T = \<G,H\>$ with corresponding marked graph pair $(G,H)$, and for all positive constants $D$, there exists a positive constant $L$ so that the distance in $\fscn$ between $T^{\phi^{-k}}$ and $T^{\psi^{-l}}$ is $\ge D$ for all $k,l>L$. 
 
Assume the notation of Lemma~\ref{WellLemma} applied to $\phi$. Thus $U^\pm_\phi$ are attracting neighborhoods of generic leaves $\gamma^\pm_\phi$ of $\Lambda^\pm_\phi$, $M_\phi$ is a positive constant and $w_{+,\phi}$ is a function defined on conjugacy classes not carried by the non-attracting subgroup system $\A_\phi$ associated to $\Lambda^\pm_\phi$. Let $W_\phi(S)=W_\phi(\F(S))$ be the corresponding function defined on one edge splittings and on the corresponding coedge~1 free factor systems, given in Definition~\ref{W(F)}. 
 

We will prove that there is a upper bound for $W_\phi(\psi^l(\F(T))$ that is independent of $l \ge 0$. To see why this suffices, note that $\F(T^{\phi^{-k}}) = \phi^k(\F(T))$ and that $\F(T^{\psi^{-l}}) = \psi^l(\F(T))$ by Lemma~\ref{BigW} \pref{BigWonS}. 
Since $W_\phi(\phi^k(\F(T)) = W_\phi(\F(T))+k$, we have that 
$$\abs{W_\phi(T^{\phi^{-k}})-W_\phi(T^{\psi^{-l}})} = \abs{W_\phi(\phi^k(\F(T))) - W_\phi(\psi^l(\F(T)))} \to \infty \quad\text{as $l,k \to \infty$}
$$ 
and so Lemma~\ref{adjacent splittings} completes the proof.
 
 Lemma~\ref{small intersection} implies that $B_1 = \{\gamma_\phi^+\}$ and $B_2 = \{\gamma_\psi^+\}$ satisfy the hypotheses of Proposition~3.1 of \SubgroupsThree. The conclusion of that proposition is the property ``separation of proper free factor systems'' from Definition~\ref{w(c)}, namely, the existence of neighborhoods $V_\phi^+$ of $\gamma_\phi^+$ and $V_\psi^+$ of $\gamma_\psi^+$ so that no proper free factor system carries both a conjugacy class carried by $V_\phi^+$ and a conjugacy class carried by $V_\psi^+$. Choose $N_\phi > 0$ so that $\phi^{N_\phi}(U^+_\phi) \subset V_\phi^+ $. 
 
 By Lemma~\ref{not in ANA}~\pref{item: one edge case}, $\F(T)$ carries a conjugacy class $[a] $ that is not carried by the non-attracting subgroup system $\A_\psi$ associated to $\Lambda_\psi^\pm$. Choose $N_\psi > 0$ so that $\psi^l([a])$ is carried by $V^+_\psi$ for all $l \ge N_\psi$. Since $\psi^{l}(\F(T))$ carries $\psi^l([a])$, it can not carry any conjugacy class that is carried by $V_\phi^+$. It follows that $\phi^{-N_\phi}\psi^l(\F(T))$ does not carry any conjugacy class that is carried by $U^+_\phi$ and so, by Lemma~\ref{WellLemma} \pref{WellAndZero}, we have $w_+(c) < 0$ for all $c$ carried by $\phi^{-N_\phi}\psi^l(\F(T))$. Thus $W_\phi(\phi^{-N_\phi}\psi^l(\F(T))) \le 0$ and $W_\phi(\psi^l(\F(T))) \le N_\phi$ for all $l \ge N_\psi$ by Lemma~\ref{BigW} \pref{BigWonS}.
\end{proof}

\section{Theorem~\ref{ThmFinitelyGenerated}: The expansion factor kernel $K$}
\label{SectionExpansionKernel}
\subsection{Setup and outline of the proof.}
\label{SectionSetupOutline}
We first recall the statement.

\vspace{,1in}

\noindent {{\bf Theorem~\ref{ThmFinitelyGenerated}.} \emph{ Suppose that $\eta \in \IAThree$ is rotationless, that $\Lambda_\eta^+ \in \L(\eta)$ is filling and that $K$ is the kernel of $\PF= \PF_{\Lambda^+} : \Stab(\Lambda_\eta^+) \cap \IAThree \to \R$. Then there exist compact surfaces $S_1,\ldots, S_m$ with nonempty boundary and a homomorphism 
$$\Theta \from K \to \mcg(S_1) \times \ldots \times \mcg(S_m)
$$ 
whose image has finite index and whose kernel is a finitely generated, abelian group of linearly growing outer automorphisms. In particular, $K$ is finitely generated.}
 \vspace{,1in}

The proof of the theorem will show that the surfaces $S_1,\ldots,S_m$ arise from $m$ different ``simultaneous geometric models'' for the subgroup $K$, as expressed precisely in Lemma~\ref{LemmaMultiedge}. When the proof is complete, and for wider application, we shall collect some additional conclusions regarding these geometric models; see Proposition~\ref{CorGeomModelFeatures}. 

Here is a brief outline of the proof, by subsections:
\begin{description}
\item[Section \ref{SectionKAttrLams}:] A study of the attracting laminations of elements of $\K$, proving amongst other things that all such laminations are geometric.
\item[Section~\ref{SectionUPG}:] A study of UPG subgroups of $K$, proving that each is finitely generated.
\item[Section~\ref{SectionOrderingAttrLams}:] A study of a natural partial ordering on the set of attracting laminations of elements of $K$.
\item[Section~\ref{SectionGeomModels}:] The statement and proof Lemma~\ref{LemmaMultiedge} wherein the homomorphism $\Theta$ is defined. Also included is a review of geometric models.
\item[Section~\ref{SectionKernelProof}:] The proof of Theorem~\ref{ThmFinitelyGenerated}.
\end{description} 

\paragraph{Standing assumptions of Section~\ref{SectionExpansionKernel}:} Throughout this section we assume the hypotheses and notation of Theorem~\ref{ThmFinitelyGenerated}, including:

\smallskip\noindent
$\bullet$ The rotationless outer automorphism $\eta \in \IAThree$.

\smallskip\noindent
$\bullet$ The filling lamination $\Lambda^+_\eta$.

\smallskip\noindent
$\bullet$ The subgroup $\Gamma = \Stab(\Lambda_\eta^+) \cap \IAThree$ and its normal subgroup $K$.

\smallskip\noindent
$\bullet$ The proper free factor system $\F_\agen $ defined in Section~\ref{SectionBackground} to be the smallest free factor system that carries the ageneric sublamination of $\Lambda^+_\eta$. Each element $\phi \in \Gamma$ preserves $\Lambda^+_\eta$ and hence preserves both its ageneric sublamination and $\F_\agen $.

\subsection{Attracting laminations of $K$} 
\label{SectionKAttrLams}
The next result says that every attracting lamination of every element of $K$ is a geometric lamination, and gives an explicit description of the behavior of generic leaves of $\Lambda^+_\eta$ in \ct s representing elements of $K$.


\begin{lemma} \label{LemmaFIX} 
Under the ``Standing assumptions of Section~\ref{SectionExpansionKernel}'', consider a rotationless $\phi \in K$ and a relative train track representative $f \from G \to G$ of $\phi$ in which $\F_\agen$ is realized by proper core filtration element $ G_r$. Suppose that $f$ satisfies the following properties (which evidently hold when $f$ is a~\ct): each \eg\ stratum is \eg-aperiodic, equivalently $\phi$ fixes each element of $\L(\phi)$; each non-fixed \neg\ stratum is a single oriented edge $E_i$ such that $f(E_i)=E_i u_i$ for some path $u_i$ in $G_{i-1}$; and each periodic Nielsen path has period one. In this situation the following hold:
\begin{enumerate}
\item \label{ItemHighestIsNEG}
The highest stratum of $G$ is an \neg-edge $E_N$.
\item \label{ItemStrataAboveGeometric}
Each \eg\ stratum of $G$ is geometric (equivalently each element of $\L(\phi)$ is geometric). 
\end{enumerate}
Furthermore, if $\gamma_\eta$ is the realization in $G$ of a generic leaf of $\Lambda^+_\eta$ then:
\begin{enumeratecontinue}
\item\label{ItemSplitsIntoNielsen}
$\gamma_\eta$ crosses $E_N$ bi-infinitely often, and each term of the highest edge splitting of $\gamma_\eta$ is a Nielsen path for $f$.
\item\label{ItemGammaEtaANA}
For each $\Lambda \in \L(\phi)$ the non-attracting subgroup system $\A_\na(\Lambda)$ of $\Lambda$ carries $\gamma_\eta$ and hence fills~$F_n$. 
\end{enumeratecontinue}
\end{lemma}

\begin{proof} We first prove~\pref{ItemHighestIsNEG}, that the highest stratum $H_N$ of $G$ is \neg. If not then $H_N$ is an \hbox{\eg-aperiodic} stratum with associated lamination pair $\Lambda^\pm_\phi$, and we shall argue to a contradiction. Proposition~3.3.3 of \BookOne\ and the assumption that $\PF_{\Lambda^+_\eta}(\phi) = 0$ imply that $\Lambda^+_\eta$ is neither an element of $\L(\phi)$ nor of $\L(\phi^\inv)$. It follows that $\Lambda^+_\phi$ and $\Lambda^-_\phi$ are not equal to $\Lambda^+_\eta$. Nor are they sublaminations of $\Lambda^+_\eta$, because $\Lambda^+_\phi$ and $\Lambda^-_\phi$ are not carried by $\F_\agen = [G_r]$ and so they are not contained in the ageneric sublamination of $\Lambda^+_\eta$ which is the unique maximal proper sublamination. Since~$\Lambda^+_\eta$ is invariant by $\phi^{\pm 1}$, the sequence $\phi^i (\gamma_\eta)$ is contained in the weakly closed set $\Lambda^+_\eta$, and hence $\Lambda^+_\eta$ contains every line to which $\gamma_\eta$ is weakly attracted under iteration of $\phi^{\pm 1}$. It follows that $\gamma_\eta$ is not weakly attracted to $\Lambda^\pm_\phi$ under iteration of $\phi^{\pm 1}$. Theorem~H of \SubgroupsThree\ implies that $\gamma_\eta$ is carried by $\A_\na(\Lambda^\pm_\phi)$. But this is impossible: if $\Lambda^\pm_\phi$ is non-geometric then $\A_\na(\Lambda^\pm_\phi)$ is a proper free factor system, whereas if $\Lambda^\pm_\phi$ is geometric then $\A_\na(\Lambda^\pm_\phi)$ consists of a proper free factor system that carries $\F$ plus a single rank~$1$ component. But $\gamma_\eta$ is not periodic, and $\gamma_\eta$ fills and so is not carried by any proper free factor system. This contradiction proves~\pref{ItemHighestIsNEG}. 

We next prove~\pref{ItemSplitsIntoNielsen}. Since the line $\gamma_\eta$ is filling it crosses $E_N$ at least once, and since $\gamma_\eta$ is birecurrent it crosses $E_N$ bi-infinitely often. Applying the basic splitting property for \neg\ edges with respect to $f$, one obtains the highest edge splitting of $\gamma_\eta$ by subdividing it at the initial vertex of each copy of $E_N$ that it crosses and at the terminal vertex of each copy of $\overline E_N$ that it crosses, and this is a bi-infinite splitting of the form $\gamma_\eta = \ldots \cdot \gamma_{-1} \cdot \gamma_0 \cdot \gamma_1 \cdot \ldots$ Again by birecurrence, each term of this splitting is repeated bi-infinitely. For each $i$, the endpoints of $\gamma_i$ are fixed and so either $\gamma_i$ is a periodic and hence fixed Nielsen path for $f$, or the combinatorial length $\abs{f_\#^k(\gamma_i)} \to \infty$ as $k \to \infty$. Assuming that some, and hence infinitely many, $\gamma_i$'s are not Nielsen paths of $f$, we shall argue to a contradiction. 
 
Choose integers $a < b$ so that neither $\gamma_a$ nor $\gamma_b$ is a Nielsen path and such that for some $a < i < b$ the term $\gamma_i$ crosses $E_N$ (in either direction). Then $f^k_\#(\gamma_a \cdot \ldots\cdot \gamma_b) = \alpha_k E_N \beta_k$ or $\alpha_k \overline E_N \beta_k$ where:
\begin{itemize}
\item[(a)] $\abs{\alpha_k}, \abs{\beta_k} \to \infty$.
\item[(b)] $\alpha_k$ and $\beta_k$ cross $E_N$ (in either direction) a uniformly bounded number of times.
\end{itemize} 
By taking a weak limit of the sequence $f^k_\#(\gamma_a \cdot \ldots\cdot \gamma_b)$ based at the $E_N$ or $\overline E_N$ subpath that is shown, we find a leaf of $\Lambda^+_\eta$ that crosses $E_N$ (in either direction) a finite non-zero number of times. Such a line would be ageneric but not contained in $G_r \subset G_{N-1}$, contradicting that the ageneric sublamination of $\Lambda^+_\eta$ is supported by $\F_\agen = [G_r]$. This completes the proof of~\pref{ItemSplitsIntoNielsen}. 

Suppose now that $\Lambda \in \L(\phi)$ and that $\Lambda' \in \L(\phi^{-1})$ are a dual pair of laminations. Item~\pref{ItemHighestIsNEG} implies that $\F_{\supp}( \Lambda) =\F_{\supp} (\Lambda')$ does not carry $\gamma_\eta$. In particular neither $\Lambda$ nor $\Lambda'$ contains $\gamma_\eta$. Item~\pref{ItemSplitsIntoNielsen} implies that $\gamma_\eta$ is fixed by $f_\#$ and is therefore not weakly attracted to either $\Lambda$ or $\Lambda'$. As above, Theorem~H of \SubgroupsThree\ implies that $\gamma_\eta$ is carried by the non-attracting subgroup system of $\Lambda$, proving~\pref{ItemGammaEtaANA}. 

For each $\Lambda \in \L(\phi)$, $\A_\na\Lambda$ is proper because it does not carry the generic leaves of~$\Lambda$. Item~\pref{ItemGammaEtaANA} therefore implies that $\A_\na\Lambda$ is not a free factor system. It follows by \SubgroupsThree, Theorem~F, that $\Lambda$ is geometric, proving~\pref{ItemStrataAboveGeometric}. 
\end{proof}

\subsection{UPG subgroups of $K$} 
\label{SectionUPG} 

The following proposition, which gives a piece of the conclusion of Theorem~\ref{ThmFinitelyGenerated}, is the main result of this subsection. Its proof appears at the end of the subsection. 

We continue to assume the hypotheses of Theorem~\ref{ThmFinitelyGenerated} and the notational setup from the end of Section~\ref{SectionSetupOutline}. 

\begin{prop} \label{PropUPG} Every $\upg$ subgroup $\cH \subgroup K$ is abelian, linear and finitely generated.
\end{prop}

We review some central concepts of \BookTwo. Given a marked graph $G$, \emph{a filtration with one-edge oriented strata} is a filtration $\filt$ in which each stratum $G_i \setminus G_{i-1}$ is a single oriented edge $E_i$. Given such a filtration, a homotopy equivalence $f \from G \to G$ is \emph{upper triangular} if for each $i$ we have $f(E_i) = E_i u_i$ where either $u_i=u_i(f)$ is a nontrivial closed path in $G_{i-1}$, or $u_i$ is trivial and $f \restrict E_i$ is the identity. 
Note that $f$ fixes each vertex of $G$ and is a relative train track map. Let~${\cQ}(G)$ be the set of upper triangular homotopy equivalences of $G$ up to homotopy relative to the vertices of $G$. The operation of composition descends to a group structure on $\cQ(G)$ (Lemma 6.1 of \BookTwo). There is a natural homomorphism $\cQ(G) \mapsto \Out(F_n)$.

\begin{ex} \label{linear example} Here is a special case of an outer automorphism $\phi$ as considered in Example~\ref{filling reducible}. Continuing the notation of that example, let $G_1 \subset G$ be roses of rank $m=3$ and rank $5$, respectively. Let $X,Y,Z$ be the edges of $G_1$ and $A,B$ the edges of $H_2=G \setminus G_1$. Fix a nontrivial word $w$ in $\<X,YX \bar Y\>$, and let $\cH'$ be the subgroup of $\cQ(G)$ whose elements have the form
$$X \mapsto X \qquad Y \mapsto YX^{3i} \qquad Z \mapsto Zw^{3j}\qquad A \mapsto A \qquad B \mapsto B 
$$
Then $\cH'$ is a rank two abelian linear subgroup and every word in $ \<X, YX\bar Y, Zw\bar Z\>$ is fixed by every element of $\cH'$. The smallest free factor that contains $ \<X, YX\bar Y\>$ equals $\<X,Y\>$ because it is contained in $\<X, Y\>$ and properly contains $\<X\>$. Similarly the smallest free factor that contains $ \<X, YX\bar Y, Zw\bar Z\>$ equals $\<X,Y,Z\>$. Choose a word $\sigma \in \<X, YX\bar Y, Zw\bar Z\>$ that fills $\<X, YX\bar Y, Zw\bar Z\>$ and hence fills $\<X,Y,Z\>$. Using this $\sigma$, let $\phi \in \Out(F_5)$ and $\Lambda \in \L(\phi)$ be as in Example~\ref{filling reducible}. Then $\cH'$ injects into $\Out(F_5)$ producing a rank two linear subgroup $\cH$ of $K$. 
\end{ex} 

\begin{definition}[Weak Filtration]
\label{DefWeakFiltration}
A filtration $\filt $ with one-edge oriented strata satisfies the \emph{(Weak Filtration)} property if for each pair $G_j \subset G_{j'}$ of consecutive core filtration elements one of the following occurs:
\begin{description}
\item [(a)] 
$j' = j+1$ and $E_{j+1}$ forms a loop component of $G_{j'}$ that is disjoint from $G_{j}$;
\item [(b)]
$j'=j+2$, the edges $E_{j+1}$, $E_{j+2}$ have the same initial endpoint not in $G_j$, they have terminal endpoints in $G_j$, and $G_{j+2}$ is obtained by attaching the endpoints of the topological arc $\overline E_{j+1} E_{j+2}$ to~$G_{j}$.
\end{description}
In case (b) we let $G_{j,\epsilon}$ be the component of the core graph $G_j$ that contains the terminal endpoint $v_{j,\epsilon}$ of $E_{j + \epsilon}$ for $\epsilon = 1,2$.
\end{definition}

Rather than work in $\cQ(G)$ itself, we will work in the subgroup $\cQ'(G)$ of $\cQ(G)$ defined as follows.

\begin{definition}[Principal Endpoints]
\label{DefOneEdgePrincipal}
Suppose that $\filt$ is a filtration with one-edge oriented strata that satisfies (Weak Filtration). Let $\E$ be the set of all $E_i \subset G$ such that the terminal endpoint of $E_i$ is contained in a circle component of $G_{i-1}$, in which case $G_{i-1} \subset G_{i+1}$ are consecutive core filtration elements falling into case (b) of Definition~\ref{DefWeakFiltration}. An element $\fG$ of $\cQ(G)$ satisfies the \emph{(Principal Endpoints)} property if for all $E_i \in \E$ we have $f(E_i) = E_i$. The set of elements of $\cQ$ satisfying (Principal Endpoints) forms a subgroup $\cQ'(G)$.
\end{definition}

The advantage of $\cQ'(G)$ is illustrated by the following lemma: the injectivity condition fails for $\cQ(G)$.

\begin{lemma} \label{uniqueness} Suppose that $\emptyset = G_0 \subset G_1 \subset\cdots\subset G_J = G$ is a filtered graph with one-edge oriented strata that satisfies (Weak Filtration). Then the natural homomorphism $\cQ'(G) \to \Out(F_n)$ is injective. 
\end{lemma}

\begin{proof} It suffices to prove that if $f \in \cQ'(G)$ and $f \restrict G_j$ is not the identity for some core filtration element $G_j$ then there is a circuit $\sigma \subset G_j$ whose corresponding conjugacy class in $F_n$ is not $\phi$-invariant. We prove this by induction on $j$ with the $j =1$ case following from the fact that $f \restrict G_1$ is the identity for all $f \in \cQ'(G)$. Assuming that $G_j \subset G_{j'}$ are consecutive core filtration elements and that $f \restrict G_j$ is the identity but $f \restrict G_{j'}$ is not the identity we will produce the desired $\sigma$. Adopting the notation of (Weak Filtration), we must be in case (b) where $j' = j+2$. 

Suppose at first that $E_{j+1}$, is not fixed by $f$. By (Principal Endpoints), $G_{j,1}$ has rank at least two. Choose $\alpha_1 \subset G_{j,1}$ to be a closed path based at $v_{j,1}$ so that the elements of $\pi_1(G_{j,1},v_{j,1})$ determined by $\alpha_1$ and $u_{j+1}$ do not commute. If $u_{j+2}$ is non-trivial, let $\alpha_2 = u_{j+2}$; otherwise let $\alpha_2 \subset G_{j,2}$ be any non-trivial closed path based at the terminal vertex of $E_{j+2}$. Define $\sigma = E_{j+1} \alpha_1 \bar E_{j+1} E_{j+2} \alpha_2\bar E_{j+2}$.  Letting $[\gamma]$ denote the unique path obtained by straightening a continuous function $\gamma \from [0,1] \to G$ with endpoints at vertices, note that $f_\#(\sigma) = E_{j+1} [u_{j+1} \alpha_1 \bar u_{j+1}]\bar E_{j+1} E_{j+2} \alpha_2\bar E_{j+2}$ and that $[u_{j+1} \alpha_1 \bar u_{j+1}] \ne \alpha_1$, by choice of $\alpha_1$. Thus $\sigma$ and $f_\#(\sigma)$ are distinct circuits and so they represent distinct conjugacy classes in $F_n$. 

The remaining case is that $E_{j+1}$ is fixed and that $E_{j+2}$ is not fixed. We can proceed as in the previous case, reversing the roles of $E_{j+1}$ and $E_{j+2}$, because $G_{j,2}$ will have rank at least $2$, unless the component $L$ of $G_{j+1}$ that contains $v_{j,2}$ is not a topological circle but its core $G_{j,2}$ is a topological circle. In this special case we must have $L = G_{j,2} \union E_{j+1}$, and so $v_{j,1} = v_{j,2}$ because $v_{j,2}$ is the unique vertex in $G_{j,2}$. The circuit $\sigma = E_{j+2} \bar E_{j+1}$ is therefore defined. As in the previous case, the circuit $f_\#(\sigma) = E_{j+2} u_{j+2}\bar E_{j+1}$ is distinct from $\sigma$.
\end{proof}

The following proposition, which makes no use of $K$, is essentially a restatement of the main result of \BookTwo, namely Theorem 1.1, the ``Kolchin Theorem for $\Out(F_n)$'', regarding the structure of finitely generated UPG subgroups of $\Out(F_n)$. We have added to the statement some details of the proof of that theorem found on the last page of \BookTwo, as well as some details from earlier lemmas and propositions in \BookOne\ and \BookTwo, and we have employed the new language of Definitions~\ref{DefWeakFiltration} and~\ref{DefOneEdgePrincipal}. The proof of the proposition tracks closely the proof of Theorem 1.1, but we have attempted to write a (mostly) self-contained proof, subject to citations of stated results in those papers.


\begin{proposition}\label{PropKolchinGraph} 
If $\cH \subgroup \Out(F_n)$ 
is any finitely generated \upg\ subgroup, and if $\emptyset = \F_0 \sqsubset \F_1 \sqsubset \cdots \sqsubset \F_L = \{[F_n]\}$ is any maximal $\cH$-invariant filtration by free factor systems, then there exists a filtered marked graph $\emptyset = G_0 \subset G_1 \subset\cdots\subset G_J = G$ with one-edge oriented strata and there exists an isomorphic lift of $\cH$ to a subgroup $\cH'$ of ${\cQ}'(G)$, such that the following hold:
\begin{enumerate}
\item\label{ItemEachIsACore}
Each $\F_\ell$ is realized by some core filtration element~$G_{j_\ell}$;
\item \label{ItemWeakFiltration}$\emptyset = G_0 \subset G_1 \subset\cdots\subset G_J = G$ satisfies (Weak Filtration);
\end{enumerate}
and in addition for each $\phi$ in $\cH$ represented by its lift $f \from G \to G$ in $\cH'$ the following hold:
\begin{enumeratecontinue}
\item\label{ItemCHPrimeRepRotationless}
$f$ is a rotationless relative train track map and every vertex of $G$ is a principal fixed point.
\item \label{ItemPeriodOne} 
Every periodic Nielsen path for $f$ with endpoints at vertices has period one.
\end{enumeratecontinue}
\end{proposition}

\begin{proof} Each extension $\F_{\ell-1} \sqsubset \F_\ell$ is a one-edge extension, by Theorem~5.1 of \BookTwo\ applied inductively to the unique component $[A]$ of $\F_\ell$ that is not a component of $\F_{\ell-1}$. The inductive construction of a filtered marked graph $G$ with one-edge oriented strata which satisfies~\pref{ItemEachIsACore} and~\pref{ItemWeakFiltration} and the existence $f \in \cQ(G)$ representing $\phi$ is straightforward; see for example \BookOne\ Lemma~2.6.7. Henceforth we assume~\pref{ItemEachIsACore} and~\pref{ItemWeakFiltration} hold.

To arrange that $f \in \cQ'(G)$ we may have to modify $\fG$ so that it satisfies (Principal Endpoints). 
We carry out this modification by induction on the height $i$ of the lowest edge $E_i \subset G$ that fails to satisfy (Principal Endpoints). We have $E_i \in \E$ and $f(E_i) \ne E_i$. By Definition~\ref{DefOneEdgePrincipal} we have core filtration elements $G_{i-1} \subset G_{i+1}$ satisfying case (b) of (Weak Filtration). Also, the edge $E_i$ has initial vertex $v \not\in G_{i-1}$ and terminal vertex $v$ in a circle component of $G_{i-1}$ consisting of a single edge $E$ with both endpoints on $v$, and furthermore $f(E_i) = E_i E^d$ for some $d \ne 0$.
%
%
Let ${\cal X}$ be the set of edges in $G \setminus E$ that are incident to $v$, so $E_i$ is the lowest edge in $\cal X$. By (Weak Filtration), $v$ is the terminal endpoint of each $X \in {\cal X}$. 
Define a homotopy equivalence $g : G \to G$ to be the identity on edges not in $ {\cal X}$ and by $g(X) = X \overline E^d$ for each $X \in {\cal X}$. Every closed edge path based at $w$ decomposes as a concatenation of edges not in ${\cal X} \union \{E\}$ and subpaths of the form $X_1 E^p \bar X_2$ for some $X_1,X_2 \in {\cal X}$ and some $p \in \Z$. It follows that $g_\#$ fixes each such path. Thus $g$ is homotopic to the identity and represents the identity element of $\Out(F_n)$. Replace $f$ with the element of $\cQ(G)$ obtained by tightening $g \circ f  \from G \to G$. It is still the case that $f$ represents $\phi$, but now we have raised the height of the lowest edge that fails to satisfy (Principal Endpoints). 
Continuing by induction, we eventually produce $f \in \cQ'(G)$ representing $\phi$. Lemma~\ref{uniqueness} implies that such an $f$ is unique. It follows that $\phi \mapsto f$ defines an isomorphism from $\cH$ to a subgroup $\cH'$ of $\cQ'(G)$. 

Since $f$ has no \eg\ strata, it is a relative train track map. The property (Weak Filtration) implies that $f$ fixes every vertex $v \in G$ and that every $f$-periodic direction at $v$ is fixed, so $f$ is rotationless. To complete the proof of~\pref{ItemCHPrimeRepRotationless} it suffices to show that each vertex $v \in G$ is principal. Let $D(v)$ be the number of fixed directions based at $v$. Since there are no \eg\ strata, it suffices to show that if a component $C$ of $\Fix(f) = \Per(f)$ is a circle then $C$ contains a vertex $v$ with $D(v) \ge 3$. Suppose that $G_j \subset G_{j'}$ are the consecutive core filtration elements such that $G_{j'}$ is the lowest core filtration element that contains a vertex $v \in C$. If $G_j \subset G_{j'}$ falls under case (a) of (Weak Filtration) then $v$ is in the circle component of $G_{j'}$ described in (a), and so there are two fixed directions in $G_{j'}$ based at $v$. Furthermore, using that $v$ has valence $\ge 3$ in $G$, by (Principal Endpoints) it follows that there is at least one more fixed direction at $v$ in $G \setminus G_{j'}$. Thus $D(v) \ge 3$ and we are done. Otherwise $G_j \subset G_{j'}$ falls under case (b) of (Weak Filtration) and the directions determined by $E_{j+1}$ and $E_{j+2}$ are incident to $v$ and are fixed. If there is another fixed direction based at $v$ then we are done. If not then $E_{j+1} \subset C \subset \Fix(f)$ and the direction determined by $\overline E_{j+1}$ at the terminal endpoint $w$ of $E_{j+1}$ is fixed. But $w$ also has two fixed directions in $G_j$ by (Weak Filtration) so $D(w) \ge 3$ and we are done in this case as well. This completes the proof of~\pref{ItemCHPrimeRepRotationless}.

The fact that $f$ satisfies item~\pref{ItemPeriodOne} would follow from the conclusion of \recognition\ Lemma 3.28 which says that every periodic Nielsen path of $f$ \emph{with principal endpoints} has period~$1$, because every endpoint is principal by~\pref{ItemCHPrimeRepRotationless}. But to justify applying \recognition\ Lemma 3.28 we must verify its two hypotheses. One hypothesis says $f$ is rotationless which is true by~\pref{ItemCHPrimeRepRotationless}. The other hypothesis is that $f$ satisfies the properties in the conclusion of \recognition\ Theorem 2.19 (a construction of nicely behaved relative train track representatives). The only Theorem 2.19 conclusion that might fail is conclusion (P) which prohibits \lq extraneous\rq\ periodic strata~$H_l$ by requiring that for each such $H_l$ there exists a filtration element $G_p$ such that for each filtration element $G_q$ we have $[G_p] \ne [G_q \union H_l]$. This property fails precisely in the setting of (Principal Endpoints) where $j=j_{\ell-1} < j_\ell=j+2$ and one or both of the edges $E_{j+1}$, $E_{j+2}$ is fixed. If this happens, for each such value of $\ell$ pick just one of $E_{j+1},E_{j+2}$ which is fixed and collapse it, producing a quotient map $q \from G \to G'$ which is a homotopy equivalence, inducing a marked graph structure on $G'$. There is a unique homotopy equivalence $f'' \from G' \to G'$ such that $q \composed f = f'' \composed q$, and let $f' \from G' \to G'$ be obtained from $f''$ by straightening edges ($f'$ is obtained directly from $f$ by ``collapsing an invariant forest'' as described on page 7 of \BH). This new homotopy equivalence $f' \from G' \to G'$ still represents $\phi$ and it satisfies the conclusions of \recognition\ Theorem 2.19. Applying \recognition\ Lemma 3.28, each periodic Nielsen path of $f'$ has period one. But the $q$-image in $G'$ of a periodic Nielsen path for $f$ with period greater than one is a periodic Nielsen path for $f'$ with period greater than one, so there are no such periodic Nielsen paths which proves~\pref{ItemPeriodOne}.
\end{proof}

Consider a subgroup $\cH' \subgroup \cQ'(G)$. We say that an edge or a point in $G$ is \emph{universally fixed} if it is fixed by each $f \in \cH'$, and that a path in $G$ with endpoints at vertices is \emph{universally Nielsen} if it is a Nielsen path for each $f \in \cH'$. A universal Nielsen path is an \emph{indivisible universal Nielsen path} if it cannot be written as a concatenation of two non-trivial universal Nielsen paths. Note that an indivisible universal Nielsen path need not be an \iNp\ for each $f \in \cH'$. 

Our next lemma gives additional conclusions for Proposition~\ref{PropKolchinGraph}, in the presence of additional hypotheses which will arise from application of Lemma~\ref{LemmaFIX}:

\begin{lemma} \label{LemmaGoodKolchinGraph} 
Continuing with the hypotheses and notations of Proposition~\ref{PropKolchinGraph}, suppose in addition that the set of universal Nielsen paths crosses every edge of $G$. Then the filtered marked graph $G$ and the subgroup $\cH' \subgroup \cQ'(G)$ may be altered so that the conclusions of Proposition~\ref{PropKolchinGraph} continue to hold and the following additional conclusion holds: 
\begin{enumeratecontinue}
\item \label{ItemNEGNielsenPathsUniv} For each edge $E_j \subset G$ one of the following holds:
\begin{enumerate}
\item \label{ItemEjUnivFixed}
$E_j$ is universally fixed; or 
\item \label{ItemEjUnivNP}
There is a universal closed Nielsen path $w_j$ in $G_{j-1}$ such that for each $f \in \cH'$ we have $f(E_j) = E_j w_j^{d_j(f)}$ where $d_j(f) \in \Z$. Moreover, the paths $\{E_j w_j^p \bar E_j \suchthat p \ne 0\}$ are the only indivisible universal Nielsen paths of height $j$.
\end{enumerate}
\end{enumeratecontinue}
\end{lemma}

\begin{proof} 
We shall modify $G$ and $\cH'$ by downward induction on $j$ to arrange \pref{ItemNEGNielsenPathsUniv} and maintain the property that the set $\C$ of universal Nielsen paths covers $G$.

If $\gamma \in \C$ has height $j$ then each term of the highest edge splitting of $\gamma$ is an element of $\C$ and is either a basic path of height $j$ or a path of height strictly less than $j$. In particular, any indivisible element of $\C$ of height $j$ is a basic path of height $j$. 

The maximality hypothesis on $\emptyset = \F_0 \sqsubset \F_1 \sqsubset \cdots \sqsubset \F_k = \{[F_n]\}$ implies a weak version of the conclusion, namely that for each edge $E_j$, if $\alpha \in \C$ is a closed indivisible universal Nielsen path of height $j$ then either $\alpha = E_j $ or $\alpha = E_j \mu \bar E_j$ for some nontrivial path $\mu$ in $G_{j-1}$. To see why, as noted above $\alpha$ is a basic path of height~$j$, and so we need only rule out the case that $\alpha$ (or $\alpha^\inv$) has the form $E_j \mu$ where $\mu$ is nontrivial in $G_{j-1}$.  Proposition~\ref{PropKolchinGraph}~\pref{ItemWeakFiltration} implies that $G_{j-2} \subset G_j$ are consecutive core graphs   and so by maximality realize   $\F_{l-1} \sqsubset \F_l$ for some $l$.   There is a basis element $a \in F_n$ such that $[a]$ is the conjugacy class determined by $\alpha$ and such that  $\F_{\ell-1} \cup \{[a]\}$ is a free factor system that is contained in $\F_l$  and properly contains $\F_{\ell-1}$.  But then maximality implies that $\F_\ell = \F_{\ell-1} \cup \{[a]\}$ which contradicts the fact that we are in  case (b) of Definition~\ref{DefWeakFiltration}.

We now turn to the downward induction argument: fixing $i$, and assuming that \pref{ItemNEGNielsenPathsUniv} holds for $j > i$ and that $\C$ crosses every edge of $G$ (which is obvious if $i$ is the maximal height), we either show that $E_i$ already satisfies \pref{ItemNEGNielsenPathsUniv}, or we adjust the construction until it is satisfied, maintaining the fact that $\C$ covers $G$. 
Let $x$ and $y$ be the initial and terminal endpoints of $E_i$. 

We may assume that that $y \in G_{i-1}$, because if not then by Proposition~\ref{PropKolchinGraph}~\pref{ItemWeakFiltration} the edge $E_i$ forms a loop component of $G_i$ and so is universally fixed.

\medskip

\noindent\textbf{Claim 1: Some element of $\C$ has height $i$.} By the inductive hypothesis, $E_i$ is crossed by an element of $\C$ with height $k \ge i$. To verify the claim, we assume that  $k > i$ and prove that $E_i$ is crossed by some $\alpha' \in \C$ with height $< k$. There is no loss in assuming that $\alpha$ is an indivisible element of $\C$.
Applying \pref{ItemNEGNielsenPathsUniv} for height $k$, the moreover part of \pref{ItemEjUnivNP} implies that $\alpha = E_k w_k^p \bar E_k$ for some $p \ne 0$ and some closed universal Nielsen path $w_k$ of height $<k$, and so $E_i$ is crossed by $\alpha' = w_k \in \C$, completing the inductive proof of Claim 1.


\newcommand\wtchp{\wt{\cH}'}

\medskip

To prepare for Claim~2, consider the universal cover $\wt G$ of $G$ with its deck transformation action by $\pi_1(G) \approx F_n$. Choose a lift $\ti x \in \wt G$ of $x$, let $\wt E_i \subset \wt G$ be the lift of $E_i$ with initial vertex $\ti x$, and let $\ti y$ be the terminal vertex of $\wt E_i$, so $\ti y$ is a lift of $y$. Let $\Gamma_{i-1} \subset \wt G$ be the component of the full pre-image of $G_{i-1}$ that contains $\ti y$. For each $f \in \cH'$ let $\ti f : \wt G \to \wt G$ be the lift of $f$ that fixes~$\ti x$. Let $\wtchp = \{\ti f \suchthat f \in \cH'\}$. Note that $\Gamma_{i-1}$ is $\ti f$-invariant for each $\ti f \in \wtchp$. A point $\ti z \in \Gamma_{i-1}$ is said to be \emph{universally fixed} if it is fixed by each $\ti f \in \wtchp$. 

Consider a line $\ti L \subset \Gamma_{i-1}$. If $\ldots \ti \mu_i \cdot \ti \mu_{i+1} \cdot \ti \mu_{i+2}\ldots$ is the highest edge splitting of $\ti L$ then $\ldots \ti f_\#(\mu_i)\cdot \ti f_\#(\ti \mu_{i+1}) \cdot \ti f_\#(\ti \mu_{i+2})\ldots$ is the highest edge splitting of the line $\ti f_\#(\ti L) \subset \Gamma_{i-1}$. Furthermore, if $\ti f_\#$ fixes $\ti L$ (equivalently, $\ti f$ fixes both points in $\bdy \ti L$) then these two splittings are equal so there exists $\tau \in \Z$ such that $\ti f_\#(\ti \mu_i) = \ti \mu_{i + \tau}$ for all $i$; it follows that $\ti f$ preserves the sequence $\V(\ti L) = \{\ti v_i\}$ of highest edge splitting vertices where $\ti v_i = \ti \mu_i \cap \ti \mu_{i+1}$, and that $\ti f(\ti v_i) = \ti v_{i +\tau}$. Note that if $\V(\ti L)$ is not bi-infinite then $\tau = 0$. 

Let $\T$ be the subgroup of covering translations such that each non-trivial element $T \from \wt G \to \wt G$ of $\T$ commutes with each $\ti f \in \wtchp$ and has axis $A(T)$ contained in $ \Gamma_{i-1}$. Given any covering translation $T$ with nontrivial axis $A(T) \subset \Gamma_{i-1}$ we have $T \in \T$ if and only if each $\ti f_\#$ fixes each point of $\bdy A(T)$ by \BookThree\ Lemma~2.4, in which case the previous paragraph applies to $\V(T)=\V(A(T))$. In particular, if $T^i \in T$ for some $i \ne 0$ then $T \in \T$. 

\medskip
\noindent\textbf{Claim 2: Either some $\ti z \in \Gamma_{i-1}$ is universally fixed or $\T$ is infinite cyclic.} By~Claim~1, we may pick a basic path $\alpha \in \C$ with height $i$. Orient $\alpha$ so that its first edge is $E_i$. Lift $\alpha$ to a path $\ti \alpha$ with initial endpoint $\ti x$ and terminal endpoint $\ti z$ and note that $\ti z$ is fixed by each~$\ti f \in \wtchp$. If $\alpha=E_i\mu$ then $\ti z \in \Gamma_{i-1}$ and we are done. 

Henceforth we assume $\Gamma_{i-1}$ has no universal fixed point, so $\alpha = E_i \mu \overline E_i$ is closed. Let $T \from \wt G \to \wt G$ be the nontrivial covering translation that carries $\ti x$ to $\ti z$. For each $\ti f \in \wtchp$ we have $T \ti f (\ti x) = T(\ti x) = \ti z = \ti f(\ti z) = \ti f T(\ti x)$. It follows that $T$ commutes with $\ti f$, because $T \ti f$ and $\ti fT$ are lifts of $f$ that agree on a point and are hence equal. The axis $A(T)$ is contained in $\Gamma_{i-1}$ because $\mu \subset G_{i-1}$. Thus $T \in \T$, proving that $\T$ is nonempty. 

Choose $T_0 \in \T$ so that the axis $A(T_0)$ has maximal height. To prove that $\T$ is infinite cyclic, we suppose that there exists $T_1 \in \T$ with $A(T_1) \ne A(T_0)$ and argue to a contradiction. Using that no point of $\Gamma_{i-1}$ is universally fixed, choose $\ti f' \in \wt\cH'$ so that it acts as a non-trivial translation on $\V(T_0)$. Under iteration by $\ti f'$, the elements of $\V(T_0)$ converge to some $Q \in \bdy A(T_0) \subset \partial \Gamma_{i-1}$. Choose $P \in \partial A(T_1) \subset \bdy\Gamma_{i-1}$ so that the elements of $\V(T_1)$ are not moving away from $P$ under iteration by $\ti f'$. The line $\ti L$ with endpoints $P$ and $Q$ is $\ti f'_\#$-invariant so its highest edge splitting vertices $\V(\ti L)$ are preserved by $\ti f'$. Since the height of $A(T_0)$ is greater than or equal to the height of $A(T_1)$, we have $\V(T_0) \cap \ti L \subset \V(\ti L)$ and $\ti f'$ moves the elements of $\V(\ti L)$ toward $Q$. As noted above, this implies that $\V(\ti L)$ is bi-infinite and so $A(T_1)$ has the the same height as $Q(T_0)$ and the elements of $\V(T_1) \cap L \subset \V(\ti L)$ are moved away from $P$. This contradiction completes the proof of Claim~2. 
 
\medskip
 
In any case where $\T$ is infinite cyclic, let $T_0 \in \T$ denote a generator. Note that $\A(T)$ and $\V(T)$ are independent of the choice of nontrivial $T \in \T$, and let them be denoted $A(\T)$ and $\V(\T)$ respectively. Note also that the group $\T$ is the stabilizer of $A(\T)$ under the action of the deck transformation group because, as seen earlier, if $T^i \in T$ then $T \in T$. 
 
We are now able to complete the proof in two special cases, each expressing a special property of the terminal point $\ti y$ of $\wt E_i$. First, if $\ti y$ is universally fixed then $E_i$ is universally fixed and (a) is satisfied. 

The second special case is that $\ti y \in \V(T)$ and that there does not exist a universal fixed point in $\Gamma_{i-1}$. From this we show that (b) is satisfied. Let $\ti w_i$ be the path from $\ti y$ to $T_0(\ti y)$ and let $w_i$ be the projection of~$\ti w_i$, a closed path which forms a circuit $c$ and which is not an iterate of any shorter closed path. Consider $\ti f \in \wtchp$. Since $\ti f_\#(A(\T))=A(\T)$ it follows that $f_\#(c)=c$ and so the terms of the highest edge splitting of $c$ are cyclically permuted by~$f_\#$. Applying Proposition~\ref{PropKolchinGraph}~\pref{ItemPeriodOne} it follows that each of those terms is fixed by $f_\#$. This shows that $f_\#(w_i)=w_i$, from which it follows that $\ti f(\ti y) = T^{d_i(f)}(\ti y)$ for some integer $d_i(f)$, and so $f(E_i) = E_i w_i^{d_i(f)}$. 
%
%
For the moreover part of (b), suppose that $\alpha \in \C$ has height $i$ and is indivisible. Let $\ti \alpha$ be the lift of $\alpha$ with initial endpoint $\ti x$. As noted above, $\alpha$ is a basic path of height $i$. If $\alpha = E_i \mu$ for some $\mu \subset G_{i-1}$ then the terminal endpoint $\ti z$ of $\ti \alpha$ is a universal fixed point in $\Gamma_{i-1}$ which contradicts the assumptions of this special case. Thus $\alpha = E_i \mu \overline E_i$ for some $\mu \subset G_{i-1}$. The covering translation $T$ that maps $\ti x$ to $\ti z$ is an element of $\T$ and so $T = T_0^p$ for some $p \ne 0$. Since $\mu$ lifts to the path connecting $\ti y \in \V(\T)$ to $T(\ti y) =T_0^p(\ti y) \in V(T_0)$, it follows that $\mu = w_i^p$ as desired.

For the general case, we will reduce to one of the two special cases by modifying $G$, sliding the terminal end of $\wt E_i$ to some point $\ti v \in \Gamma$ as described on pages 579--581 of section 5.4 of \BookOne. Choose $\ti v$ as follows, based on properties of $\ti v$, of the path $\ti\sigma$ from $\ti y$ to $\ti v$, and of its projection $\sigma$. If possible, choose $\ti v$ to be universally fixed, and make the choice so that $\ti v$ is the only universally fixed point on $\ti\sigma$, equivalently $E_i \sigma$ is an indivisible universal Nielsen path. Otherwise, using by Claim~2 that $\T$ is infinite cyclic, choose $\ti v \in \V(\T)$.  Let $\ti \sigma$ be the path from $\ti y$ to $\ti v$, and let $\sigma$ be the projected image in $G_{i-1}$ of $\ti\sigma$ (note that the ``special case'' is characterized by saying that we may choose $\ti v$ so that $\sigma$ is trivial).  Note also that $v$ is contained in the core of $G_{i-1}$.  If $\ti v \in \V(\T)$ this follows from      $\ti v \in \A(T_0)$.  
If $\ti v$ is universally fixed then it follows from the \lq weak version of the conclusion\rq\ established at the beginning of the proof, which implies that the indivisible universal Nielsen path $E_i \sigma$ is not a closed curve. 

  One slides $E_i$ along $\sigma$ by a folding operation, identifying a proper initial segment of $\overline E_i$ with $\sigma$. The effect is that $E_i$ is replaced by an edge $E'_i$ such that its lift $\wt E'_i$ with initial endpoint $\ti x$ has terminal endpoint~$\ti v$. This folding operation produces a new marked graph $G'$ and homotopy equivalences $p \from G \to G'$ and $p' \from G' \to G$ that are the \lq identity\rq\ on the common edges of $G$ and $G'$, that are homotopy inverses of each other and that preserve the markings. Also, under these homotopy equivalences there is a filtration by oriented one-edge strata on $G'$ corresponding to the given one on $G$ with $E_i$ replaced by $E'_i$. 

For each $f \in \cQ'(G)$ there is an $f' \in \cQ'(G')$ such that $f'(E') = (pfp')_\#(E')$ for each edge $E'$ of $G'$. This induces an isomorphism from $\cH' \subgroup \cQ'(G)$ to a subgroup of $\cQ'(G')$, which we continue to call $\cH'$. The new $\cH'$ continues to satisfy all the conclusions
of Proposition~\ref{PropKolchinGraph}. It is shown in the last paragraph on page 580 of \BookOne\ that $p_\#$ induces a period preserving bijection between the periodic Nielsen paths of $f$ and the periodic Nielsen paths of $f'$. This implies \pref{ItemPeriodOne} and implies that the inductive hypothesis still applies to our new $\cH'$. As discussed on the top of page 581, the lift of $f'$ to $\ti f'$ is still defined and $\ti f' \restrict \Gamma_{i-1} = \ti f \restrict \Gamma_{i-1}$. The positive effect for us is that that the terminal end of $E_i'$ is now $\ti v$ and we are reduced to one of the two special case. The proofs of all these assertions are routine applications of sliding and are left to the reader. 

This completes the inductive step, and hence the proof, of Lemma~\ref{LemmaGoodKolchinGraph}.
\end{proof}

%

\begin{proof}[Proof of Proposition~\ref{PropUPG}.] Let $\cH \subgroup K$ be UPG. We shall prove that each finitely generated subgroup $\cH_0 \subgroup \cH$ is abelian and that each element of $\cH_0$ is linearly growing. This completes the proof, because it follows that $\cH$ itself is abelian and that each of its elements is linearly growing, and one then applies the theorem that every abelian subgroup of $\Out(F_n)$ is finitely generated (see \cite{BassLubotzky:LinearCentral}, or \BookThree).

Pick any maximal $\cH_0$-invariant filtration by free factor systems $\emptyset = \F_0 \sqsubset \F_1 \sqsubset \cdots \sqsubset \F_N$ such that for some $ m \in \{0,\ldots,N\}$ we have $\F_\agen = \F_m$. Applying Proposition~\ref{PropKolchinGraph} we obtain $G$, a filtered marked graph with oriented one-edge strata, in which $ \F_\agen$ is realized by a filtration element $ G_t$, and we obtain a lift of $\cH_0$ to $\cH'_0 \subgroup \cQ'(G)$, and these satisfy the conclusions of Proposition~\ref{PropKolchinGraph}. 

For each $\phi \in \cH_0$ with corresponding lift $f \in \cH'_0$, since $f$ has no \eg\ strata, since $f \in \cQ'(G)$ and so is upper triangular, and since Proposition~\ref{PropKolchinGraph}~\pref{ItemPeriodOne} holds, the relative train track map $f$ satisfies the hypotheses of Lemma~\ref{LemmaFIX}, and its conclusions follow. From its conclusions~\pref{ItemSplitsIntoNielsen} combined with the fact that $\gamma_\eta$ fills, the terms of the highest edge splitting of $\gamma_\eta$ in $G$ form a collection of universal Nielsen paths which cover every edge of $G$. 
 The hypotheses of Lemma~\ref{LemmaGoodKolchinGraph} are therefore satisfied, and its conclusions follows. 

From conclusion~\pref{ItemNEGNielsenPathsUniv} of Lemma~\ref{LemmaGoodKolchinGraph} it follows that the elements of $\cH'_0$ commute with each other. Since $\cH_0$ is isomorphic to $\cH'_0$, it is also abelian. Item~\pref{ItemNEGNielsenPathsUniv} of Lemma~\ref{LemmaGoodKolchinGraph} also implies that each element of $\cH_0$ is linearly growing. 
\end{proof}

\subsection{Ordering attracting laminations} 
\label{SectionOrderingAttrLams}
In various places we will need to compare the subgroup systems in different free factors $F \subgroup F_n$ using the relation $\sqsubset$. But the $\sqsubset$ relation itself is not defined between subgroup systems in different free groups. To get around this, for any subgroup system $\A$ in $F$ we may identify $\A$ with a subgroup system in the ambient free group $F_n$, using malnormality of $F$ to obtain a \hbox{$\sqsubset$-preserving} bijection between the set of $F$-conjugacy classes of subgroups of~$F$ and the set of $F_n$-conjugacy classes of subgroups of~$F$.

 Let $\L(F_n)$ be the set of ordered pairs $(\Lambda,\phi)$ such that $\phi \in \Out(F_n)$ and $\Lambda \in \L(\phi)$. We usually abbreviate the ordered pair using subscript notation $\Lambda_\phi$; the meaning should be clear by context. When $\phi$ is understood we abbreviate further and write simply $\Lambda$.

Consider $\Lambda_\phi \in \L(F_n)$. Consider also a $\phi$-invariant free factor system of the form \hbox{$\F=\{[F]\}$} that carries~$\Lambda_\phi$. We may identify $\Lambda_\phi$ with an attracting lamination of $\phi \restrict F \in \Out(F)$ whose nonattracting subgroup system in the free group $F$ is denoted $\A_\na(\Lambda_\phi \restrict F)$. Then, as said above, we may also identify $\A_\na(\Lambda_\phi \restrict F)$ as a subgroup system in any free factor containing $F$.

The proof of Theorem~\ref{ThmFinitelyGenerated} is structured in part as an induction based on the following strict partial ordering on the set $\L(F_n)$. 
\begin{notn} \label{PartialOrder} We write $\Lambda_\phi \prec  \Lambda_\psi $ if either
\begin{enumerate}
\item $\F(\Lambda_\phi)$ is properly contained in $\F(\Lambda_\psi)$, \ \ or 
\item\label{ItemLaminationOrderSecond}
$\F(\Lambda_\phi) = \F(\Lambda_\psi)$ and, letting $\F$ denote this 
subgroup system, $\A_\na(\Lambda_{\psi} \restrict \F)$ is properly contained in $\A_\na(\Lambda_{\phi} \restrict \F)$.
\end{enumerate}
\end{notn}

\begin{lemma} \label{LemmaUniformBound} There is a uniform bound, depending only on $n$, to the length $s$ of a strictly ordered sequence $\Lambda_{\psi_1} \prec \Lambda_{\psi_2} \prec  \ldots \prec  \Lambda_{\psi_s}$ in $\L(F_n)$.
\end{lemma}

\begin{proof} Denote $\Lambda_{\psi_i}$ by $\Lambda_i$. Since $\rank(\F(\Lambda_{i}))$ is strictly increasing and bounded above by~$n$, we may assume that (2) holds for each term in the sequence. Under that assumption there is a free factor $F_r \subgroup F$ such that $\F(\Lambda_i) = \F = \{[F_r]\}$ is independent of $i$. Also, each $\A_{\na}(\Lambda_{i} \restrict \F)$ is a vertex group system for~$F_r$ by Proposition 1.4 of \SubgroupsThree. The uniform bound on $s$ comes then from the uniform bound (Proposition 3.2 of \SubgroupsOne) on the length of a strictly nested sequence of vertex groups.
 \end{proof}
 
\subparagraph{Remark.} As noted in the proof of Lemma~\ref{LemmaUniformBound}, if $\F=\{[F]\}$ is a free factor supporting $\A_\na(\Lambda_\phi)$ then $\A_\na(\Lambda_\phi)$ is a vertex group system in $F$. One can also show that for any free factor system of the form $\F'=\{[F']\} \sqsupset \F$, the subgroup system in $F'$ identified with $\A_\na(\Lambda_\phi)$ is also a vertex group system; we omit the proof since we do not need this fact.

\subsection{Geometric models.} 
\label{SectionGeomModels}
We review further details of ``geometric models'' and definitions of geometricity for \eg\ strata and for attracting laminations. 

Consider a rotationless $\phi \in \Out(F_n)$, a representative \ct\ $f \from G \to G$, and an \eg\ stratum~$H_r$. A \emph{weak geometric model} for $H_r$ consists of a compact connected surface~$S$, a component $\bdy_0 S$ of $\bdy S$, a pseudo-Anosov homeomorphism $\Psi \from S \to S$ taking $\bdy_0 S$ to $\bdy_0 S$, a 2-complex $Y$ obtained as a quotient $j \from G_{r-1} \disjunion S \to Y$ by using a gluing map $\alpha \from \bdy S - \bdy_0 S \to G_{r-1}$ that is $\pi_1$-injective on each component, and an extension of the embedding $G_{r-1} \inject Y$ to an embedding $G_r \inject Y$, such that the following properties hold: there is a deformation retraction $d \from Y \mapsto G_r$; 
denoting $dj = d \composed j \from S \to G_r$, the maps $S \xrightarrow{\Psi} S \xrightarrow{dj} G_r$ and $S \xrightarrow{dj} G_r \xrightarrow{f} G_r$ are homotopic; the intersection of $G_r$ with $\bdy_0 S$ in $Y$ is a single point $x \in H_r - G_{r-1}$; and the closed path based at $x$ that goes around $\bdy_0 S$ is homotopic in $Y$ to a closed indivisible Nielsen path in $G$. A \emph{geometric model} is obtained as the quotient $X$ of $Y \disjunion G$ by identifying the copies of $G_{r}$ in $Y$ and in $G$; the deformation retraction extends to $d \from X \to G$. The stratum $H_r$ is \emph{geometric} if a weak geometric model for $H_r$ exists; if this is the case, then the map $dj \from S \to G$ is $\pi_1$-injective, and if $\Lambda^u \subset S$ is the unstable geodesic lamination of $\Psi$ with respect to some hyperbolic structure on $S$ then the map of line spaces $\B(\pi_1 S) \to \B(\pi_1 G) = \B(F_n)$ induced by $dj$ takes the set of leaves of $\Lambda^u$ homeomorphically to the attracting lamination $\Lambda$ corresponding to $H_r$. As said earlier, an attracting lamination $\Lambda \in \L(\phi)$ is \emph{geometric} if the \eg\ stratum corresponding to $\Lambda$ in \emph{some} representative \ct\ $\phi$ is geometric, which occurs if and only if this holds for \emph{every} representative \ct. See Section~\ref{SectionBackground} for a brief review of conditions on $H_r$ equivalent to geometricity, and see Section 5.3 of \BookOne\ and Section 2 of \SubgroupsOne\ for more details.

 In $X$ define the \emph{complementary subgraph} to be $L = (G \setminus H_r) \union \bdy_0 S$.
 
 
\begin{fact} \label{Fact:PropertiesOfL} The inclusion $L \inject X$ is $\pi_1$-injective on each component and the image subgroups are mutually malnormal. If $\A_\na\Lambda$ fills then the following hold: \begin{itemize}
\item[(i)] $\A_\na\Lambda$ is represented in $X$ by the subgraph $L$, meaning $\A_\na\Lambda$ equals the set of conjugacy classes of the image subgroups. 
\item[(ii)] $L$ is disjoint from the manifold interior of $S$, and the latter is therefore an open subset of~$X$.
\end{itemize}
\end{fact}

\begin{proof} $\pi_1$-injectivity and mutual malnormality follow from Lemma~2.7 of \SubgroupsOne. To prove (i), by \SubgroupsThree\ Definition~1.2 and Remark~1.3 it follows that $\A_\na(\Lambda_m)$ is represented by a subgraph $K \subset L$ that contains $G_{r-1} \union \bdy_0 S$. By the assumption that $\A_\na\Lambda$ fills it follows that the subgraph $K$ must also contain every edge of $G \setminus G_{r}$, and so $K=L$. The proof of (ii) follows from \SubgroupsOne\ Definition 2.10, which shows that at any ``attaching point'' where $L$ touches the interior of $S$, one can pull $L$ away from $S$ inserting an edge, constructing a topological model which exhibits $L$ as being supported on a proper free factor, contrary to the hypothesis that $\A_\na\Lambda$ fills.
\end{proof}

When $\Lambda \in \L(\phi)$ is geometric, and when a geometric model is specified as above, the surface $S$ together with the associated monomorphism $\mu = dj_* \from \pi_1(S) \inject F_n$ will be called the \emph{surface system associated to $\Lambda$ with respect to the geometric model for $\phi$}. The 
following fact says, in essence, that the surface system associated to 
$\Lambda$ is well-defined independent of the choice of geometric model. 

\begin{fact}
\label{FactSurfaceUnique}
Consider 
$\phi \in \Out(F_n)$ and a geometric $\Lambda \in \L(\phi)$ with the property that $\A_\na\Lambda$ fills~$F_n$. Let $\mu^i \from \pi_1(S^i) \inject F_n$ ($i=1,2$) be the two surface systems associated to $\Lambda$ with respect to geometric models of \eg\ strata corresponding to $\Lambda$ in two \cts\ representing rotationless iterates of~$\phi$. Then there exists a homeomorphism $h \from S^1 \to S^2$, unique up to isotopy, and an inner automorphism $i \from F_n \to F_n$, such that $i \composed \mu^1 = \mu^2 \composed h_* \from \pi_1(S^1) \to F_n$.
\end{fact}

\begin{proof} To prove isotopy uniqueness of $h$, given another such $h' \from S^1 \to S^2$, the two monomorphisms $\mu^2_* \, \composed h_*$, $\mu^2_* \, \composed h'_* \from \pi_1(S^1) \to F_n$ differ by postcomposing with an inner automorphism of $F_n$, and so the two monomorphisms $\mu^2_*$, $\mu^2_* \composed h'_* \composed h^\inv_* \from \pi_1(S^2) \to F_n$ differ by postcomposing with inner automorphism of $F_n$. Since the image of $\pi_1(S^2)$ in $F_n$ is its own normalizer (\SubgroupsThree\ Lemma~2.7~(2)), the latter two monomorphisms differ by precomposition with an inner automorphism of $\pi_1(S^2)$. It follows that $h'_* \composed h^\inv_*$ is itself an inner automorphism of $\pi_1(S^2)$. By the Dehn-Nielsen-Baer theorem \cite{FarbMargalit:primer}, $h'h^\inv $ is isotopic to the identity.

We turn to proof of existence of $h$. For $i=1,2$ let $f^i \from G^i \to G^i$ be \cts\ representing a rotationless power of $\phi$ with \eg\ strata $H^i_{r_i}$ associated to $\Lambda$, and let $X^i$ be a geometric model for $H^i_{r_i}$ with all accompanying notations indicated with a superscript~$i$, so that $\mu^i = dj^i \from \pi_1(S^i) \to \pi_1(G^i) \approx F_n$ is the associated surface system. From any marking change map $g \from G^1 \to G^2$ we obtain a marking change map $X^1 \xrightarrow{d^1} G^1 \xrightarrow{g} G^2 \subset X^2$ also denoted $g \from X^1 \to X^2$. Let $\rho^i$ be the indivisible Nielsen path in $G^i$ of height $r_i$, with base point $x^i = H^i_{r_i} \intersect \bdy_0 S^i$ and let $L^i = (G^i \setminus H^i_{r_i}) \union \bdy_0 S^i$ be the complementary subgraph. 

By Fact~\ref{Fact:PropertiesOfL}, $[L^1] = \A_\na\Lambda=[L^2]$. Thus by the homotopy extension theorem we may homotope the marking change map $g \from X^1 \to X^2$ so as to take $L^1$ to $L^2$ by a homotopy equivalence. By Fact~\ref{Fact:PropertiesOfL} we may apply the following fact to the homotopy equivalence of pairs $g \from (X^1,L^1) \to (X^2,L^2)$:

\begin{fact}\label{FactNonreflexiveHE}
For $i=1,2$ let $K^i$ be a finite graph, $S^i$ a compact connected surface with nonempty boundary and negative Euler characteristic, $\beta^i \from \bdy S^i \to K^i$ a map that is $\pi_1$-injective on each component, and $j^i \from K^i \disjunion S^i \to X^i$ the quotient map to a finite connected complex defined by identifying $x \sim \beta^i(x)$ for each $x \in \bdy S^i$. For any homotopy equivalence of pairs $g \from (X^1,K^1) \to (X^2,K^2)$ there exists a homeomorphism $h \from S^1 \to S^2$ such that the maps $j^2 \composed h$, $g \composed j^1$ are homotopic. 
\end{fact}
\noindent
Putting off its proof for a moment, from the conclusion of Fact~\ref{FactNonreflexiveHE} it follows that the maps $dj^2 \composed h, \, g \composed dj^1 \from S^1 \to G^2$ are homotopic, and so the two monomorphisms the two monomorphisms $dj^2_* \composed h_*$, $dj^1_* \composed g_* = dj^1_*\from \pi_1 S^1 \to \pi_1 G^2 = F_n$ differ by postcomposing with an inner automorphism of~$F_n$. This completes the proof of Fact~\ref{FactSurfaceUnique}. \end{proof}

\begin{proof}[Proof of Fact~\ref{FactNonreflexiveHE}] This proof is very close to that of \SubgroupsOne\ Lemma 2.21, which is the special case that equations $X^1=X^2$ and $K^1=K^2$ hold, but one major step is different. Denote $g^1=g$, which we may assume is simplicial, by the simplicial approximation theorem. Following \cite{ScottWall}, decompose $X^1$ into a graph of spaces: the edge spaces are the interiors of the components of a regular neighborhood $N(\bdy S^1)$; one vertex space is $\closure(S^1-N(\bdy S^1))$ whose fundamental group is identified with $\pi_1(S^1)$; the other vertex spaces are the components of~$K^1$. By a ``curve'' we mean a homotopically nontrivial closed curve. Applying Bass-Serre theory to this graph of spaces, and using that curves in distinct components of $\bdy S^1$ are not homotopic in $S^1$, one concludes the following: for each curve $c$ in $S^1$, if $c$ is homotopic in $X^1$ to a curve $c'$ in $S^1$ such that $c,c'$ are not homotopic in $S^1$, or if $c$ is homotopic in $X^1$ to a curve in $K^1$, then $c$ is homotopic into $\bdy S^1$; and if two curves in $K^1$ are homotopic in $X^1$ then they are homotopic in $K^1$. 

Consider the subcomplex $Z^1 = (g^1)^\inv(K^2) \subset X^1$ which contains $K^1$. If $c$ is a nonperipheral curve in $\interior(S^1)$ then $c$ is not contained in $Z^1$ because if it were then there would be a curve $c'$ in $K^1$ such that $g^1(c),g^1(c')$ are homotopic, forcing $c,c'$ to be homotopic, a contradiction. Each component of $Z^1$ is therefore contained either in a neighborhood of a component of $\bdy S^1$ or in a subdisc of $S^1$. Components of $Z^1$ of the latter type can be removed by homotopy of $g^1$. It follows that, with appropriate choice of base points, $(g^1)_*(\pi_1 S^1) \subset \pi_1 S^2$. Using a homotopy inverse $g^2 \from (X^2,K^2) \to (X^1,K^1)$, the exact same argument applies with the superscripts $1,2$ reversed. It follows in turn that $(g^1)_*(\pi_1 S^1) = \pi_1 S^2$ (the corresponding step of \SubgroupsOne\ Lemma 2.21 applies \BookOne\ Lemma~6.0.6 which does not apply here). This isomorphism takes peripheral elements of $\pi_1 S^1$ to peripheral elements of $\pi_1 S^2$, and so by the Dehn-Nielsen-Baer theorem it is induced by a homeomorphism $h \from S^1 \to S^2$, which evidently satisfies the conclusions of the lemma.
\end{proof}

While the following is a well known to experts, we give a simple direct proof for completeness, and to motivate the proof given below of Lemma~\ref{LemmaMultiedge}~\pref{item:carried by S}.

\begin{fact} \label{FactLamsAreLams} Suppose that $S$ is a surface with boundary and that $\mu \in \mcg(S)$ corresponds to $\phi\in \Out(F_n)$ under an identification of $\pi_1(S)$ with $F_n$. Then $\L(\phi)$ is the set of unstable laminations for a Thurston decomposition of $\mu$.
\end{fact}

\begin{proof} After passing to an iterate of $\mu$ we may assume that each unstable lamination $\Lambda_\mu$ for $\mu$ has a $\mu$-invariant leaf $\gamma_\mu$. That leaf is birecurrent and non-periodic because $\Lambda_\mu$ is minimal and not a closed curve. The line $\gamma_\mu$ has an attracting neighborhood $U$ in the weak topology for the action of $\mu$: letting $\ti\mu \from \wt S \to \wt S$ be a lift of $\mu$ for which there is a $\ti\mu$-invariant lift $\ti\gamma_\mu$ of $\gamma_\mu$, by Nielsen Theory the endpoints of $\ti\gamma_\mu$ form an attracting set for the action of $\ti\mu$ on $\bdy\pi_1 S$, and so a small neighborhood of this attracting set produces the desired set $U$. It follows that $\Lambda_\mu \in \L(\phi)$ by Definition 3.1.5 of \BookOne. 

Suppose now that $\Lambda_\phi \in \L(\phi)$. Choose a hyperbolic structure on $S$ in which the reducing curves in the Thurston decomposition of $\mu$ are geodesic. After passing to an iterate we may assume that components in the Thurston decomposition are invariant, that the restriction of $\mu$ to each component is either the identity or pseudo-Anosov, that $\phi$ is rotationless and that there is a $\mu$-invariant generic leaf $\gamma_\phi$ in the realization of $\Lambda_\phi$ in $S$. Suppose that $[a] \in F_n$ is weakly attracted to $\Lambda_\phi$ under the action of $\phi$. For each $k \ge 0$, let $\alpha_k$ be the closed geodesic in $ S$ corresponding to $\phi^k([a]) =\mu^k([\alpha_0])$. The number of intersections of $\alpha_k$ with the set of reducing curves of the Thurston decomposition is independent of $k$ while the length of $\alpha_k$ goes to $\infty$. It follows that the ends of the geodesic in $S$ corresponding to any weak limit of the $\alpha_k$'s are disjoint from the reducing curves. Applying this to the birecurrent weak limit $\gamma_\phi$ we have that $\gamma_\phi$ is realized by a geodesic that is disjoint from the reducing curves and so contained in a single, necessarily pseudo-Anosov, component $S_0$ of the Thurston decomposition. We may assume without loss that $\alpha \subset S_0$.
 
Letting $\Lambda^u_0$ and $\Lambda^s_0$ be the unstable and stable foliations for $\mu \restrict S_0$, we may homotop $\alpha_0$ to a closed curve $\alpha_0'$ that it is an alternating concatenation of geodesic paths $\sigma_i \subset \Lambda^u_0$ and $\tau_i \subset \Lambda^s_0$. By iterating the pseudo-Anosov homeomorphism representing $\mu \restrict S_0$, we obtain closed curves $\alpha'_k$ realizing $\phi^k([a])$ that decompose into a concatenation of geodesic paths in $\Lambda^u_0$ whose lengths $\to \infty$ and paths in $\Lambda^s_0$ whose lengths $\to 0$. It follows that any birecurrent weak limit of the $\alpha'_k$'s is a weak limit of $\Lambda^u_0$, and so by minimality, is a leaf of $\Lambda^u_0$. Thus $\gamma_\phi$ is dense in both $\Lambda^u_0$ and $\Lambda_\phi$ so these two laminations are equal and we are done. 
\end{proof}

 Associated to a compact surface $S$ with nonempty boundary are its mapping class group $\mcg(S)$, which is the group of homeomorphisms of $S$ modulo isotopy, and its boundary relative mapping class group $\mcg(S,\bdy S)$, which is the group of all homeomorphisms that restrict to the identity on $\bdy S$ modulo isotopy through such homeomorphisms. We denote the finite index subgroup of $\mcg(S)$ consisting of elements that setwise fix each component of $\partial S$ by $\mcg_0(S)$. The induced surjective homomorphism $\mcg(S,\bdy S) \to \mcg_0(S)$ will be called the \emph{despinning homomorphism}.

Given $\Lambda_\phi \in \L(F_n)$ with associated surface system $\pi_1 S \inject F_n$, we identify $\pi_1 S$ as a subgroup $\pi_1 S \subgroup F_n$ with conjugacy class $[\pi_1 S]$ and stabilizer subgroup $\Stab[\pi_1 S] \subgroup \Out(F_n)$. The subgroup $\pi_1 S$ is its own normalizer in $F_n$ (\SubgroupsOne\ Lemma 2.7~(2)), and so there is a well-defined induced homomorphism $\Stab[\pi_1 S] \to \Out(\pi_1 S)$ (\SubgroupsOne\ Fact~1.4). Associated to each oriented component of $\bdy S$ is a conjugacy class in the group $\pi_1 S$ called a \emph{peripheral conjugacy class}. According to the Dehn-Nielsen-Baer Theorem the subgroup of $\Out(\pi_1 S)$ that preserves the set of peripheral conjugacy classes is naturally isomorphic to $\mcg(S)$. 

\begin{ex} \label{surface example} Suppose that $F_{m+2} = F_m \ast \<A,B\>$ and that $s \in F_m$ fills $F_m$. In Example~\ref{filling reducible}, we constructed an element $\phi \in \Out(F_{m+2})$ with a filling lamination $\Lambda$ whose stabilizer contains every element $\theta \in \Out(F_{m+2})$ that is represented by an automorphism $\Theta$ that fixes each element of $\<A,B,s\>$. Here we refine the construction to produce an example in which the stabilizer of $\Lambda$ contains a boundary relative mapping class group.

Identify $F_m$ with $\pi_1(S,v)$ for some surface $S$ with unique boundary component $\sigma$. Let $X$ be the two complex obtained by attaching a pair of loops $A, B$ to a basepoint $v$ in $\sigma$ and identify $\pi_1(X,v)$ with $F_{m+2} = F_m \ast \<A,B\>$. Let $s \in F_{m+2}$ be the element determined by $\sigma$ and note that $s$ fills $F_m$ (Lemma~2.5 of \SubgroupsOne\ (1)). Each $\nu \in \mcg(S,\partial S)$ is represented by a homeomorphism $h:S \to S$ that pointwise fixes $\sigma$ and so extends by the identity on $A$ and $B$ to a homotopy equivalence of $X$ whose induced action on $\pi_1(X,v)$ fixes $\<\sigma,A,B\>$. There is an induced injective homomorphism $\mcg(S,\partial S) \to \Out(F_{m+2})$ whose image is contained in the stabilizer of $\Lambda$. 
\end{ex} 


\begin{lemma} \label{LemmaMultiedge} Under the ``Standing assumptions of Section~\ref{SectionExpansionKernel}'', suppose that $\cH \subgroup K$, that $\emptyset = \F_0 \sqsubset \F_1 \sqsubset \cdots \sqsubset \F_k = \{[F_n]\}$ is a maximal $\cH$-invariant filtration by properly nested free factor systems and that $\F_\agen$ is one of the $\F_i$'s. Let $1 \le i_1 < \cdots < i_M \le k$ be the indices $i$ for which $\F_{i-1} \sqsubset \F_i$ is a multi-edge extension. Suppose also that a rotationless $\phi_m \in \cH$ is given for each $m=1,\ldots,M$ so that $\phi_m$ is irreducible rel $\F_{i_m-1} \sqsubset \F_{i_m}$ with associated attracting lamination~$\Lambda_m$. Noting by Lemma~\ref{LemmaFIX} that each $\Lambda_m$ is geometric and that each $\A_\na(\Lambda_m)$ fills~$F_n$, let $\pi_1 S_m \inject F_n$ be the surface system associated to $\Lambda_m \in \L(F_n)$ (which is well-defined in the sense of Fact~\ref{FactSurfaceUnique}). 
\begin{enumerate}
\item \label{item:maps to mcg} 
$\cH$ stabilizes each $[\pi_1 S_m]$, and the induced homomorphism $\cH \mapsto \Stab[\pi_1 S_m] \inject \Out(\pi_1 S_m)$ has image in $\mcg(S_m)$, inducing a homomorphism $\xi_m : \cH \to \mcg(S_m)$.
\item \label{item:carried by a surface} For each $\psi \in \cH$ and $\Lambda_\psi \in \L(\psi)$ there exists $m \in \{1,\ldots, M\}$ such that the following hold:
\begin{enumerate}
\item \label{item:carried by S}$\Lambda_\psi$ is carried by $[\pi_1 S_m]$.
\item \label{item:nestinglams}
Either $\Lambda_\psi \prec \Lambda_m$ or $\psi$ is irreducible rel $\F_{i_m-1} \sqsubset 
\F_{i_m}$. 
\item \label{item:kernel of xi} 
$\xi_m (\psi)$ is non-trivial. 
\end{enumerate}
\item \label{item:despinning}
There is a homomorphism 
$\Theta \from \prod_1^M \mcg(S_m,\bdy S_m) \to \Ker(PF) \cap \bigcap_1^M \Stab[\pi_1 S_m]
$ \break
such that its composition with the homomorphism
$\bigcap_1^M \Stab[\pi_1 S_m] \to \prod_1^M \Out(\pi_1 S_m)
$
has image in $\prod_1^M \mcg(S_m)$, and such that the composition $\prod_1^M \mcg(S_m,\bdy S_m) \to \prod_1^M \mcg(S_m)$ is the product of the despinning homomorphisms and in particular has finite index image.
\end{enumerate}
\end{lemma}

\begin{proof} Throughout this proof, by applying the existence theorem (Theorem 4.28 of \recognition) for \cts, we choose the \cts\ representing all rotationless elements of $\cH$ to have core filtration elements representing the free factor systems $\F_1 \sqsubset\cdots\sqsubset \F_k$.

For the proofs of \pref{item:maps to mcg}, \pref{item:carried by a surface} we set the notation of appropriate geometric models. For each $m=1,\ldots,M$ choose a \ct\ $f^m \from G^m \to G^m$ representing $\phi_m$. Let $H^m_{r(m)} \subset G^m$ be the \eg\ geometric stratum corresponding to $\Lambda_m$, having height $r(m)$ in~$G^m$. Let $Y_m \subset X_m$ be the weak geometric model and the geometric model, respectively, for the stratum $H^m_{r(m)}$ with respect to~$f^m$; we may assume that the surface system $\pi_1 S_m \to F_n$ that is given in the hypothesis arises from the surface $S_m$ of that geometric model. Denote its upper boundary by $\bdy_0 S_m$, and let $\bdy_\low S_m = \bdy S_m - \bdy_0 S_m$ denote its \emph{lower boundary}. Thus, there is a quotient map $j_m \from G^m_{r(m)-1} \disjunion S_m \to Y_m$ defined by attaching each component of $\bdy_\low S_m$ to $G^m_{r(m)-1}$ by a $\pi_1$-injective map, the embedding $G^m_{r(m)-1} \inject Y_m$ extends to an embedding $G^m_{r(m)} \to Y_m$, and $X_m$ is the quotient of $Y_m \disjunion G^m$ by identifying the copies of $G^m_{r(m)}$. Let $L_m = (G^m \setminus H^m_{r(m)}) \union \bdy_0 S_m$ be the complementary subgraph.

\medskip

Conclusion~\pref{item:maps to mcg} exactly matches the conclusion of Theorem~J of \SubgroupsFour, although that theorem is stated only for the case when the geometric stratum in question is the top stratum. But we easily reduce to that case by restricting $\phi_m$ to the component $\F' = \{[F']\}$ of $\F_{i_m}$ that carries $\Lambda_m$: the restriction $\phi' = \phi_m \restrict F' \in \Out(F')$ is rotationless; it is represented by the \ct\ $f' \from G' \to G'$ which is the restriction of $f^m$ to the component $G^m$ containing $H' = H^m_{r(m)}$, and $H'$ is its top stratum; the component $Y'$ of $Y_m$ that contains $H'$ is a geometric model for $H'$ relative to $f'$; and $S_m$ is the surface of that geometric model. Theorem~J then indeed applies in that restricted context, and its conclusions immediately imply item~\pref{item:maps to mcg}.

\medskip

For the proof of conclusion~\pref{item:carried by a surface}, fix $\psi \in \cH$ and $\Lambda_\psi \in \L(\psi)$. Passing to a power we may assume $\psi$ is rotationless. Choose a \ct \ representing $\psi$ with core filtration elements $G_{a(i)}$ representing each $\F_{i}$. Let $H_s$ be the stratum corresponding to $\Lambda_\psi$, choose $i$ so that $a(i-1) < s \le a(i)$ (taking $a(0)=0$). Then $\Lambda_\psi$ is carried by $\F_i$ but not by $\F_{i-1}$ and hence $\F_{i-1} \sqsubset \F_i$ is a multi-edge extension. In particular, $i=i_m$ for some $m=1,\ldots,M$, and $\Lambda_\psi$~is not carried by $\F_{i_m-1}$.

Applying \pref{item:maps to mcg}, $\psi$ preserves $[\pi_1 S_m]$ and restricts to a mapping class $\psi_m=\xi_m(\psi) \in \MCG(S_m) \subgroup \Out(\pi_1 S_m)$. As in the proof of (1), we may assume that $S_m$ corresponds to the highest stratum of $f^m :G^m \to G^m$, equivalently we may assume $X_m = Y_m$. Let $N_m \subset S_m$ be a collar neighborhood of $\bdy_\low S_m$ and let $S^*_m = \cl(S_m - N_m)$. The quotient map $j_m$ restricts to an embedding of $S^*_m$ as a closed subset of $Y_m$ and to an embedding of $\interior(N_m)$ as an open subset of $Y_m$ such that $Y_m - \interior(N_m)$ is the disjoint union of $G^m_{r(m)-1}$ and $S^*_m$. 

Choose a homeomorphism $\bdy_\low S_m \cross [0,1] \approx N_m$ under which $(x,0) \approx x$, and define an \emph{arc fiber in $N_m$} to be the image under this homeomorphism of a set of the form $x \cross [0,1]$ for some $x \in \bdy_\low S_m$. Each arc fiber in $N_m$ has one endpoint on $\bdy_\low S_m$ and opposite endpoint on $\bdy_\low S^*_m \approx \bdy_\low S_m \times 1$. Also define an \emph{arc fiber in $Y_m$} to be the embedded image under $j_m$ of an arc fiber in $S_m$, having one endpoint in $G^m_{r(m)-1}$ and opposite endpoint on $\bdy_\low S^*_m$. There is a deformation restriction $S_m \to S^*_m$ which collapses each arc fiber in $N_m$ to its endpoint on $\bdy_\low S^*_m$; under this homotopy equivalence we have an induced isomorphism $\mcg(S^*_m) \approx \mcg(S_m)$. Also, there is a deformation retraction from $\closure(Y_m - \ S^*_m) = G^m_{r(m)-1} \union j(N_m)$ to $G^m_{r(m)-1}$ which collapses each arc fiber in $Y_m$ to its endpoint on $G^m_{r(m)-1}$.  

As in the proof of Fact~\ref{FactSurfaceUnique}, the homotopy extension theorem and Fact~\ref{FactNonreflexiveHE} implies the existence of a homotopy equivalence $\Psi \from Y_m \to Y_m$ representing $\psi$ such that $\Psi$ restricts to a homeomorphism of $S^*_m$, and to a homeomorphism of $\bdy_\low S^*_m$, and to a self-homotopy equivalence of $G^m_{r(m)-1}$ (namely the restriction of $f^m$), and its restriction $\interior(N_m)$ has image contained in $\interior(N_m) \union G^m_{r(m)-1}$.

Our strategy for proving \pref{item:carried by S} is to choose a closed curve $\gamma$ representing a conjugacy class $[\gamma$] that is weakly attracted to $\Lambda_\psi$ under iteration by $\Psi$ and to prove that each birecurrent weak limit of $\psi$-iterates of $[\gamma]$ is carried by either $S_m$ or $G^m_{r(m)-1}$. Since a generic leaf of $\Lambda_\psi$ cannot satisfy the latter it must satisfy the former and we will be done. This strategy is similar to the proof of Fact~\ref{FactLamsAreLams} above. To carry out this strategy we have to understand how $\Psi$-iterates of $\gamma$ intersect $\bdy_\low S^*_m$. 

Following \SubgroupsOne\ Section~2.2 (but ignoring the top boundary $\bdy_0 S_m$), we obtain a graph of spaces decomposition of $Y_m$, having as edge spaces the components of $\interior(N_m)$, and as vertex spaces the \emph{surface vertex space} $S^*_m$ and the components of $G^m_{r(m)-1}$ called the \emph{non-surface vertex spaces}. 

Define a \emph{vertex path} in $Y_m$ to be a finite path which intersects $\bdy_- S^*_m$ solely at its endpoints. A vertex path is \emph{essential} if it is not homotopic rel endpoints into $\bdy_- S^*_m$. A vertex path contained in $S^*_m$ it is called a \emph{surface vertex path}, and one in $\closure(Y_m - \ S^*_m)$ is called a \emph{non-surface vertex path}. Every essential non-surface vertex path is homotopic rel endpoints to a unique path of the form $\alpha * q * \alpha'$ where $\alpha'$ is an arc fiber in $Y_m$ oriented to have initial endpoint on $\bdy_- S^*_m$, $q$ is a path in $G^m_{r(m)-1}$, and $\alpha$ is an arc fiber in $Y_m$ oriented to have terminal endpoint on $\bdy_- S^*_m$.

Also define an \emph{essential vertex ray} in $Y_m$ to be a ray $[0,\infty) \to Y_m$ intersecting $\bdy_- S^*_m$ solely at its endpoint, whose lift to the universal cover is a proper map. As above we may speak about \emph{surface vertex rays} and \emph{nonsurface vertex rays}. 

Every homotopically nontrivial closed curve $\gamma$ in $Y_m$ may be homotoped so that it is \emph{efficient} meaning that $\gamma$ has one of the following types:
\begin{description}
\item[Edge curve:] $\gamma$ is contained in $\bdy_\low S^*_m$; \, or
\item[Surface vertex curve:] $\gamma$ is contained in the surface vertex space $S^*_m$, it is homotopically nontrivial in $S^*_m$, and it is not homotopic to an edge curve; \, or
\item[Nonsurface vertex curve:] $\gamma$ is contained in a nonsurface vertex space, it is homotopically nontrivial in that vertex space, and is not homotopic to an edge curve; \, or
\item[Efficient concatenation:] $\gamma$ is an alternating concatenation of essential surface vertex paths and essential non-surface vertex paths.
%
\end{description}
By mimicking similar arguments for surfaces, one can prove that for every pair of efficient closed curves $\gamma,\gamma'$ in $Y_m$, if $\gamma,\gamma'$ are homotopic then they are homotopic through efficient closed curves of the same type, and in particular $\gamma,\gamma'$ have the same type. In the case that $\gamma,\gamma'$ are both efficient concatenations, it follows that their concatenation expressions have the same number of terms and (up to cyclic permutation) the same sequence of essential vertex paths up to homotopy through vertex paths. This may be proved by starting with a homotopy $h \from S^1 \times [0,1] \to Y_m$ between $\gamma$ and $\gamma'$, chosen (after a perturbation rel $S^1 \times \{0,1\}$) so as to be in general position with respect to $\bdy_- S^*_m$. Note that since $\gamma,\gamma'$ are essential, no component of the pullback of $\bdy_- S^*_m$ is an arc whose endpoints are either both on $S^1 \times 0$ or both on $S \times 1$. One can then do a further homotopy of $h$ rel $S^1 \times \{0,1\}$ so as to remove circle components of the pullback of $\bdy_- S^*_m$.




Consider now any efficient closed curve $\gamma$ whose corresponding conjugacy class in $F_n$ is weakly attracted to $\Lambda_\psi$ under iteration by $\psi$, and so in particular a generic leaf $L$ of $\Lambda_\psi$ is a weak limit of $\psi^i[\gamma]$. The images $\Psi^i(\gamma)$ are also efficient closed curves, all of the same type as~$\gamma$.
%
%
If $\gamma$ is an edge curve or a nonsurface vertex curve then so are all of its $\Psi$-iterates and their weak limits, contradicting that $L$ is not carried by $G^m_{r(m)-1}$. If $\gamma$ is a surface vertex curve then so are all of its $\Psi$-iterates and we are done. 

We are reduced to the case that $\gamma$ has an efficient concatenation, denoted (up to cyclic permutation) as
$$\gamma = \nu_1\, \mu_1 \, \cdots \, \nu_k \, \mu_k
$$
where each $\nu_k$ is an essential surface vertex curve and each $\mu_k$ is an essential non-surface vertex curve. For each $i \ge 1$ we have an efficient concatenation expression of the same type, namely
$$\Psi^i(\gamma) = \alpha^i_1\,  \mu^i_1\,  \cdots\,  \alpha^i_{2K}\,  \mu^i_{2K}
$$
Since the length of the sequence of conjugacy classes $[\psi^i(\gamma)]$ goes to infinity, it follows that the line $L$ may be represented by an efficient concatenation with at most $2K$ terms. If the number of terms is $\ge 2$ then the first and last terms are essential rays, but since $L$ is birecurrent we get a contradiction. It follows that $L$ is a single term and hence is contained in one of the vertex spaces, completing the proof of \pref {item:carried by S}.

Fact~\ref{FactLamsAreLams} implies that $\Lambda_\psi$ is the unstable foliation for a pseudo-Anosov component $S' \subset S_m$ in the Thurston decomposition for $\xi_m(\psi)$. In particular, \pref{item:kernel of xi} is satisfied. 

For \pref{item:nestinglams} let $\C(S_m)$ be the set of conjugacy classes of $F_n$ carried by $[\pi_1(S_m)]$. The \lq span argument\rq\ for geometric strata (see Lemma 7.0.7 of \BookOne\ or Proposition 2.15~(4) of \SubgroupsOne) implies that $\F(\Lambda_m)$ carries each element of $\C(S_m)$.
It follows by \pref{item:carried by S} that $\Lambda_\psi$ is a weak limit of elements of $\C(S_m)$, and so $\Lambda_\psi$ is carried by $\F(\Lambda_m)$, proving that $\F(\Lambda_\psi) \sqsubset \F(\Lambda_m)$. If $\F(\Lambda_\psi) \ne \F(\Lambda_m)$ then $\Lambda_\psi \prec \Lambda_m$ and we are done so suppose that $\F(\Lambda_\psi) = \F(\Lambda_m)$, which we now simply call $\F$. 

Consider now  $\A_\na(\Lambda_{m} \restrict \F)$ regarded as a subgroup system in $F_n$. A conjugacy class in $F_n$ is carried by $\A_\na(\Lambda_{m} \restrict \F)$ if and only if it is carried by $\A_\na(\Lambda_{m})$ and by $\F$. We know that $\F \sqsubset \F_{i_m}$ and that a conjugacy class in $\F_{i_m}$ is carried by $\A_\na(\Lambda_{m})$ if and only if it is carried by $\F_{i_m -1}$ or is represented by an iterate of $\partial_0 S_m$. It follows that a conjugacy class is carried by $\A_\na(\Lambda_m \restrict \F)$ if and only if it is carried by the meet $\F \wedge \F_{i_m -1}$ or is represented by an iterate of $\partial_0 S_m$. Each such conjugacy class is carried by $\A_\na(\Lambda_\psi)$, which we can see as follows. First, applying conclusion~\pref{item:maps to mcg}, the mapping class $\xi_m(\psi) \in \MCG(S_m)$ permutes boundary components, and so each iterate of $\bdy_0 S_m$ is $\psi$-periodic and so is not weakly attracted to $\Lambda_\psi$. Second, the set of conjugacy classes carried by $\F_{i_m-1}$ is $\psi$-invariant, and hence no sequence of them can weakly converge to $\Lambda_\psi$ which is not carried by $\F_{i_m-1}$. This proves that $\A_\na(\Lambda_m \restrict \F) \sqsubset \A_\na(\Lambda_{\psi} \restrict \F)$. If $\A_\na(\Lambda_m \restrict \F) \ne \A_\na(\Lambda_{\psi} \restrict \F)$ then $\Lambda_\psi \prec \Lambda_m$ and we are done so suppose that $\A_\na(\Lambda_m \restrict \F) = \A_\na(\Lambda_{\psi} \restrict \F)$. In this case every non-peripheral element of $[\pi_1(S_m)]$ is weakly attracted to $\Lambda_\psi$ so $\xi_m(\psi)$ is a pseudo-Anosov mapping class and to complete the proof of \pref{item:carried by a surface} it remains to show that $\psi$ is 
irreducible relative to $F_{i_m-1} \sqsubset F_{i_m}$. 

A conjugacy class $[a]$ carried by $\F_{i_m}$ but not by $\A_\na(\Lambda_m)$ is represented by a closed curve $\gamma$ that is either contained in $S_{m}$ or has an efficient representation in $Y_m$ with at least one term $\mu_j$ that is a surface vertex path. By Nielsen--Thurston theory (specifically \SubgroupsOne\ Proposition~2.14), either $\gamma$ is homotopic to a power of $\partial_0 S_m$ or $\psi^i([a])$ weakly converges to~$\Lambda_\psi$. This proves that $\A_\na(\Lambda_{\psi} \restrict \F_{i_m}) = \F(G_{i_m-1}) \cup \< \partial_0S_m\>$. 

 If $\F' \sqsubset \F_{i_m}$ is a $\psi$-invariant free factor system that properly contains $\F_{i_m-1}$ then $\F'$ contains a conjugacy that is not carried by $\A_\na(\Lambda_{\psi})$ and so $\F'$ also carries $\Lambda_\psi$. Lemma~7.0.7 and Corollary~7.0.8 of \BookOne\ therefore imply that $\F'$ carries $[\pi_1(S_m)]$ and hence also $\Lambda_m$. Since $\phi$ is irreducible rel $\F_{i_m-1} \sqsubset 
\F_{i_m}$, it must be that $\F' = \F_{i_m}$ and so $\psi$ is irreducible rel $\F_{i_m-1} \sqsubset 
\F_{i_m}$. This completes the proof of~\pref{item:nestinglams} and hence of~\pref{item:carried by a surface}.

\medskip

The first step in proving \pref{item:despinning} is to show that if $m \ne m'$ then $[\pi_1 S_{m'}]$ is carried by $\A_\na(\Lambda_{m}) $. 
 This is obvious for $m' < m$ because $[\pi_1 S_{m'}]$ is carried by $\F_{i_{m'}}\sqsubset \F_{i_{m-1}} \sqsubset \A_\na(\Lambda_{m})$. On the other hand, if $m' > m$ then it follows from the geometric model for $\phi_{m'}$ that $\Lambda_m$ is not carried by $[\pi_1 S_{m'}]$. It then follows from the $\phi_m$-invariance of $[\pi_1 S_{m'}]$ that no conjugacy class in $[\pi_1 S_{m'}]$ is weakly attracted to $\Lambda_m$ and so each conjugacy class in $[\pi_1 S_{m'}]$ is carried by $\A_\na(\Lambda_{m}) $ as desired.

 For each $m \in \{1,\ldots, M\}$ and each $\nu \in \mcg(S_{m }, \partial S_{m })$ let $g({m },\nu) :X_{m } \to X_{m }$ be a homotopy equivalence that restricts to the identity on $L_{m }$ and to a homeomorphism that represents $\nu$ on $S_{m }$; let $\theta({m},\nu)$ be the element of $\Out(F_n)$ determined by $g({m},\nu)$. Lemma~\ref{LemmaFIX}\pref{ItemGammaEtaANA} and Fact~\ref{Fact:PropertiesOfL} imply that $\A_\na(\Lambda_{m})$ is realized by $L_m$ and carries $\gamma_\eta$. It follows that $\theta({m},\nu)$ fixes each leaf of $\Lambda^+_\eta$ and fixes each conjugacy class carried by $[\pi_1 S_{m'}]$ for $m\ne m'$. The former implies that $\theta({m},\nu) \in \Ker(PF)$ and the latter implies that $\theta({m},\nu)$ preserves $[\pi_1 S_{m'}]$ and induces the identity element of $\mcg(S_{m'})$ for each $m' \ne m$. By construction $\theta({m},\nu)$ preserves $[\pi_1 S_{m}]$ and induces the element of $\mcg(S_{m})$ that is the image of $\nu$ under the despinning homomorphism. 
 
 \medskip
 
 Define $\Theta \from \prod_1^M \mcg(S_m,\bdy S_m) \to \Ker(PF) \cap \bigcap_1^M \Stab[\pi_1 S_m]
$ by 
$$\Theta(\nu_1,\ldots,\nu_M) = \theta(1,\nu_1)\composed \theta(2,\nu_2)\composed \ldots \composed \theta(M,\nu_M)$$
The restriction of $\Theta(\nu_1,\ldots,\nu_M) $ to each $[\pi_1(S_m)]$ is a well defined element of $\mcg(S_m)$ so we have an induced homomorphism $$\prod_1^M \mcg(S_m,\bdy S_m) \to \prod_1^M \mcg(S_m)$$ which by construction satisfies
$$\Theta(\nu_1,\ldots,\nu_M) \mapsto (\nu_1',\ldots, \nu_M')$$
 where $\nu'_m$ is the image of $\nu_m$ under the despinning homomorphism. This completes the proof of \pref{item:despinning} and also the proof of lemma.
\end{proof}

\subsection{Proof of Theorem~\ref{ThmFinitelyGenerated}, and some corollaries.}
\label{SectionKernelProof}
Enumerate $K = \{\psi_a \suchthat a=1,2,\ldots\}$ and define the nested sequence of finitely generated subgroups $A^1 \subgroup A^2 \subgroup A^3 \subgroup \cdots$ where $A^a = \<\psi_1,\ldots,\psi_a\>$. Note that $K = \union A^a$ is finitely generated if and only if $A^a=K$ for some $a$. 

Consider one of the finitely generated subgroups $A = A^a$. Choose $\emptyset = \F_0 \sqsubset \F_1 \sqsubset \cdots \sqsubset \F_I = \{[F_n]\}$ to be a maximal $A$-invariant filtration by free factor systems, one of which is $\F_\agen$. Let $1 \le i_1 < \cdots < i_M \le I$ be the indices $i$ for which $\F_{i-1} \sqsubset \F_i$ is a multi-edge extension. By Theorem D of \SubgroupsZero\ for each $m=1,\ldots,M$ there exist $\phi_m \in A$ such that $\phi_m$ is irreducible rel $\F_{i_m-1} \sqsubset 
\F_{i_m}$; let $\Lambda_{\phi_m}$ be the corresponding geometric (by Lemma~\ref{LemmaFIX}) element of $\L(\phi_m)$, and let $\pi_1(S_m) \inject F_n$ be the associated surface system. Applying Lemma~\ref{LemmaMultiedge} to $\cH = A$ using $\phi_1,\ldots,\phi_M \in A$, for each $m=1,\ldots,M$ let $\xi_m : A \to \mcg(S_m)$ be the homomorphism given by Lemma~\ref{LemmaMultiedge}~\pref{item:maps to mcg}. Define $\xi = (\xi_{1},\ldots, \xi_{M}) \from A \to \mcg(S_{1}) \times \cdots \times \mcg(S_{M})$. Item Lemma~\ref{LemmaMultiedge}~\pref{item:kernel of xi} implies that $\Ker(\xi)$ is \upg, and Proposition~\ref{PropUPG} then implies that $\Ker(\xi)$ is linear, abelian and finitely generated. 

For $m=1,\ldots,M$ let $p_m$ be the maximal length of a chain $\Lambda_{\theta_1} \prec \Lambda_{\theta_2} \prec \ldots \prec \Lambda_{\phi_m}$ in the partial order of (Notation~\ref{PartialOrder}) that ends in $\Lambda_{\phi_m}$, where $\Lambda_{\theta_1}, \Lambda_{\theta_2}, \ldots \in \L(F_n)$. By Lemma~\ref{LemmaUniformBound} we have $p_m \le P$ for some fixed $P$. Define the complexity $cx=cx^a$ to be the sequence of $p_m$'s rewritten in non-decreasing order. Since $M \le I \le 2n-1$, the complexity $cx$ is an element of a finite set depending only on~$n$ and $P$, namely the nondecreasing sequences of length~$\le 2n-1$ with entries in $\{1,2,\ldots,P\}$. Order this set lexicographically.

Consider the next finitely generated subgroup $A' = A^{a+1} = \<\psi_1,\ldots,\psi_a,\psi_{a+1}\>$. Choose $\emptyset = \F'_0 \sqsubset \F'_1 \sqsubset \cdots \sqsubset \F'_J = \{[F_n]\}$ to be a maximal $A'$-invariant filtration by free factor systems, one of which is $\F_\agen$, and let $ 1 \le j_1 < \cdots < j_{L} \le J$ be the indices $j$ for which $\F'_{j-1} \sqsubset \F'_j$ is a multi-edge extension. Applying Theorem D of \SubgroupsZero\ as above, for $l = 1,\ldots,L$ choose $\phi'_l \in A'$ such that $\phi'_l$ is irreducible rel $\F'_{j_l-1} \sqsubset \F'_{j_l}$, and with the following additional property: if there exists $m \in \{1,\ldots, M\}$ such that $\phi_m$ is irreducible rel $\F'_{j_l-1} \sqsubset \F'_{j_l}$ and such that $\Lambda_{\phi_m}$ is the element of $\L(\phi_m) $ corresponding to the extension $\F'_{j_l-1} \sqsubset \F'_{j_l}$ then $\phi'_l = \phi_m$. Define~$\Lambda_{\phi'_l}$, 
$\xi'_l$, $\xi'$, $p'_l$, and $cx'=cx^{a+1}$ as above.

What happens in effect is that this procedure ends after a finite number of iterations, although this conclusion only comes after the fact, once we have proved by other means that $K$ is finitely generated.

%
%
%
Lemma~\ref{LemmaMultiedge}~\pref{item:carried by a surface} (together with choice of $\phi'_l$ above) implies that each time the construction is repeated as above, for each $m \in \{1,\ldots, M\}$ there exists $l \in \{1,\ldots, L\}$ such that either $\Lambda_{\phi_m} = \Lambda_{\phi'_l}$ or $\Lambda_{\phi_m} \prec \Lambda_{\phi'_l}$; 
in the former case $p_m = p'_l$ and in the latter case $p_m < p'_l$. It follows that $cx$ is less than or equal to $cx'$ in lexicographical order, with equality holding only if $M = L$ and for all $1 \le m \le M$ we have $(\Lambda_{\phi_m},\phi_m) = (\Lambda_{\phi'_m},\phi'_m)$. Since the set of complexities is finite, the complexity sequence $cx^1, cx^2, cx^3, \ldots$ is eventually constant. It follows that the subset of $\L(F_n)$ given by $\{(\Lambda_{\phi_m},\phi_m) \suchthat 1 \le m \le M\}$ is eventually constant, and so by Lemma~\ref{LemmaFIX}\pref{ItemGammaEtaANA} and Lemma~\ref{FactSurfaceUnique} it also follows that the surface systems $\mu_m \from \pi_1 S_m \to F_n$ associated to the $(\Lambda_{\phi_m},\phi_m)$ are eventually constant, and hence the group $\mcg(S_1) \times \cdots \times \mcg(S_M)$ is eventually constant. The (eventually defined) sequence of homomorphisms to this group from the groups $A^1 \subgroup A^2 \subgroup \cdots$ are (eventually) consistent with the inclusions, and hence these homomorphisms fit together to define a homomorphism $ K \mapsto \mcg(S_{1}) \times \cdots \times \mcg(S_{M})$ whose kernel is \upg\ and hence, by Proposition~\ref{PropUPG}, is finitely generated, linear and abelian. By Lemma~\ref{LemmaMultiedge}~\pref{item:despinning}, the image of this homomorphism has finite index, completing the proof of Theorem~\ref{ThmFinitelyGenerated}. \qed

\begin{corollary} \label{CorollaryOneEdge} If $\F$ is a proper $K$-invariant free factor system and $K$ is fully irreducible relative to the extension $\F \sqsubset \{[F_n]\}$ then $\F \sqsubset \{[F_n]\}$ is a one edge extension.
\end{corollary}

\begin{proof} If $\F_{k-1} \sqsubset \F_k = \{[F_n]\}$ is a multi-edge extension then, since $K$ is finitely generated, we may apply Theorem D of \SubgroupsZero\ to produce an element $\theta \in K$ which is fully irreducible relative to the extension $\F_{k-1} \sqsubset \F_k$. Choosing a \ct\ representing a rotationless power of $\theta$ in which $\F_\agen$ and $\F_{k-1}$ are represented by filtration elements, the highest stratum of that \ct\ must be \eg, contradicting Lemma~\ref{LemmaFIX}. It follows that $\F_{k-1} \sqsubset \F_k = \{[F_n]\}$ is a one-edge extension.
\end{proof}


 
\begin{corollary}\label{CorGeomModelFeatures} Suppose that $\phi \in K$, that $\Lambda^+_\phi \in \L(\phi)$ and that $[\pi_1(S)]$ is the surface subgroup associated to $\theta$ and $\Lambda^+_\phi$. Then each generic leaf $\gamma_\phi$ of $\Lambda^+_\phi$ is carried by $\A_\na(\Lambda^+_\eta)$ and $[\pi_1(S)] \sqsubset [\A_\na(\Lambda^+_\eta)]$. 
 \end{corollary}
 
 \begin{proof} Choose a maximal $K$-invariant filtration $\emptyset = \F_0 \sqsubset \F_1 \sqsubset \cdots \sqsubset \F_k = \{[F_n]\}$ by free factor systems, one of which is $\F_\agen$. Corollary~\ref{CorollaryOneEdge} implies that $\F_{k-1} \sqsubset \F_k = \{[F_n]\}$ is a one-edge extension. Each $\theta \in K$ has a rotationless iterate that is represented by a \ct\ preserving $\emptyset = \F_0 \sqsubset \F_1 \sqsubset \cdots \sqsubset \F_k = \{[F_n]\}$. It follows that each generic leaf $\gamma_\theta$ of each element of $\L(\theta)$ is carried by $\F_{k-1}$. Applying this to $\phi^{\eta^i} \in K$ for each $i \in \Z$ and noting that $\eta^i(\gamma_\phi)$ is a generic leaf of $\phi^{\eta^i}$, we see that the $\eta$ orbit of $\gamma^+_\phi $ is carried by $\F_{k-1}$. Since $\F_{k-1}$ does not contain any generic leaf of $\Lambda^+_\eta$, it follows that $\gamma^+_\phi$ is not weakly attracted to $\Lambda^+_\eta$. For the same reason, $\gamma_\phi$ is not weakly attracted to $\Lambda^-_\eta$. \SubgroupsOne\ Theorem H therefore implies that $\gamma_\phi$ is carried by $\A_\na(\Lambda^+_\eta)$. 

 By \SubgroupsOne\ Proposition~2.15 (4), the free factor supports $\F_{supp}(\gamma_\phi)$ and $\F_{supp}([\pi_1(S)])$ are equal. We can therefore complete the proof by showing that $\F_{supp}(\gamma_\phi)\sqsubset \A_\na(\Lambda^+_\eta)$. This is evident if $\A_\na(\Lambda^+_\eta)$ is a free factor system or equivalently if $\Lambda^+_\phi$ is non-geometric. If $\Lambda^+_\phi$ is geometric then by \SubgroupsThree\ \lq Remark: The case of the top stratum\rq\ there is a rank one component $[\langle c \rangle]$ of $\A_\na(\Lambda^+_\eta)$ such that $\A_\na(\Lambda^+_\eta) - [\langle c \rangle]$ is a free factor system. Since $\gamma_\phi$ is not carried by $[\langle c \rangle]$, it is carried by $\A_\na(\Lambda^+_\eta) - [\langle c \rangle]$ and hence
 $$ \F_{supp}(\gamma_\phi)\sqsubset \A_\na(\Lambda^+_\eta) - [\langle c \rangle] \sqsubset \A_\na(\Lambda^+_\eta) $$
 \end{proof}

\section{Completion of the proof of Theorem~\ref{disjoint axes}} 
\label{SectionProofCompletion}


First we prove conclusions~\pref{ItemCoAxial} and~\pref{ItemIndependent} of Theorem~\ref{disjoint axes} which together say that if $\phi$, $\psi$ act loxodromically on $\FS(F_n)$ and have filling lamination pairs $\Lambda^\pm_\phi$, $\Lambda^\pm_\psi$, then either $\{\Lambda^\pm_\phi\} = \{\Lambda^\pm_\psi\}$ and $\{\bdy_\pm \phi\} = \{\bdy_\pm\psi\}$, or $\{\Lambda^\pm_\phi\} \intersect \{\Lambda^\pm_\psi\} = \emptyset$ and $\{\bdy_\pm \phi\} \intersect \{\bdy_\pm\psi\} = \emptyset$. 

If $\Lambda^+_\phi \ne \Lambda^+_\psi$ then by Corollary~\ref{distinct ends} it follows that $\bdy_- \phi \ne \bdy_- \psi$. By inverting $\phi$ and/or $\psi$ as needed it then follows that if $\{\Lambda^\pm_\phi\} \intersect \{\Lambda^\pm_\psi\} = \emptyset$ then $\{\bdy_\pm\phi\} \intersect \{\bdy_\pm\psi\} = \emptyset$.

It remains to show that if $\Lambda^+_\phi = \Lambda^+_\psi \equiv \Lambda^+$ then $\bdy_- \phi = \bdy_- \psi$ and $\bdy_+ \phi = \bdy_+ \psi$, for by applying Corollary~\ref{distinct ends} to $\phi^\inv$ and $\psi^\inv$ it then follows that $\Lambda^-_\phi = \Lambda^-_\psi$. By \BookOne\ Corollary 3.3.1 the expansion factor homomorphism has infinite cyclic image, and so after replacing $\phi$ and $\psi$ by iterates we may assume that $\PF_{\Lambda^+}(\phi) = \PF_{\Lambda^+} (\psi)$. The kernel $K$ of $\PF_{\Lambda^+} \from \Stab(\Lambda^+) \cap \IAThree \to \R$ is finitely generated by Theorem~\ref{ThmFinitelyGenerated}. 
Let $\F_\agen=\F_\agen(\Lambda^+)$ be the free factor support of the ageneric sublamination of $\Lambda^+$, as defined in Section~\ref{SectionBackground}, and note that from the definition we evidently have $K \subgroup \Stab(\Lambda^+) \subgroup \Stab(\F_\agen(\Lambda^+))$. Choose a maximal $K$-invariant filtration $\emptyset = \F_0 \sqsubset \F_1 \sqsubset \cdots \sqsubset \F_k = \{[F_n]\}$ by free factor systems, one of which is $\F_\agen$, and so $\F_\agen \sqsubset \F_{k-1}$.


Corollary~\ref{CorollaryOneEdge} implies that $\F_{k-1} \sqsubset \F_k = \{[F_n]\}$ is a one-edge extension.
Choose a marked graph $G$ with a subgraph $H$ such that $[H] = \F_{k-1}$ and such that $G \setminus H$ is a single edge $E$. Since $[H]$ is $K$-invariant, each $\phi \in K$ is represented by a homotopy equivalence $h:(G,H) \to (G,H)$ such that $h(E) = \bar v E u$ for some paths $u,v \subset H$. Thus the one-edge free splitting $\<G,H\>$ is $K$-invariant. Since $K^\phi = K$, it follows that $\<G,H\>^{\phi^n}$ is $K$-invariant for each $n \in \Z$. Applying this to $\phi^{-n}\psi^n \in K$ we have $\<G,H\>^{\phi^n} = \<G,H\>^{\phi^n(\phi^{-n}\psi^n)} = \<G,H\>^{\psi^n}$. Thus the actions of $\phi$ and $\psi$ on $\fscn$ have a common orbit, and so in the Gromov boundary of $\FS(F_n)$ the forward limit point of this orbit is $\bdy_+ \phi = \bdy_+ \psi$ and its backward limit point is $\bdy_- \phi = \bdy_- \psi$.

\medskip

Next we prove conclusions~\pref{ItemLamRepCorrespond} and~\pref{ItemLamRepEquivariant} of Theorem~\ref{disjoint axes} which are concerned with the relation between laminations $\Lambda \subset \B$ and points $\beta \in \bdy \FS(F_n)$ defined by $\Lambda \leftrightarrow \beta$ if and only if there exists $\phi \in \Out(F_n)$ acting loxodromically on $\FS(F_n)$ with filling attracting lamination $\Lambda=\Lambda^+_\phi$ and with repelling point $\beta = \bdy_-\phi$. This relation satisfies the equivariance conclusion~\pref{ItemLamRepEquivariant} because if $\theta \in \Out(F_n)$ then $\theta^\inv \phi \theta$ acts loxodromically, and clearly we have $(\bdy_- \phi)^\theta = \bdy_-(\theta^\inv \phi \theta)$ and $\Lambda^+_{\theta^\inv \phi \theta} = \theta^\inv(\Lambda^+_\phi)$ which fills~$F_n$. Conclusion~\pref{ItemLamRepCorrespond} requires that this relation is a bijection between filling attracting laminations and loxodromic repelling points, and this is so because as we have seen just above, for any $\phi,\psi \in \Out(F_n)$ acting loxodromically we have $\Lambda^+_\phi = \Lambda^+_\psi$ if and only if $\bdy_-\phi = \bdy_-\psi$.

\bibliographystyle{amsalpha} 
\bibliography{mosher} 

\end{document}